\documentclass{sn-jnl}

\usepackage[caption=false]{subfig}% Support for small, `sub' figures and tables

\usepackage[natbibapa,nodoi]{apacite}

\usepackage{cleveref}             % Kluge Verweise 

\usepackage{multirow}

\usepackage{tabularx}
\usepackage{booktabs}
\usepackage{hhline}
\usepackage{enumitem}

\newcommand{\E}{{\mathbb{E}}}

\newcommand{\bqan}{\begin{eqnarray}}
\newcommand{\eqan}{\end{eqnarray}}

\newcommand{\vepsilon}{\vep}

\newcommand{\diag}{\operatorname{diag}}

\newcommand{\tr}{\operatorname{tr}}

\newcommand{\va}{\boldsymbol{a}}
\newcommand{\vb}{\boldsymbol{b}}

\newcommand{\ve}{\boldsymbol{e}}
\newcommand{\vf}{\boldsymbol{f}}
\newcommand{\vh}{\boldsymbol{h}}

\newcommand{\vr}{\boldsymbol{r}}
\newcommand{\vx}{\boldsymbol{x}}
\newcommand{\vv}{\boldsymbol{v}}

\newcommand{\vA}{\boldsymbol{A}}
\newcommand{\vC}{\boldsymbol{C}}
\newcommand{\vE}{\boldsymbol{E}}
\newcommand{\vH}{\boldsymbol{H}}
\newcommand{\vI}{\boldsymbol{I}}
\newcommand{\vJ}{\boldsymbol{J}}
\newcommand{\vL}{\boldsymbol{L}}
\newcommand{\vM}{\boldsymbol{M}}
\newcommand{\vN}{\boldsymbol{N}}
\newcommand{\vP}{\boldsymbol{P}}
\newcommand{\vR}{\boldsymbol{R}}
\newcommand{\vT}{\boldsymbol{T}}
\newcommand{\vU}{\boldsymbol{U}}
\newcommand{\vV}{\boldsymbol{V}}
\newcommand{\vX}{\boldsymbol{X}}
\newcommand{\vY}{\boldsymbol{Y}}
\newcommand{\vZ}{\boldsymbol{Z}}

\newcommand{\vGamma}{\boldsymbol{\Gamma}}
\newcommand{\vep}{\boldsymbol{\epsilon}}

\newcommand{\vmu}{\boldsymbol{\mu}}

\newcommand{\vSigma}{\boldsymbol{\Sigma}}

\newcommand{\vDelta}{\boldsymbol{\Delta}}
\newcommand{\vLambda}{\boldsymbol{\Lambda}}
\newcommand{\vzeta}{\boldsymbol{\zeta}}
\newcommand{\veins}{{\bf 1}}
\newcommand{\vnull}{{\bf 0}}
\newcommand{\vUpsilon}{\boldsymbol{\Upsilon}}
\newcommand{\To}{\longrightarrow}            % Abbildungspfeil
\newcommand{\ind}{1\hspace{-0.7ex}1}
\newcommand{\Cdot}{\cdot}
\DeclareMathOperator{\rank}{rank}
\DeclareMathOperator{\Cov}{Cov}
\DeclareMathOperator{\Corr}{Corr}
\DeclareMathOperator{\Var}{Var}
\DeclareMathOperator{\vech}{vech}
\newtheoremstyle{Test1}% name of the style to be used
  {2 \baselineskip}% measure of space to leave above the theorem. E.g.: 3pt
  {1.5 \baselineskip}% measure of space to leave below the theorem. E.g.: 3pt
  {\itshape}% name of font to use in the body of the theorem
  {-0.0ex}% measure of space to indent
  {\fontfamily{ppl}\fontseries{l}\fontshape{n}}% name of head font
  {:}% punctuation between head and body
  {\newline}% space after theorem head; " " = normal interword space
   {}% Manually specify head
\theoremstyle{Test1}
\newtheorem{Sa}{Theorem}[section]
\newtheorem{theorem}{Theorem}[section]
\newtheorem{re}[Sa]{Remark}
 
\newtheorem{Le}[Sa]{Lemma}
\newtheorem{Ko}[Sa]{Corollary}
\newcommand{\Lan}{\mathcal{O}}
\newcommand{\lan}{ \scriptstyle \mathcal{O}\textstyle}
\newcolumntype{x}[1]{!{\centering\arraybackslash\vrule width #1}}
\begin{document}

\title{Testing  Hypotheses about Correlation Matrices in General MANOVA Designs}
\author*[1]{\fnm{Paavo} \sur{Sattler}}\email{paavo.sattler@tu-dortmund.de}
\author[2]{\fnm{Markus} \sur{Pauly}}\email{markus.pauly@tu-dortmund.de}
\affil[1]{Institute for Mathematical Statistics and Industrial Applications, Department of Statistics, TU Dortmund University, Joseph-von-Fraunhofer-Straße 2-4, 44221 Dortmund, Germany}
\affil[2]{Institute for Mathematical Statistics and Industrial Applications, Department of Statistics, TU Dortmund University, Joseph-von-Fraunhofer-Straße 2-4, 44221 Dortmund, Germany; {Research Center Trustworthy Data Science and Security, UA Ruhr, Joseph-von-Fraunhofer-Straße 25}}

%\author{
%\name{Paavo Sattler\textsuperscript{a}\thanks{CONTACT Paavo Sattler. Email: paavo.sattler@tu-dortmund.de} and Markus Pauly\textsuperscript{a}}
%\affil{\textsuperscript{a}Institute for Mathematical Statistics and Industrial Applications, Faculty of Statistics, Technical
%University of Dortmund, Joseph-von-Fraunhofer-Straße 2-4, 44221 Dortmund, Germany}
%}
%

%\title{ Testing  Hypotheses about Correlation Matrices in General MANOVA Designs}
%\author[1]{Paavo Sattler\corref{cor1}}
%\ead{paavo.sattler@tu-dortmund.de}
%\address[1]{Institute for Mathematical Statistics and Industrial Applications, Faculty of Statistics, Technical
%University of Dortmund, Joseph-von-Fraunhofer-Straße 2-4, 44221 Dortmund, Germany}
%\author[1]{Markus Pauly}
%\cortext[cor1]{Corresponding author.}
%

\abstract{
Correlation matrices are an essential tool for investigating the dependency structures of random vectors or comparing them. We introduce an approach for testing a variety of null hypotheses that can be formulated based upon the correlation matrix. Examples cover 
MANOVA-type hypothesis of equal correlation matrices as well as testing 
for special correlation structures such as, e.g., sphericity. Apart from existing fourth moments, our approach requires  no other assumptions, allowing applications in various settings. To improve the
small sample performance, a bootstrap technique is proposed and theoretically justified. {Based on this, we also present a procedure to simultaneously test the hypotheses of equal correlation and equal covariance matrices. The performance of all new} test statistics is compared with existing procedures through extensive simulations.}


\maketitle
\section{\protect{Motivation and Introduction}}\label{int}
%\noindent 
Covariance matrices contain %{test}
a multitude of information about a random vector. Therefore, they were the topic of manifold investigations.
For testing, an important hypothesis is the equality of covariance matrices from different groups. This was, e.g., investigated in \cite{bartlett1953} as well as \cite{boos2004}. Moreover, \cite{gupta2006} proposed tests for a given covariance matrix. 
Extending on both, \cite{sattler2022} proposed a unifying framework that allows for investigations of various hypotheses about the covariance. Additional examples, e.g., cover 
testing equality of traces or comparing diagonal elements of the covariance matrix.
%\noindent
However, covariance matrices are not scale-invariant. This entails some disadvantages using them in the analysis of random vectors' dependency structure. For example, a simple change of a measuring unit can completely change the matrix. For this reason alone, it is more useful to consider the correlation matrix instead when inferring dependency structures.\\
This already starts in the bivariate case, i.e. the investigation of correlations.
{For ordinal data, rank-based correlation measures are most common. For example, \cite{perreault2022} provides an approach for general hypotheses testing %with Kendall's $\tau$ in a non-parametric setting 
while \cite{nowak2021} focused on simultaneous confidence intervals and multiple contrast tests for Kendall's $\tau$. Spearman's rank-correlation coefficient  $\rho$ was investigated in \cite{gaißer2010} for the hypothesis of the correlation matrix being an equicorrelation matrix.
For metric data, the most common measure of correlation is the Pearson correlation coefficient.%, also it can not be used for data which are only ordinal.
} Here, 
tests and confidence intervals for investigating or comparing correlation coefficients have been discussed for the case of one, two or multiple group(s), see (\cite{fisher1921}; \cite{efron1988}, \cite{sakaori2002}, \cite{gupta2006},  \cite{tian2008}, \cite{omPa:2012}, \cite{welz2022}) and the references cited therein.
%\noindent
For larger dimensions, equality of correlation matrices was investigated, for example, in \cite{jennrich1970}. Moreover, different hypotheses regarding the structure of correlation matrices were treated in   \cite{joereskog}, \cite{steiger1980} and \cite{wu2018}.
However, the above approaches usually require strong prerequisites on the distribution (such as multivariate normal distribution or particular properties of the moments), the components or the setting (such as bivariate or special structures). Moreover, most can be used only for a few specific hypotheses.
Thus, to obtain an approach with fewer assumptions, which is at the same time applicable for a multitude of hypotheses, we expand the approach of \cite{sattler2022} to the treatment of correlation matrices.\\
%\noindent
In the following section, the statistical model will be introduced together with examples of different null hypotheses that can be investigated using the proposed approach. Afterwards, the asymptotic distributions of the proposed test statistics are derived (Section~\ref{Asymptotics regarding the vectorized correlation}). In Section~\ref{Resampling Procedures} and \ref{Taylor}, a resampling strategy {and a Taylor-based Monte-Carlo approach are} used to generate critical values and improve our tests' small sample behaviour. { A combined testing procedure 
which simultaneously checks the hypothesis of equal correlation matrices and covariance matrices is presented in \Cref{combined}.} The simulation results regarding type-I-error control and power are discussed in Section~\ref{Simulations}, while an illustrative data analysis of EEG-data is conducted in Section~\ref{Illustrative Data Analysis}. All proofs are deferred to a technical supplement, where also more simulations can be found.
%\noindent
\section{Statistical Model and Hypotheses} \label{mod}
To allow {for broad applicability}, we use a general semiparametric model, given by independent $d$-dimensional random vectors 
\[\vX_{ik}= \vmu_i + \vep_{ik},\]
where the index $i=1, \dots, a$ refers  to the treatment group and $k=1, \dots,n_i$ to the individual, on which $d$-variate observations are measured. 
As analyzing a scalar's correlation is useless,  we assume $d\geq 2$. {This model as well as all subsequent results can in principle be expanded for different dimensions of the groups as done in \cite{friedrich2017permuting}. However, as the notation would be a bit cumbersome, we did not follow this approach to increase the readability.}
The residuals $\vep_{i1},\dots,\vep_{in_i}$ are assumed to be centred $\E(\vep_{i1}) = \vnull_d$ and i.i.d. within each group, 
with finite fourth moment $\E(||\vep_{i1}||^4) < \infty$, {while  across groups they are only independent}. Thus, different distributions per group are possible. For our asymptotic considerations, we additionally assume
\setitemize{leftmargin=0.9cm}
\begin{itemize}
\vspace*{0.2cm}

  \item[(A1)] $\frac{n_i}{N}\to \kappa_i\in (0,1],~i= 1,...,a$ for $\min(n_1,\dots,n_a)\to\infty$ \vspace*{0.2cm}
\end{itemize}
with $N = \sum_{i=1}^a n_i$, to preclude that a single group dominates the setting. {In the following we use $\stackrel{\mathcal{P}}{\To}$ to  denote convergence in probability and $\stackrel{\mathcal{D}}{\To}$ for convergence in distribution as sample sizes increase. It follows from the context whether this means $N\to \infty$ or $n_i\to \infty$, respectively.}\\
%The only difference to the conditions for testing hypotheses regarding the covariance is that the $\vV_i$ here has to be strictly positive definite and not just semidefinite. 
%%So we require $\Cov(\vep_{i1})=\vV_i > 0,~ i=1, \dots, a$.
{Moreover, to define correlation matrices we additionaly assume that the covariance matrices $\Cov(\vep_{i1})=\vV_i=(V_{ijk})_{j,k\leq d}$ have positive diagonals $V_{i11},...,V_{idd} > 0$. %, while semidefinite was sufficient in \cite{sattler2022} for testing hypotheses regarding covariance matrices.
}
%{The only difference to the conditions from \cite{sattler2022} for testing hypotheses regarding covariance matrices, is that the $\vV_i$ here has to be strictly positive definite and not just semidefinite.} 
{Throughout}, the so-called half-vectorization operation $\vech$ is used for the covariance {matrices}, which {puts} the upper triangular elements of a $d\times d$ matrix into a $p=d(d+1)/2 $ dimensional vector. But, for a correlation matrix, the diagonal elements are always one and therefore contain no information, so this is not the best choice here.
Hence, a new vectorization operation $\vech^-$ is defined, which we will call the upper-half-vectorization. With $\vR_i$ as the correlation matrix for the i-th group,  this vectorization operation allows us to define \[\vr_i=\vech^-(\vR_i)= (r_{i12},\dots,r_{i1d},r_{i23},\dots,r_{i2d},\dots,r_{i(d-1)d})^\top, \  i=1,\dots,a,\] containing just the upper triangular entries of  $\vR_i$ which are not on the diagonal. The resulting vector has dimension $p_u=d(d-1)/2$, which is substantially smaller than $p$.
We now formulate hypotheses in terms of the pooled correlation vector $\vr= (\vr_1^\top,\dots,\vr_a^\top)^\top$ as 
\bqan\label{eq:hypo corr}
 \mathcal  H_0^{\vr}: \vC \vr = \vzeta,
\eqan
with a proper hypothesis matrix $\vC\in\mathbb{R}^{m\times a p_{u}}$ and a vector $\vzeta\in \mathbb{R}^{m}$. Hereby, we allow $m$ to be much smaller than $ap_u$, which is often useful (e.g. for computational reasons) for hypotheses matrices without full rank. {As explained in \cite{sattler2022} for the case of covariance matrices, $\vC$ does not have to be a projection matrix as $\vzeta$ is allowed to be different from the zero vector, see, for example, hypothesis (c) below}.\\

Hypotheses which are part of the setting \eqref{eq:hypo corr} are, among others:\\

(a)  {\bf Testing Homogeneity of correlation matrices:}
 $$
\mathcal H_0^{\vr}: \vR_1 = \dots = \vR_a, \text{ resp. } {\mathcal H_0^{\vr}:(\vP_a\otimes \vI_{p_u}) \vr=\vnull_{ap_u}},
 $$
{with $\vP_a:=\veins_a \veins_a^\top-\vI_a/a$ and $\otimes$ denoting the Kronecker product. This hypothesis  was, for example, investigated in \cite{jennrich1970}, while }for $d=2$, this includes the problem of testing the null hypothesis 
  $$
 \mathcal  H_0^{\vr}: \rho_{1} = \rho_2 = \dots = \rho_a
 $$
 of equal correlations $\rho_i = \Corr(X_{i11},X_{i12}), i=1,\dots,a$ within \eqref{eq:hypo corr}. This again contains testing equality of correlations between two groups, see, e.g. \cite{sakaori2002}, \cite{gupta2006}  or \cite{omPa:2012}.\\\\
(b){ \bf Testing a diagonal correlation matrix:}
$$ \mathcal  H_0^{\vr}: \vR_1 = \vI_d, \text{ resp. } \mathcal H_0^{\vr}:{ \vI_{p_u}\vr_1=\vnull_{p_u}}.$$
A test procedure for zero correlations was introduced by \cite{bartlett1951} for the one-group setting. More hypotheses on the structure of the covariance matrix can be found, for example, in  \cite{joereskog}, \cite{steiger1980} and \cite{wu2018}. \\
 
 (c)  {\bf Testing for a given correlation:} Let $\vR$ be a given correlation matrix, like, e.g., an autoregressive or compound symmetry matrix. For $a=1$, we then also cover testing the null  hypothesis 
 $$
 \mathcal  H_0^{\vr}: \vR_1 = \vR\text{ resp. }{\mathcal  H_0^{\vr}: \vI_{p_u}\vr_1=\vech^-(\vR) %\text{ for a given matrix } \vR
 }.
 $$ For $d=2$, this  also contains the issue of testing the null  hypothesis 
  $$
 \mathcal  H_0^{\vr}: \rho_{1} = 0
 $$
 of uncorrelated random variables with  $\rho_1 = \Corr(X_{111},X_{112})$, see  e.g. \cite{aitkin1968}. \\\\
 (d)  {\bf Testing for equal correlations:} For $a=1$, we are interested whether the correlation between all components is the same, i.e. 
 $$
\mathcal H_0^{\vr}: \vr_{11}=\vr_{12}=...=\vr_{1p_u}{\text{ resp. } \mathcal H_0^{\vr}: \vP_{p_u}\vr_1=\vnull_{p_u}.}$$
This kind of hypothesis is connected with {a} compound symmetry matrix and was, e.g. investigated in \cite{wilks1946} and \cite{box1950} for special settings.

\section{Asymptotics regarding the vectorized correlation}\label{Asymptotics regarding the vectorized correlation}
To infer null hypotheses of the kind $ \mathcal H_0^{\vr}: \vC \vr = \vzeta,$ it is necessary first to investigate the asymptotic distribution of $\vC \widehat \vr$, while $\widehat \vr$ is the pooled vector of upper-half-vectorized empirical correlation matrices. { Thereto Theorem 3.1 from \cite{sattler2022} is shortly repeated first.
Therein, for $\vv=(\vv_1^\top,...,\vv_a^\top)^\top=(\vech(\vV_1)^\top,...,\vech(\vV_a)^\top)^\top$ and  with empirical covariance matrices $\widehat \vV_i$ and $\widehat{\vv}=(\widehat\vv_1^\top,...,\widehat\vv_a^\top)^\top=(\vech(\widehat\vV_1)^\top,...,\vech(\widehat\vV_a)^\top)^\top$ it holds }
\begin{equation}\sqrt{N} \vC(\widehat \vv -\vv)\stackrel{\mathcal D}{\longrightarrow} \mathcal{N}_{ap}\left(\vnull_{ap},\vC\vSigma \vC^\top\right).\label{Konvergenz}\end{equation}
{Here, the covariance matrix is defined as $\vSigma=\bigoplus_{i=1}^a \frac{1}{k_i} \vSigma_i$,
where $\bigoplus$ denotes the direct sum and   $\vSigma_i=\Cov(\vech(\vepsilon_{i1}\vepsilon_{i1}^\top)^\top) $ for $i=1,\dots, a$.} First, some additional matrices have to be defined to use this result for correlation matrices. Let $\ve_{k,p}=(\delta_{k \ell})_{\ell=1}^p$ define the $p$-dimensional vector, which contains a one in the k-th component and zeros elsewhere. 
 
 Moreover, we need a $d$-dimensional auxiliary  vector $\va=(a_1,...,a_d)$, given through $a_k=1+\sum_{j=1}^{k-1} (d+1-j)$, $k=1,...,d$. It contains the position of components in the half-vectorized matrix, which are the diagonal elements of the original matrix.
 In accordance with this, we define the $p_u$-dimensional vector $\vb$, which contains the numbers from one to $p$ in ascending order without the elements from $\va$. This vector $\vb$   contains the position of components in the half-vectorized matrix, which are {non-diagonal} elements.\\
 With these vectors we are able to define a  $d\times d$ matrix $\vH=\boldsymbol{1}_{d}\va$ and the vectors $\vh_1=\vech^-(\vH)$ and $\vh_2=\vech^-(\vH^\top)$. Finally, we can define the matrices
\[
\vM_1=\sum_{\ell=1}^{p_u} \ve_{\ell,p_u}(\ve_{{\vh_1}_{\ell},p}+\ve_{{\vh_2}_\ell,p})^\top\quad \text{ and } \quad \vL_p^u=\sum_{\ell=1}^{p_u} \ve_{\ell,p_u}\ \ve_{{\vb}_{\ell},p}^\top.\]
This allows us to formulate a connection between the $\vech$ operator and the $\vech^-$ operator, since the matrix $\vL_p^u$ fulfils $\vL_p^u \vech(\vA)=\vech^-(\vA)$ for each arbitrary matrix $\vA\in \mathbb{R}^{d\times d}$. This matrix is comparable to the elimination matrix from \cite{magnus1980} and adapted to this special kind of half-vectorization.\\
With all these matrices, a connection can be found between $\sqrt{n_i}(\widehat \vv_i-\vv_i)$ and $\sqrt{n_i}(\widehat \vr_i-\vr_i)$, which 
allows {getting} the requested result by applying \eqref{Konvergenz}. The approach to connect vectorized correlation and vectorized covariance is based on \cite{browne} and \cite{nel1985}, and adapted to the current more general setting.
\begin{theorem}\label{CorTheorem1}
With the previously defined matrices $\vL_{p}^u$, $\vM_1$ and 
\[\vM(\vv_i,\vr_i):=\left[\vL_p^u-\frac 1 2\diag(\vr_i)\vM_1   \right] \diag(\vech((v_{i11},...,v_{i dd})^\top (v_{i11},...,v_{i dd})))^{-\frac 1 2},\]
it holds 
\[\sqrt{n_i}(\widehat \vr_i-\vr_i)=\vM(\vv_i,\vr_i)\sqrt{n_i} (\widehat \vv_i-\vv_i)+\lan_P(1).\] 
Thus
\[\sqrt{n_i}(\widehat \vr_i-\vr_i)\stackrel{\mathcal D}{\To}\mathcal{N}_{p_u}\large(\vnull_{p_u},\underbrace{\vM(\vv_i,\vr_i)\vSigma_i\vM(\vv_i,\vr_i)^\top}_{=:\vUpsilon_i}\large)\]
and { if Assumption (A1) is fulfilled also}
\[\sqrt{N}(\widehat \vr-\vr)\stackrel{\mathcal D}{\To}\mathcal{N}_{a p_u}\left(\vnull_{a p_u},\bigoplus\nolimits_{i=1}^a {\kappa_i}^{-1}\vUpsilon_i\right)=\mathcal{N}_{a p_u}(\vnull_{ap_u},\vUpsilon). \]
\end{theorem}

\begin{re}
{The asymptotic normality of $\sqrt{n_i}(\widehat \vr_i-\vr_i)$ was also shown for similar statistics in the past. But through the existing relation between  $\sqrt{n_i}(\widehat \vv_i-\vv_i)$ and $\sqrt{n_i}(\widehat \vr_i-\vr_i)$ it is, for example, possible to express $\vUpsilon$ as a function of $\vSigma$, which allows to construct an estimator $\widehat \vUpsilon$ using $\widehat \vSigma$. In Sections \ref{Taylor} and \ref{combined} we show that this relation provides further %interesting 
opportunities.}

\end{re}
To use this result, we have to estimate the matrices $\vUpsilon_1,...,\vUpsilon_a$, which is done with
\[\widehat \vUpsilon_i=\vM(\widehat\vv_i,\widehat\vr_i)\widehat \vSigma_i\vM(\widehat\vv_i,\widehat\vr_i)^\top,\]
and $\widehat \vUpsilon:=\bigoplus_{i=1}^a \frac{n_i}{N}\widehat \vUpsilon_i$. {To this end, $\widehat \vSigma_i$ has to be a consistent estimator for  $\vSigma_i$. This is fulfilled, e.g., for the estimator from \cite{sattler2022},  given through}
\[\frac{1}{n_i-1} \sum\limits_{k=1}^{n_i}\left[\vech\left(\widetilde \vX_{ik}\widetilde \vX_{ik}^\top- \sum\limits_{\ell=1}^{n_i}\frac{\widetilde \vX_{i\ell}\widetilde \vX_{i\ell}^\top}{n_i}\right)\right]\left[\vech\left(\widetilde \vX_{ik}\widetilde \vX_{ik}^\top- \sum\limits_{\ell=1}^{n_i}\frac{\widetilde \vX_{i\ell}\widetilde \vX_{i\ell}^\top}{n_i}\right)\right]^\top,\]
{with $\widetilde \vX_{ik}:=\vX_{ik}-\overline \vX_i$.
 As continuous functions of consistent estimators the estimators $\widehat \vUpsilon_i$ are consistent.}  With this asymptotic result, test statistics based on quadratic forms can be formulated  through: 
% 
% \begin{theorem}\label{Verteilung}  
%Let $\vE(\vC,\widehat \vUpsilon) \in \mathbb{R}^{m\times m}$ be some symmetric matrix which can be written as a function of the hypothesis matrix $\vC\in \mathbb{R}^{m \times ap_u}$ and the covariance matrix estimator $\widehat{\vUpsilon}\in \mathbb{R}^{ap_u\times ap_u}$. Additionally, assume that  $\vE(\vC,\widehat \vUpsilon)\stackrel{\mathcal{P}}{\To}\vE(\vC,\vUpsilon)$.
%Then, under the null  hypothesis $\mathcal{H}_0^r:\vC\vr=\vzeta$, the quadratic form $\widehat Q_{\vr} $ defined by
%\[\widehat Q_{\vr}=N\left[ \vC\widehat \vr - \vzeta\right]^\top \vE(\vC,\widehat \vUpsilon)\left[  \vC\widehat \vr - 
%\vzeta\right]\]
% has asymptotically a ``weighted $\chi^2$-distribution'', i.e.  for $N\to \infty$ it holds that
%\[\widehat Q_{\vr}\stackrel{\mathcal {D}}{\To} \sum\nolimits_{\ell=1}^{ap_u} \lambda_\ell B_\ell,
%\]
%{where $\lambda_\ell, \ell =1,\dots, a p_u,$ are the eigenvalues of $\vUpsilon^{1/2}\vC^\top\vE(\vC,\vUpsilon)\vC \vUpsilon^{1/2}$ and $B_\ell  \stackrel
%{i.i.d.} { \sim} \chi_1^2$.}
%\end{theorem}

 \begin{theorem}\label{Verteilung}  
{Let $\vE_N \in \mathbb{R}^{m\times m}$ fulfilling $\vE_N \stackrel{\mathcal{P}}{\To} \vE$, where $\vE\in \mathbb{R}^{m\times m}$ is symmetric. 
Then, under the null  hypothesis $\mathcal{H}_0^r:\vC\vr=\vzeta$, the quadratic form $\widehat Q_{\vr} $ defined by
\[\widehat Q_{\vr}=N\left[ \vC\widehat \vr - \vzeta\right]^\top \vE_N\left[  \vC\widehat \vr - 
\vzeta\right]\]
 has asymptotically a ``weighted $\chi^2$-distribution'', i.e.  for $N\to \infty$ it holds that
\[\widehat Q_{\vr}\stackrel{\mathcal {D}}{\To} \sum\nolimits_{\ell=1}^{ap_u} \lambda_\ell B_\ell,
\]}
{where $\lambda_\ell, \ell =1,\dots, a p_u,$ are the eigenvalues of $\vUpsilon^{1/2}\vC^\top\vE\vC \vUpsilon^{1/2}$ and $B_\ell  \stackrel
{i.i.d.} { \sim} \chi_1^2$.}
\end{theorem}

Thus quadratic-form type test statistics can be formulated, similar to \cite{sattler2022}  also for the vectorized correlation matrix. 
{Here we often consider symmetric matrices $\vE_N$, which can be written as function of $\vC$ and $\widehat \vUpsilon$ fullfilling  $\vE_N(\vC,\widehat \vUpsilon)=\vE_N\stackrel{\mathcal{P}}{\To}\vE=\vE(\vC,\vUpsilon)$.}
Examples cover: \[\begin{array}{lrl}
&ATS_{\vr}\hspace*{-0.1cm}&={N}\left[\vC(\widehat\vr-\vr)\right]^\top \left[\vC(\widehat\vr-\vr)\right]/\tr(\vC \widehat{\vUpsilon}\vC^\top),\\
&WTS_{\vr}\hspace*{-0.1cm}&={N}\left[\vC(\widehat\vr-\vr)\right]^\top (\vC \widehat{\vUpsilon}\vC^\top)^+\left[\vC(\widehat\vr-\vr)\right] 
.\end{array}\]
 This, e.g., leads to an asymptotic valid test {$\varphi_{WTS_{\vr}}=\ind\{ {WTS_{\vr}} \notin (-\infty,\chi^2_{\rank(\vC);1-\alpha}]\}$}, in case of $\vUpsilon>0$.
 {This is, however, hard to verify due to its complex structure. %  Through the structure of $\vSigma$ and $\vUpsilon$, this assumption is hard to justify, 
 We therefore do not treat the WTS in the following.}

  {Conversely the ATS is no asymptotic pivot but 
the simulation results from \cite{sattler2022} suggest to use its  Monte-Carlo version.
Hereto, the according matrices from \Cref{Verteilung} were estimated and by this we also get the estimated eigenvalues. By generating $B_1,...B_{ap_u}$ the weighted sum can be calculated, and by repeating this frequently (e.g. W=10,000 times), the required $\alpha$ quantiles of the weighted $\chi_1^2$ distribution can be estimated,  denoted by $q_{\alpha}^{MC}$. For example, this gives us the test  $\varphi_{ATS_{\vr}}:=\ind\{ ATS_{\vr} \notin (-\infty,q_{1-\alpha}^{MC}]\}$, and a similar approach can be used for all quadratic forms fulfilling \Cref{Verteilung}.}\\
%It would be more convenient to find a direct connection between  $(\widehat \vr -\vr)$ and $\vech^-(\widehat\vV-\vV)$, instead of having a result for the half vectorization and using this for developing results for the upper half-vectorization. Unfortunately, this has not been possible so far, although a result for $\sqrt{N}\vech^-(\widehat\vV-\vV)$  could be developed {in analogy to} Theorem 3.1 from \cite{sattler2022}. However, it has to be taken into account that we need the covariance matrix's diagonal elements to calculate the correlation matrix. Moreover, dependencies between components and the structure of $\widehat \vR -\vR$ make this task quite challenging. Therefore, this workaround has to be done with only minimal impact on the computation time.
\section{Resampling Procedures}\label{Resampling Procedures}
A resampling procedure may be useful, on the one hand, for a better small sample approximation and, on the other hand, for quadratic forms with critical values that are difficult to calculate. Since the simulations from \cite{sattler2022} showed clear advantages of the parametric bootstrap, we focus on this approach but also present a wild bootstrap approach in the supplement.
For every group  $\vX_{i 1},..., \vX_{i n_i}$, $i=1,...,a$, we calculate the covariance matrix $\widehat \vUpsilon_i$. With this covariance matrix we generate i.i.d. random vectors $\vY_{i 1}^\dagger,...,\vY_{i n_i}^\dagger$ ${\sim} \mathcal N_{p_u}(\vnull_{p_u},\widehat\vUpsilon_i)$ which are independent of the realizations and calculate their sample covariances $\widehat\vUpsilon_i^\dagger$  respectively $\widehat\vUpsilon^\dagger:=\bigoplus_{i=1}^a\frac{N}{n_i} \widehat \vUpsilon_i^\dagger$. For these random vectors, we now consider the asymptotic distribution.

\begin{theorem}\label{PBTheorem1}If  Assumption (A1) is fulfilled, it holds:
Given the data, the conditional  distribution of
\begin{itemize}

\item[(a)]{ $\sqrt{N}\ \overline \vY_i^\dagger$, for $i=1,...,a$, converges weakly to $ \mathcal{N}_{p_u}\left(\vnull_{p_u},{\kappa_i}^{-1}\vUpsilon_i\right)$ in probability}.% Since we have $\widehat \vUpsilon_i^\dagger \stackrel{\mathcal{P}}{\to} \vUpsilon_i$, the unknown covariance matrix $\vUpsilon_i$ can be estimated through $\widehat \vUpsilon_i^\dagger$.\\
\item[(b)]  { $\sqrt{N}\ \overline \vY^\dagger$ converges weakly to $ \mathcal{N}_{ap_u}\left(\vnull_{a p_u}, \vUpsilon\right)$ in probability.} %Since we have $\widehat \vUpsilon^\dagger \stackrel{\mathcal{P}}{\to}  \vUpsilon$, the unknown covariance matrix $\vUpsilon$ can be estimated through $\widehat \vUpsilon^\dagger.$
\end{itemize}
{Moreover, we have $\widehat \vUpsilon_i^\dagger \stackrel{\mathcal{P}}{\to} \vUpsilon_i$ and $\widehat \vUpsilon^\dagger \stackrel{\mathcal{P}}{\to}  \vUpsilon$. Thus, the unknown covariance matrices can be estimated through these estimators.}
\end{theorem}

%\begin{theorem}\label{PBTheorem1}If  Assumption (A1) is fulfilled, it holds:\\
%(a) For $i=1,...,a$, the conditional distribution of $\sqrt{N}\ \overline \vY_i^\dagger$, given the data, converges weakly to $ \mathcal{N}_{p_u}\left(\vnull_{p_u},{\kappa_i}^{-1}\cdot \vUpsilon_i\right)$ in probability. Since we have $\widehat \vUpsilon_i^\dagger \stackrel{\mathcal{P}}{\to} \vUpsilon_i$, the unknown covariance matrix $\vUpsilon_i$ can be estimated through $\widehat \vUpsilon_i^\dagger$.\\
%(b) The conditional distribution of $\sqrt{N}\ \overline \vY^\dagger$, given the data, converges weakly to $ \mathcal{N}_{ap_u}\left(\vnull_{a p_u}, \vUpsilon\right)$ in probability. Since we have $\widehat \vUpsilon^\dagger \stackrel{\mathcal{P}}{\to}  \vUpsilon$, the unknown covariance matrix $\vUpsilon$ can be estimated through $\widehat \vUpsilon^\dagger.$
%\end{theorem}
In consequence of \Cref{PBTheorem1}, it is reasonable to calculate the bootstrap version of the previous quadratic forms, 
\[\begin{array}{cl} Q_{\vr}^\dagger &=N[ \vC\overline \vY^\dagger  ]^\top {\vE}[ \vC\ \overline \vY^\dagger  ]\quad \text{ resp. }\quad Q_{\vr}^\dagger =N[ \vC\overline \vY^\dagger  ]^\top \vE_N(\vC,\widehat \vUpsilon^\dagger)[ \vC\ \overline \vY^\dagger]. \\
\end{array}\]
{The first can be used if the limiting matrix $\vE=\vE(\vC,\vUpsilon)$ of $\vE_N(\vC,\widehat \vUpsilon)$ is known. In the second version $\vE=\vE(\vC,\vUpsilon)$ is estimated via 
$\vE_N=\vE_N(\vC,\widehat \vUpsilon^\dagger)$ which is reasonable if from $\vA_n\stackrel{\mathcal{P}}{\to}\vA$, for random matrices $\vA_n$ and $\vA$
it follows $\vE_N(\vC,\vA_n) \stackrel{\mathcal{P}}{\to}  \vE(\vC,\vA)$ (see e.g. \cite{vanderVaart1996}). This is fulfilled, if the function 
is continuous in its second component which, e.g., holds for the trace operator used in the ATS and which we assume in what follows.}
%Similar to \cite{sattler2022}, two important quadratic forms are given by
%\[\begin{array}{ll}ATS_{\vr}^\dagger\hspace*{-0.1cm}&=N[ \vC \overline \vY^\dagger ]^\top[ \vC\overline \vY^\dagger  ]/\tr(\vC \widehat \vUpsilon^\dagger \vC^\top),\\
%WTS_{\vr}^\dagger\hspace*{-0.1cm}&=N[ \vC\overline \vY^\dagger ]^\top(\vC\widehat \vUpsilon^\dagger \vC^\top)^+[ \vC \overline \vY^\dagger ].\end{array}\]
{The bootstrap versions always approximate the null distribution of $\widehat Q_{\vr}$, as established below.}
\begin{Ko}\label{KorParametric}
For each parameter vector $\vr\in \mathbb{R}^{a p_u}$ and $\vr_0\in \mathbb{R}^{a p_u}$ with $\vC\vr_0=\vzeta$,  under Assumption (A1) we have
\[\begin{array}{l}
\sup\limits_{x\in\mathbb{R}}\big\lvert P_{\vr}(Q_{\vr}^\dagger\leq x\lvert \vX)-P_{\vr_0}(\widehat Q_{\vr}\leq x)\big\lvert \stackrel{\mathcal P}{\To}0,\\[1.5ex]
%\sup\limits_{x\in\mathbb{R}}\big\lvert P_{\vr}(WTS_{\vv}^\dagger\leq x\lvert \vX)-P_{\vv_0}(WTS_{\vv}(\widehat\vSigma)\leq x)\big\lvert \stackrel{\mathcal P}{\To}0,\\[1.5ex]
%\sup\limits_{x\in\mathbb{R}}\big\lvert P_{\vr}\left(MATS_{\vv}^\dagger\leq x\lvert \vX\right)-P_{\vv_0}(MATS_{\vv}(\widehat\vSigma_0)\leq x)\big\lvert \stackrel{\mathcalP}{\To}0,\\
\end{array}\] where $P_{\vr}$ denotes the (un)conditional distribution of the test statistic when $\vr$ is the true underlying vector.
 \end{Ko}
{This motivates the definition of bootstrap tests  as $\varphi_{Q_{\vr}}^\dagger:=\ind\{Q_{\vr}^\dagger \notin (-\infty,c_{Q_{\vr}^\dagger,1-\alpha}]\}$, where
$c_{Q_{\vr}^\dagger, 1-\alpha}$ denotes the conditional quantile of $Q_{\vr}^\dagger$.}
{The above results ensure that these tests are of asymptotic level $\alpha$, and we will use it %for different bootstrap test statistics, 
%based on
for the ATS.}

% $\varphi_{ATS_{\vr}}^\dagger:=\ind\{ ATS_{\vr} \notin (-\infty,c_{ATS_{\vr}^\dagger,1-\alpha}]\}$, where $c_{ATS_{\vr}^\dagger, 1-\alpha}$  the conditional quantile of $ATS_{\vr}^\dagger$ given the data.
% Similar bootstrap tests can be derived for other quadratic forms. These are $\varphi_{WTS_{\vr}}^\dagger:=\ind\{WTS_{\vr} \notin (-\infty,c_{WTS_{\vr}^\dagger,1-\alpha}]\}$, with $c_{WTS_{\vr}^\dagger, 1-\alpha}$ or more general   $\varphi_{Q_{\vr}}^\dagger:=\ind\{Q_{\vr}^\dagger \notin (-\infty,c_{Q_{\vr}^\dagger,1-\alpha}]\}$, with $c_{Q_{\vr}^\dagger, 1-\alpha}$ as the conditional quantile of $Q_{\vr}^\dagger$.
Fisher z-transformed vectors are often used in the analysis of correlation matrices instead of the original vectorized correlation matrices. Although the root of this approach is the distribution of the Fisher z-transformed correlation in the case of normally distributed observations, it is also used for tests without this distributional restriction, see, e.g., \cite{steiger1980}.  
In principle, we could also consider our tests together with the transformed vector. This approach assumes that all components of $\vzeta$ differ from one, which is always possible to ensure. We can define tests based on the transformation for each of our quadratic forms, including the tests based on bootstrap or Monte-Carlo simulations. 
Our simulations showed that the tests based on this transformation have more liberal behaviour than the original one, which was also mentioned in \cite{omPa:2012} for the case of bivariate permutation tests. Since all of our test statistics were already a bit liberal, it is not useful to consider these versions further. However, some details on this are given in the supplement. Instead, we propose an additional Taylor approximation of higher order as described below.
\section{Higher order of Taylor approximation}\label{Taylor}
The main result of \Cref{Asymptotics regarding the vectorized correlation} is the connection between the vectorized empirical covariance matrix and the vectorized empirical correlation matrix given through \[\sqrt{n_i}(\widehat \vr_i-\vr_i)=\vM(\vv_i,\vr_i)\sqrt{n_i} (\widehat \vv_i-\vv_i)+\lan_P(1).\] Since this is based on a Taylor approximation, it leads to the question of whether a higher order of approximation could be useful here. In this way we get \begin{equation}\sqrt{n_i}(\widehat \vr_i-\vr_i)=\vM(\widehat\vv_i,\widehat\vr_i) \vY_i+ f_{\widehat\vv_i,\widehat \vR_i}(\vY_i)/\sqrt{n_i}+\lan_P(\sqrt{{n_i^{-1}}}) \end{equation}\label{TayEq}
with 
$\vY_i\sim\mathcal{N}_{p}\left(\vnull_{p},\vSigma_i\right)$ and a function $f_{\widehat\vv_i,\widehat \vR_i}:\mathbb{R}^p\to \mathbb{R}^{p_u}$. This equation's derivation, together with the complex concrete form of the function $f_{\widehat\vv_i,\widehat \vR_i}$, can be found in the supplementary material. Although the part  $f_{\widehat\vv_i,\widehat \vR_i}(\vY_i)/\sqrt{n_i}$ is asymptotically negligible, for small sample sizes, it can affect the performance of the corresponding test.
For this purpose, we propose an additional Monte-Carlo based approach with this  Taylor approximation. Thereto for each group we generate $\vY_i\sim\mathcal{N}_{p}(\vnull_{p},\widehat \vSigma_i)$ and
transform it to $\vY_i^{Tay}:=\vM(\widehat\vv_i,\widehat\vr_i) \vY_i+ f_{\widehat\vv_i,\widehat \vR_i}(\vY_i)/\sqrt{n_i}$. This leads to the pooled vector $\vY^{Tay}=(n_1^{-1/2}{\vY_1^{Tay}}^\top,....,n_a^{-1/2}{\vY_a^{Tay}}^\top)^\top$. {So, we get a version of this quadratic form by calculating}
\[{Q_{\vr}^{Tay} =N[ \vC\vY^{Tay}  ]^\top \vE_N(\vC,\widehat \vUpsilon)[ \vC \vY^{Tay}]}.\] 
If we repeat this frequently enough, an asymptotic correct level $\alpha$ test can be constructed by comparing $\widehat Q_{\vr}$ with the empirical $1-\alpha$ of $Q_{\vr}^{Tay}$. Since there it is only necessary to calculate $\widehat \vUpsilon$ one time, this approach is clearly less time-consuming than the corresponding bootstrap approaches. \\\\
It would also be possible to use a Taylor approximation of higher order. Since the transformation gets more complex and our simulations suggest that this only affects very small sample sizes, we renounce this here.

\section{Combined tests for covariance and correlations}\label{combined}

The equality of two covariance matrices and the equality of two correlation matrices are interesting null hypotheses with a strong connection between them. 
As suggested by the Associate Editor, we here exemplify how the methodology extends to simultaneously comparing covariances and correlations, by using the results from \cite{sattler2022}. Hereto we use a so-called multiple contrast test (see, e.g. \cite{konietschke2013} for more details) and define
\[\vT=\begin{pmatrix}
T_1\\
\vdots\\
T_p
\end{pmatrix}= \sqrt{N}\left(\begin{pmatrix}
\widehat V_{1 11}\\
\widehat V_{1 22}\\
\vdots\\
\widehat V_{1 dd}\\
\widehat \vr_1
\end{pmatrix}- \begin{pmatrix}
\widehat V_{2 11}\\
\widehat V_{2 22}\\
\vdots\\
\widehat V_{2 dd}\\
\widehat \vr_2
\end{pmatrix}\right).\]
This statistic fulfills  $\vT\stackrel{\mathcal{D}}{\to}\mathcal{N}_p(\vnull_p,\vGamma)$ under the null hypothesis $\mathcal{H}_0:\vV_1=\vV_2$. The concrete form of the covariance matrix $\vGamma$ and the derivation can be found in the appendix.
In case of $\Gamma_{11},...,\Gamma_{pp}>0$, we could consider $\widetilde \vT:= \diag(\Gamma_{11},...,\Gamma_{pp})^{-1/2} \vT$, which is under the null hypothesis of equal covariances asymptotically normal distributed with expectation vector $\vnull_p$ and covariance/correlation matrix $\widetilde \vGamma:= \diag(\Gamma_{11},...,\Gamma_{pp})^{-1/2}\vGamma \diag(\Gamma_{11},...,\Gamma_{pp})^{-1/2}$. Then an asymptotic correct level $\alpha$ test would be given through $\max(\lvert\widetilde T_{1}\lvert ,....,\lvert\widetilde T_{p}\lvert)\geq z_{1-\alpha,2,\widetilde \vGamma}$. Here, with $z_{1-\alpha,2,\widetilde \vGamma}$ we denote a two-sided $(1-\alpha)$ equicoordinate quantile of $\mathcal{N}_p(\vnull_p,\widetilde \vGamma)$, see e.g. \cite{bretz2001}.\\
This allows us to  simultaneously consider the hypotheses of equal covariance and equal correlation. If $\max(\lvert\widetilde T_{d+1 }\lvert ,....,\lvert\widetilde T_{p}\lvert)\geq z_{1-\alpha,2,\widetilde \vGamma}$ then both groups have different dependence structures and therefore neither equal covariance matrices nor equal correlation matrices. In cases with $\max(\lvert\widetilde T_{d+1 }\lvert ,....,\lvert\widetilde T_{p}\lvert)< z_{1-\alpha,2,\widetilde \vGamma}$ but $\max(\lvert\widetilde T_{1}\lvert ,....,\lvert\widetilde T_{d}\lvert)\geq z_{1-\alpha,2,\widetilde \vGamma}$ both groups have the same dependence structure but different variances of the components, which implies equal correlation matrices but different covariance matrices. Moreover, from the fact which component of $\widetilde \vT$ is bigger than $z_{1-\alpha,2,\widetilde \vGamma}$, it gets clear where the differences might come from.\\\\
Unfortunately, the condition $\Gamma_{11},...,\Gamma_{pp}>0$ is difficult to verify, so again, a bootstrap approach would be a solution, which was, for example, also done in \cite{friedrich2017mats} or \cite{umlauft2019}. Since it is less complicated and at the same time has a better small sample approximation in \Cref{Simulations}, we introduce another Taylor-based Monte-Carlo approach.
To this end, we use that with matrix $\vA\in\mathbb R^{d\times p}$ given through
$\vA=\sum_{\ell=1}^{d} \ve_{\ell,d}\ \ve_{{\va}_{\ell},p}^\top $
it holds 

\[\sqrt{n_i}\left(
\begin{pmatrix}
\widehat V_{i 11}\\
\widehat V_{i 22}\\
\vdots\\
\widehat V_{i dd}\\
\widehat \vr_i
\end{pmatrix}-
\begin{pmatrix}
V_{i 11}\\
V_{i 22}\\
\vdots\\
V_{i dd}\\
\vr_i
\end{pmatrix}\right)= \begin{pmatrix}
\vA\\
\vM(\vv_i,\vr_i)
\end{pmatrix} \vY_i+ \frac 1 {\sqrt{n_i}}\begin{pmatrix}\vnull_d\\
f_{\widehat\vv_i,\widehat \vR_i}(\vY_i)
\end{pmatrix}+\lan_P(\sqrt{n_i^{-1}})\]
with 
$\vY_i\sim\mathcal{N}_{p}(\vnull_p,\vSigma_i)$. Based on this, we  generate
$\vY_i\sim\mathcal{N}_{p}(\vnull_{p},\widehat \vSigma_i)$ similar to \Cref{Taylor} and
transform it to \[\vY_i^{Tay}:= \begin{pmatrix}
\vA\\
\vM(\widehat \vv_i,\widehat \vr_i)
\end{pmatrix} \vY_i+ \frac 1 {\sqrt{n_i}}\begin{pmatrix}\vnull_d\\
f_{\widehat\vv_i,\widehat \vR_i}(\vY_i)
\end{pmatrix}.\]

With this transformed vectors we calculate $\vT^{1,Tay}=\sqrt{N}(n_1^{-1/2}\vY_1^{Tay}-n_{2}^{-1/2}\vY_2^{Tay})$, which has the same asmyptotic distribution as $\vT$, and repeat this $B$ times to get $\vT^{1,Tay},...,\vT^{B,Tay}$. Then, for $\ell=1,...,p$, we denote with $q_{\ell,\beta}^{Tay}$ the empirical $(1-\beta)$ quantile for $T_\ell^{1,Tay},..., T_\ell^{B,Tay}$. To control the family-wise type-I-error rate by $\alpha$, the appropriate $\beta$ can be found through
\[\widetilde \beta= \max\left(\beta \in \left\lbrace 0,\frac{1}{B},..., \frac{B-1}{B}\right\rbrace\Big\lvert \frac{1}{B} \sum\limits_{b=1}^B \max\limits_{\ell=1,...,p}\left(T_\ell^{b,Tay}>q_{\ell,\beta}^{Tay}\right)\leq \alpha\right).\]

Then we get an asymptotic correct level $\alpha$ test for the null hypothesis of equal covariance matrices by rejecting the null hypothesis if and only if 
\[\max\limits_{\ell=1,...,p}\left( T_\ell>q_{\ell,\widetilde \beta}^{Tay}\right)=1\]
or equivalent

\[\max\limits_{\ell=1,...,p} \left(\frac{T_\ell}{q_{\ell,\widetilde \beta}^{Tay}}\right)>1,\]
where we set $0/0:=1$. Each component of the vector-valued test statistic $\vT$ is treated in the same way since the same $\widetilde \beta<\alpha$ is used. This procedure allows the same conclusions on the reason for the rejection as the above approach base on the equicoordinate quantile.

Of course, this combined test could be generalized to compare more than two groups by using an appropriate Tukey-type contrast matrix (see \cite{konietschke2013}).

\section{Simulations}\label{Simulations}
{We analyze the type-I-error rate and power for different hypotheses to investigate the performance.
Here, we focus on hypotheses with an implemented algorithm to have an appropriate competitor.}
In the  \textsc{R}-package \textit{psych} by \cite{psych} {several} of these tests are included, like $\varphi_{Jennrich}$ from \cite{jennrich1970} and  $\varphi_{Steiger_{Fz}}$ from \cite{steiger1980}  for equality of correlation matrices. Testing whether the correlation matrix is equal to the identity matrix can be investigated with  $\varphi_{Bartlett}$ from \cite{bartlett1951} and again with $\varphi_{Steiger_{Fz}}$ resp. $\varphi_{Steiger}$. Hereby $\varphi_{Steiger_{Fz}} $ is the same test statistic as $\varphi_{Steiger}$ but uses a Fisher z-transformation on the vectorized correlation matrices. Therefore we consider the following hypotheses:
\begin{itemize}
\item[$A_{\vr}$:]Homogeneity of correlation matrices:
 $\mathcal H_0^{\vr}: \vR_1 =  \vR_2 $,\vspace*{0.1cm}
\item[$B_{\vr}$:]Diagonal structure of the covariance matrix $\mathcal  H_0^{\vr}: \vR_1 = \vI_p\text{ resp. }\vr_1=\vnull_{p_u}$,
\end{itemize}
with $\alpha=0.05$. Further hypotheses are investigated in the supplementary material {together with other settings and dimensions}.
The hypothesis matrices are chosen as the projection matrices $\vC(A_{\vr})={\vP_2}\otimes \vI_{p_u}$ and $\vC(B_{\vr})=\vI_{p_u}$ while $\vzeta$ is in both cases a zero vector with 
appropriate dimension.\\
We use 1,000 bootstrap steps for our parametric bootstrap, 10,000 simulation steps for the Monte-Carlo approach and 10,000 runs for all tests to get reliable results.
 For $\varphi_{Steiger}$ and $\varphi_{Steiger_{Fz}}$ {from \cite{psych}},  the actual test statistic is multiplied with the factor $({N-3})/{N}$. This approach is based on a specific result of the Fisher z-transformation of the correlation vector of normally distributed random vectors { for small sample sizes.
Asymptotically this factor has no influence, but it is also used for  $\varphi_{Steiger}$, where no Fisher z-transformation was done}. 
 
To get a better impression of the impact of such a multiplication, we also include our  ATS with parametric bootstrap using such multiplication and denote this with an m for multiplication. This also simplifies the comparison of the tests under equal conditions. At last, we simulated Monte-Carlo based ATS {using} our Fisher z-transformation. We denote it by ATSFz, while an additional version ATSFz-m is simulated, which is formed by multiplication with the factor {$(N-3)/N$}.\\
To have a comparative setting to \cite{sattler2022}  we used d=5 and therefore $p_u=10$, while for one group we have {$n\in \{25,50,125,250\}$} and for two groups we have  
$n_1=0.6\cdot N$ and $n_2=0.4\cdot N$ with $N \in\{50, 100, 250, 500\}$. We considered $5$-dimensional observations generated independently according to the model $\vX_{ik} = \mu_i + \vV^{1/2} \vZ_{ik}, i=1,\dots, a, k=1,\dots,n_i$ with $\vmu_1=(1^2,2^2,...,5^2)/4$ and error terms based on
\begin{itemize}
\item a standardized centred t-distribution with 9 degrees of freedom.
\item a standard normal distribution, i.e. $Z_{ikj} \sim  \mathcal {N}(0,1).$
\item a standardized centred skewed normal distribution with location parameter $\xi=0$, scale parameter $\omega=1$ and $\alpha=4$. The density of a skewed normal distribution is given through $\frac{2}{\omega} \varphi\left(\frac{x-\xi}{\omega}\right)\Phi\left(\alpha\left(\frac{x-\xi}{\omega}\right)\right)$, where $\varphi$ denotes the density of a standard normal distribution and $\Phi$ the according distribution function. 
\end{itemize}
 For $A_{\vr}$ we use $(\vV_{1})_{ij}=1-{\lvert i-j\lvert/2d}$  for the first group and for the second group we multiply this covariance matrix with $\diag(1,1.2,...,1.8)$. Thus we have a setting where the covariance matrices are different, but the correlation matrices are equal. To investigate $B_{\vr}$ we use the matrix,  $\vV=\diag(1,1.2,...,1.8)$.
\subsection{Type-I-error}
The results of hypothesis $A_{\vr}$ can be seen in \Cref{tab:SimA1Cor} for the Toeplitz covariance matrix. Here, the values in the $95\%$ binomial interval $[0.0458, 0.0543]$ are printed in bold.
It is interesting to note that the type-I-error rate of $\varphi_{Steifer_{Fz}}$ and $\varphi_{Jennrich}$ differs more and more from the $5\%$ rate for increasing sample sizes. Therefore, these tests should not be used, at least for our setting.
In contrast, all of our tests are too liberal but show a substantially better type-I-error rate. Moreover, these tests fulfil Bradley's liberal criterion (from \cite{bradley1978}) for $N$ larger than 50. {This criterion is often consulted by applied statisticians, for example in quantitative psychology. It states that a procedure is 'acceptable' if the type-I-error rate is between $0.5\alpha$ and $1.5\alpha$.} \\\\
 Similar to the results for covariance matrices, the bootstrap version has slightly better results than Monte-Carlo based tests, while the error rates get closer to the nominal level for larger sample sizes. In each setting, the Monte-Carlo based ATS formulated for Fisher z-transformed vectors has a better performance than the classical one, especially for smaller sample sizes.
\captionsetup[table]{position=below,skip=-0.1cm}
\setlength{\tabcolsep}{3.2pt}
\begin{table}[htbp]
\begin{center}
\begin{tiny}
\begin{tabular}{x{1.3pt}lx{1.3pt}c|c|c|cx{1.3pt}c|c|c|cx{1.3pt}c|c|c|cx{1.3pt}}\specialrule{1.3pt}{0pt}{0pt}
   \multicolumn{1}{x{1.3pt}lx{1.3pt}}{} &\multicolumn{4}{ cx{1.3pt}}{$t_9$}&\multicolumn{4}{cx{1.3pt}}{Normal}&\multicolumn{4}{cx{1.3pt}}{Skew Normal}\\
   %\specialrule{1.3pt}{0pt}{0pt}
\specialrule{1.3pt}{0pt}{0pt}
    \hspace*{.1cm}N&   50&100&250&500&   50&100&250&500&   50&100&250&500
       \\ \specialrule{1.3pt}{0pt}{0pt}
\hspace*{-.2cm} ATS-Par& .0755 & .0657 & .0568 & \bf{.0538} & .0703 & .0627 &\bf{ .0524} & \bf{.0519} & .0734 & .0635 &\bf{ .0530} & .0567 \\  \hline
\hspace*{-.2cm} ATS-Par-m& .0653 & .0600 & .0553 & \bf{.0528} & .0609 & .0579 & \bf{.0502} & \bf{.0511} & .0626 & .0602 &\bf{ .0518 }& .0557 \\  \hline
\hspace*{-.2cm} ATS & .0818 & .0652 & .0581 & \bf{.0542} & .0776 & .0640 & \bf{.0527} & \bf{.0511} & .0798 & .0666 &\bf{ .0536} & .0567 \\   \hline
\hspace*{-.2cm} ATS-Tay  & .0603 &\bf{ .0528 }& \bf{.0523} &\bf{ .0511} & .0567 &\bf{ .0536} &\bf{ .0482} & \bf{.0495} & .0579 & \bf{.0536} & \bf{.0486} & \bf{.0539} \\   \hline
\hspace*{-.2cm}  ATSFz& .0781 & .0634 & .0568 &\bf{ .0540} & .0726 & .0628 & \bf{.0516} & \bf{.0508} & .0757 & .0659 &\bf{ .0528} & .0562 \\  \hline
\hspace*{-.2cm} ATSFz-m & .0683 & .0592 & .0556 &\bf{ .0530 }& .0623 & .0582 & \bf{.0503} & \bf{.0501} & .0642 & .0604 &\bf{ .0512} & .0554 \\ 
  \hline
 \hspace*{-.2cm} SteigerFz & .0931 & .1106 & .1303 & .1296 & .0908 & .1130 & .1249 & .1229 & .0902 & .1158 & .1177 &.1272 \\ 
\specialrule{1.3pt}{0pt}{0pt}
\end{tabular}
\end{tiny}
\end{center}
\caption{Simulated type-I-error rates ($\alpha=5\%$) in scenario $A_r$ ($\mathcal{H}_0^{\vr}:\vR_{1}=\vR_{2}$)  for ATS and Steiger's test.  The observation vectors have dimension 5, covariance matrix $(\vV_1)_{ij}=1-{\lvert i-j\lvert/2d}$ resp. $\vV_2=\diag(1,1.2,...,1.8)\vV_1$ and there is always the same relation between group samples size with $n_1:=0.6\cdot N$ resp. $n_2:=0.4\cdot N$.}\label{tab:SimA1Cor}\end{table}
  It can be seen that a correction factor, like $(N-3)/N$, which was used for ATS-Par-m,  clearly improves the small sample performance. 
 In contrast to the results from \cite{sattler2022}, for testing correlation matrices, the small sample approximation seems a bit worse, making such a correction factor useful. 
The best performance for smaller sample sizes can be seen for our Taylor-based Monte-Carlo approach, which has for {$N>50$} an error rate within the 95\% interval. 
 Therefore, for this setting, we recommend ATS-Tay, while for larger sample size, ATS-Par-m and ATSFz-m also show good results.\\
 
For hypothesis $B_{\vr}$, the results are included in \Cref{tab:SimBCor}.
Here, our Taylor-based approach has slightly conservative values, and therefore ATS-Para-m and ATSFz-m have the best results of our test statistics, while the results are slightly better than for hypothesis $A_{\vr}$.
 For example, the number of values in the $95\%$ binomial interval  $[0.0458, 0.0543]$, printed bold, clearly increases.
  With the correction factor, these tests have a better type-I-error rate than $\varphi_{Steiger}$ in all settings and are comparable to $\varphi_{Steiger_{Fz}}$.
\setlength{\tabcolsep}{3.1pt}
\begin{table}[htp]
\begin{center}
\begin{tiny}
\begin{tabular}{x{1.3pt}lx{1.3pt}c|c|c|cx{1.3pt}c|c|c|cx{1.3pt}c|c|c|cx{1.3pt}}\specialrule{1.3pt}{0pt}{0pt}
   \multicolumn{1}{x{1.3pt}lx{1.3pt}}{} &\multicolumn{4}{ cx{1.3pt}}{$t_9$}&\multicolumn{4}{cx{1.3pt}}{Normal} &\multicolumn{4}{cx{1.3pt}}{Skew Normal}\\
   %\specialrule{1.3pt}{0pt}{0pt}
\specialrule{1.3pt}{0pt}{0pt}
  \hspace*{.1cm}N&   25&50&125&250&   25&50&125&250&   25&50&125&250
       \\ \specialrule{1.3pt}{0pt}{0pt}
\hspace*{-.2cm} ATS-Par & .0863 & .0620 & \bf{.0537} & \bf{.0501} & .0835 & .0605 & \bf{.0493} & .0559 & .0846 & .0612 & \bf{.0543} & \bf{.0509} \\  \hline
\hspace*{-.2cm}  ATS-Par-m & \bf{.0507} & .0448 & \bf{.0473} & \bf{.0471} & .0455 & .0422 & .0429 & \bf{.0517} & \bf{.0484} & .0440 & \bf{.0470} &\bf{ .0481} \\    \hline
\hspace*{-.2cm} ATS  &.0974 & .0646 &\bf{ .0533} & \bf{.0492} & .0940 & .0621 & \bf{.0506} & \bf{.0541} & .0961 & .0630 & \bf{.0531} & \bf{.0507} \\  \hline
\hspace*{-.2cm}  ATS-Tay   &\bf{ .0512} & .0394 & .0417 & .0431 & .0556 & .0436 & .0417 &\bf{ .0500} &\bf{ .0504} & .0408 & .0433 & .0456 \\  \hline
\hspace*{-.2cm}  ATSFz & .0916 & .0607 &\bf{ .0522} & \bf{.0490} & .0901 & .0605 & \bf{.0495} & \bf{.0541} & .0914 & .0606 & \bf{.0523} & \bf{.0499} \\ \hline
\hspace*{-.2cm}  ATSFz-m   & \bf{.0523} & \bf{.0459} & .0453 & \bf{.0458} & \bf{.0508} & .0436 & .0420 & \bf{.0512} & \bf{.0526} & .0442 & \bf{.0462} & \bf{.0462} \\  \hline
% \hspace*{-.2cm}  ATS-Par-Cov\hspace*{-.15cm}  \hspace{-.1cm} & .0732 & .0545 & .0441 & .0452 & .0667 &  \bf{.0539} & \bf{ .0466} &  \bf{.0496} & .0712 & .0557 &  \bf{.0467} & .0451 \\ \hline 
\hspace*{-.2cm} Steiger  & .0220 & .0326 & .0419 & \bf{.0457} & .0255 & .0355 & .0421 & \bf{.0507} & .0252 & .0333 & \bf{.0458} & \bf{.0479} \\   \hline
\hspace*{-.2cm}   SteigerFz  \hspace{-.1cm}  &\bf{ .0515} & \bf{.0471} & \bf{.0489} & \bf{.0491} & .0558 & \bf{.0517} & \bf{.0492} & \bf{.0536} & .0548 &\bf{.0486} & \bf{.0542} & \bf{.0499} \\  \hline
\hspace*{-.2cm}  Bartlett & \bf{.0470} & \bf{.0480} & \bf{.0488} & \bf{.0500} & \bf{.0525} & \bf{.0535} & \bf{.0479} & \bf{.0539} &\bf{ .0531} & .0456 & \bf{.0516} & \bf{.0506} \\  
\specialrule{1.3pt}{0pt}{0pt}
\end{tabular}
\end{tiny}
\end{center}
\caption{Simulated type-I-error rates ($\alpha=5\%$) in scenario $B_{\vr}$ ($\mathcal{H}_0^{\vr}:\vr_{1}=\vnull_{10}$)  for ATS, Steiger's and Bartlett's test. The observation vectors have dimension 5 and covariance matrix $(\vV)=diag(1,1.2,1.4,1.6,1.8)$.}\label{tab:SimBCor}
\end{table}

 Nevertheless, $\varphi_{Bartlett}$ is a test only developed for this special hypothesis and therefore has an excellent error rate through all distributions. But in addition to the type-I error rate, also other properties are highly relevant.

 \subsection{Power}
The ability to detect deviations from the null hypothesis is also an important test criterion. To this aim, we also investigate the power of some of the tests mentioned above. We choose a quite simple kind of alternative suitable for our situation. 
As covariance matrix we consider $\vV_1+\delta\cdot \vJ_d$ for $\delta \in [0,3.5]$ in hypothesis $A_{\vr}$ and for $\delta\in [0,0.75]$ in hypothesis $B_{\vr}$. 
 The reason for this considerable difference in the $\delta$ range is that for hypothesis $B_{\vr}$, the second summand changes the setting from uncorrelated to correlated. For hypothesis $A_{\vr}$, it just increases the correlations, which is clearly more challenging to detect. 
\captionsetup[figure]{position=below,skip=-0.1cm}
\begin{figure}[h]
\begin{minipage}[t]{0.99\textwidth}\vspace{0pt} 
\includegraphics[width=\textwidth,trim= 11.5mm 3mm 10.mm 6mm,clip]{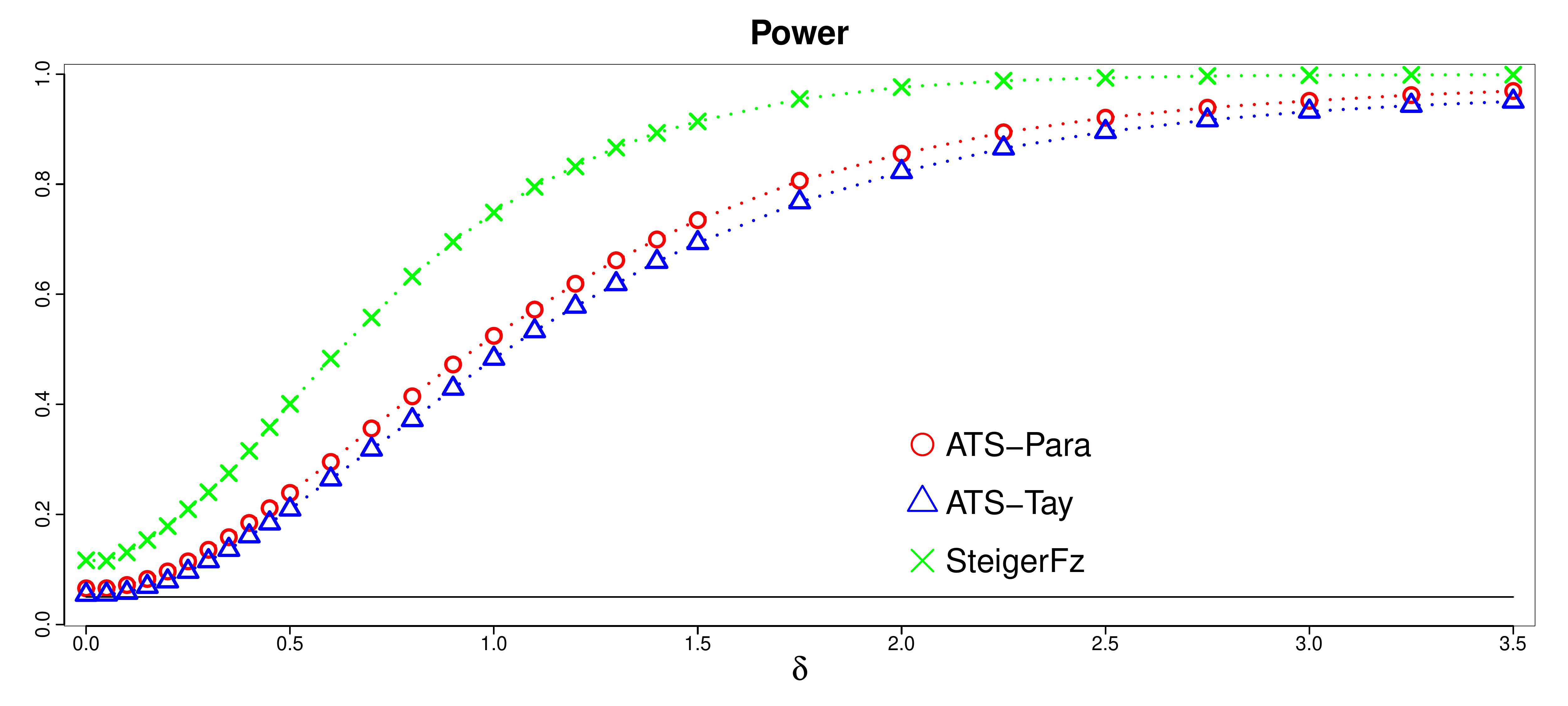} 
\end{minipage}\\ 
\begin{minipage}[t]{0.99\textwidth}\vspace{0pt} 
\includegraphics[width=\textwidth,trim= 11.5mm 4mm 10mm 5mm,clip]{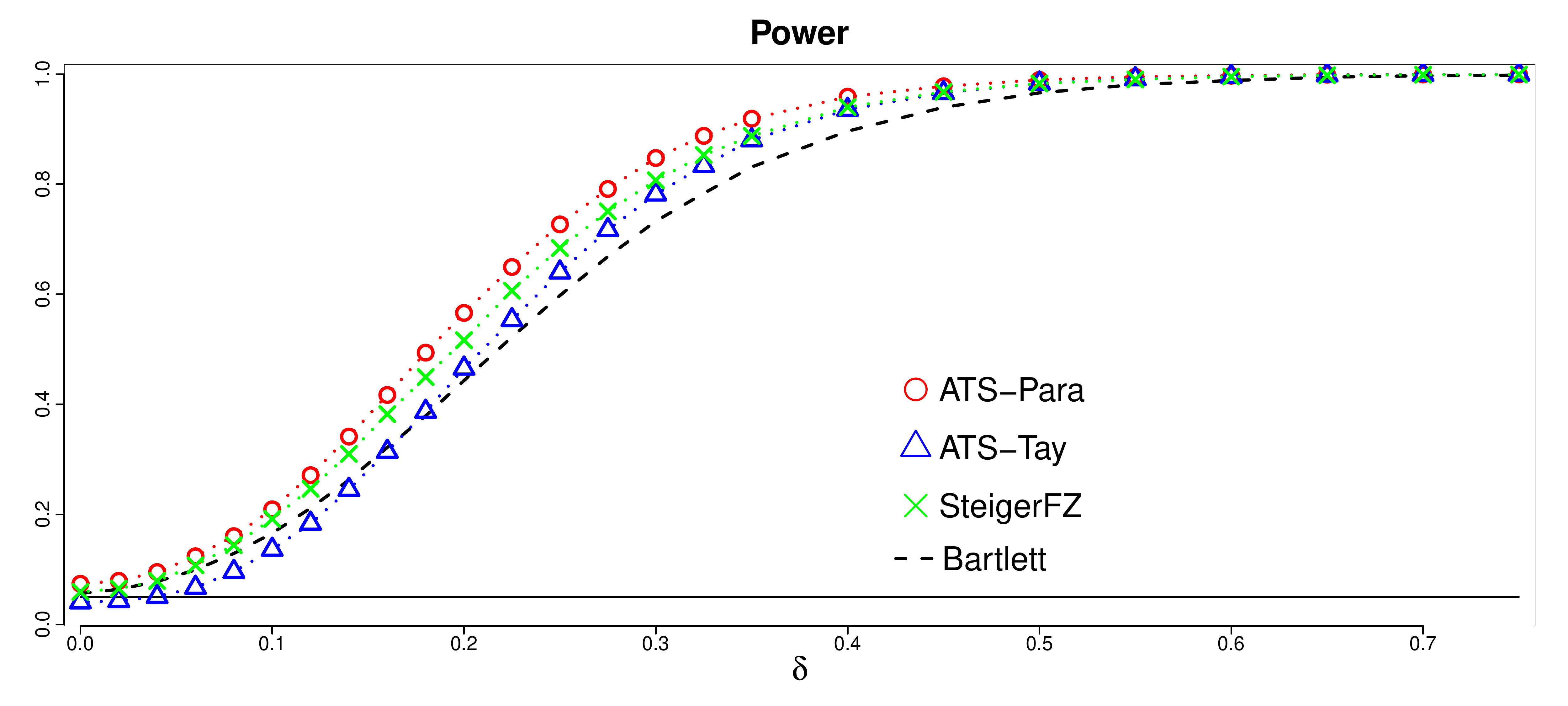} 
\end{minipage} % 
\caption{ 
Simulated power curves of different tests for the hypothesis $A_{\vr}$ ($\mathcal H_0^{\vr}: \vR_1 =  \vR_2 $) above, and
hypothesis $B_{\vr}$ ($\mathcal H_0^{\vr}: \vr_1 =  \vnull_{10} $) below, for a 5-dimensional skewed normal distribution. For hypothesis $A_{\vr}$ it holds $n_1=60$, $n_2=40$ and the covariance matrix is
$(\vV_2)_{ij}=1-{\lvert i-j\lvert/2d}$ resp. $\vV_1=\vV_2+\delta \vJ_5$.  The covariance matrix  for hypothesis $B_{\vr}$ is $\vV=\vI_5+\delta \vJ_5$ and $n_1=50$.   }

\label{fig:PowerCorAB} 
\end{figure} 
Due to computational reasons, we simulate only one sample size, which is $N=250$  resp. $n_1=125$ and consider error terms based on skewed normal distribution, while results for the gamma distribution can be found in the supplementary material. We simulate only the tests with good results for their type-I-error rate, which were for $A_{\vr}$  $\varphi_{ATS_{\vr}^\dagger}$ and  $\varphi_{ATS_{\vr}-Tay}$ as well as $\varphi_{SteigerFZ}$ as competitor, despite its performance in \Cref{tab:SimA1Cor}.
Because of the similarity of the results from the parametric bootstrap and the more classical Monte-Carlo based approach, we do not consider further test statistics here. 
For hypothesis $A_{\vr}$ \Cref{fig:PowerCorAB} shows that the Taylor-approximation makes the test {slightly} less liberal, and therefore reduces the power. This effect can be seen as a shift and does not influence the slope. In general, the power of both approaches is quite good, also $\varphi_{SteigerFZ}$ has even higher power, which furthermore increases faster. But since this test becomes even more liberal for increasing sample size, this is not surprising and makes it for this setting not recommendable.
Based on the results from \Cref{tab:SimBCor} for hypothesis $B_{\vr}$ we consider$\varphi_{ATS_{\vr}^\dagger}$ and  $\varphi_{ATS_{\vr}-Tay}$ as well as $\varphi_{SteigerFz}$ and $\varphi_{Bartlett}$ while the setting is the same. 
In \Cref{fig:PowerCorAB} for hypothesis $B_{\vr}$ it can be seen that $\varphi_{ATS_{\vr}-Tay}$ has for smaller $\delta$ similar power as $\varphi_{Bartlett}$ which was specially developed for this hypothesis, while for larger deviation, the power of the Monte-Carlo approach increases. Here, $\varphi_{SteigerFZ}$ has slightly less power than the parametric bootstrap due to the slightly liberal behaviour of the bootstrap approach.\\

The type-I error rate and the power for both hypotheses show that our developed tests are useful in many situations, although partially large sample sizes are necessary for good results. This is a known fact for testing hypotheses regarding correlation matrices, which was, for example, mentioned in \cite{steiger1980}.
 Therefore the results of ATS-Tay investigating $A_{\vr}$ for smaller sample sizes are even more convincing.\\
\enlargethispage{\baselineskip}
{All in all,  the results of this section, together with the additional results from the supplement, allow us to give some general recommendations.
Since, SteigerFz has a type-I-error rate of more than 9\% in \Cref{tab:SimA1Cor} which also grows for increasing sample sizes, it is not useful apart from single hypotheses like $\mathcal{H}_0^{\vr}:\vr_{1}=\vnull_{10}$. This hypothesis can be checked with Bartlett, which had good results but only allows this hypothesis.

On the contrary, for all considered hypotheses, ATS-Par-m and ATS lead to good results for moderate to large sample sizes, while they are liberal for small sample sizes. Such liberality was, for example also mentioned in \cite{perreault2022} for tests based on Kendall's $\tau$. As intended, the bootstrap improve the behaviour for smaller sample sizes, but not enough.
In case of small sample sizes, the Taylor-based approach is recommended as it exhibited good small sample performance in all settings. Only for larger samples sizes it is outperformed by ATS-Par-m.\\
Although tests regarding correlation matrices are challenging and known to need large sample size, this lead to useful tests, which also provides more flexibility and possible applications than existing ones.}

%the simulation showed that our developed tests have substantially better performance than the existing tests in hypothesis $A_{\vr}$, while for $B_{\vr}$, they are nearly as good as 
%$\varphi_{Bartlett}$ but provides more flexibility and possible applications.

\section{\textsc{Illustrative Data Analysis}}\label{Illustrative Data Analysis} 
To illustrate the method, we take a closer look at the EEG data set from the \textsc{R}-package \textit{manova.rm} by \cite{manova}. In this study from \cite{staffen2014},  conducted at the University Clinic of Salzburg (Department of Neurology), electroencephalography (EEG) data were measured from 160 patients with different diagnoses of impairments. These are Alzheimer's disease (AD), mild cognitive impairment (MCI), and subjective cognitive complaints (SCC). Thereby, the last diagnosis can be differentiated between subjective cognitive complaints with minimal cognitive dysfunction (SCC+) and without (SCC-).

MANOVA-based comparisons were already made in \cite{bathke2018}. However, covariance and correlation comparisons were only made in a descriptive way. We now complement their analyses.
In \Cref{tab:EEG1} the number of patients divided by sex and diagnosis can be found. Since in \cite{bathke2018} and \cite{sattler2022} no distinction between SCC+ and SCC- {was} made, we consider both together as diagnosis SCC. 

\begin{table}[h]

\centering
\begin{small}
\caption{Number of observations for the different factor level combinations of sex and diagnosis.}
\label{tab:EEG1}

\begin{tabular}{l|c|c|c|c|}
&AD&MCI&SCC+&SCC-\\
\hline
male&12&27&14&6\\
\hline
female&24&30&31&16\\
\hline
\end{tabular}
\end{small}
\end{table}

The observation vector's dimension is $d=6$, since there are two kinds of measurements (z-score for brain rate and Hjorth complexity) and three different electrode positions (frontal, temporal, and central), and therefore $p_u=15$. For the evaluation of our results, we should keep in mind that all sample sizes are rather small in relation to this dimension.
The considered hypotheses are:
\begin{itemize}
\item[a)] Homogeneity of correlation matrices between different diagnoses,
\item[b)] Homogeneity of correlation matrices between different sexes,

\end{itemize} 
while we will denote the corresponding hypothesis regarding the covariance matrix with $\mathcal{H}_0^{\vv}$.\\

In \cite{sattler2022}, homogeneity of covariance matrices between different diagnoses as well as different sexes were investigated. Here, we consider the more general hypothesis of equal correlation matrices between the diagnoses and the sexes. Thereby, it is interesting to compare the results from the homogeneity of covariance matrices with those from testing the homogeneity of correlation matrices. We expect higher p-values for equality of correlation through the larger hypothesis, but each rejection of equal correlation matrices directly allows us to reject the corresponding equality of covariance matrices. In \Cref{EEGResultate} for both hypotheses, the p-values for the ATS with parametric bootstrap are displayed, while for equality of correlations, we additionally use our test based on the Taylor-based Monte-Carlo approach. For all considered bootstrap tests, 10,000 bootstrap runs are done, as well as 10,000 Monte-Carlo-steps.\\

\captionsetup[table]{position=below,skip=.4cm}
\begin{table}[htp]
\centering
\begin{small}
\begin{tabular}{x{1.3pt}llx{1.3pt}cx{1.3pt}cx{1.3pt}cx{1.3pt}cx{1.3pt}}
\specialrule{1.3pt}{0pt}{0pt}

&&ATS-Par for $\mathcal{H}_0^{\vv}$ &ATS-Par   for $\mathcal{H}_0^{\vr}$&ATS-Tay   for $\mathcal{H}_0^{\vr}$ \\
 &&p-value   &  p-value&  p-value \\
\specialrule{1.3pt}{0pt}{0pt}
male&AD vs. MCI&.1000 & \textbf{.0389}&.0866  \\ \hline
male&AD vs. SCC& \textbf{.0452 } & \textbf{.0077}&\textbf{.0363}  \\ \hline
male&MCI vs. SCC & \textbf{.0289} & .0712&.1279 \\ \hline
female&AD vs. MCI&  .0613 & .3626&.3749  \\ \hline
female&AD vs. SCC& \textbf{.0128} & .4937&.4892  \\ \hline
female&MCI vs. SCC & .5656 & .8788&.8693 \\ \hline
AD& male vs. female & .1008 & \textbf{.0290}&.0658 \\ \hline
MCI& male vs. female   & .2455 & .6702&.6577   \\ \hline
SCC& male vs. female   & .2066 & .1744&.2238 \\ 

\specialrule{1.3pt}{0pt}{0pt}\end{tabular}
\caption{P-values of different ATS for testing equality of correlation matrices and equality of covariance matrices. }\label{EEGResultate}
\end{small}

\end{table}

Interestingly, for two hypotheses, the p-value ATS-Par for equal correlation matrices is rejected at level $5\%$, while we could not reject the smaller hypothesis of equal covariance matrices. But for both hypotheses, the sample sizes are rather small with $N<40$. Our simulation results for $d=5$ showed that the ATS with parametric bootstrap is too liberal for small sample sizes, which might be the reason why the larger hypotheses can be rejected, and the smaller ones can not. But also ATS-Tay, which had a better small sample performance, rejected $\mathcal{H}_0^{\vr}$ one time, while $\mathcal{H}_0^{\vv}$ was not rejected.
Moreover, it can be seen that the difference between some hypotheses is relatively small, like for the first three hypotheses, but it can also be quite large as for the comparison of women with $AD$ and with SCC. This shows that from a rejection of $\mathcal{H}_0^{\vv}$, no conclusion on $\mathcal{H}_0^{\vr}$ can be drawn.

Due to the small sample size in relation to the dimension of the vectorized correlation matrix $p_u=15$, even the results of ATS-Tay are a bit liberal. But nevertheless, the corresponding rejections allow various ideas for further investigations.

\section{\textsc{Conclusion \& Outlook}}\label{Conclusion & Outlook}

In the present paper, a series of new test statistics was developed to check general null hypotheses formulated in terms of correlation matrices. The proposed method can be used for many popular quadratic forms, and the low restrictions allow their application for a variety of settings. In fact, existing procedures have more restrictive assumptions or could only be used for special hypotheses or settings. The diversity of possible null hypotheses, these low restrictions and the easy possibility of expansion, like using a Fisher z-transformation, make our approach attractive.\\
We proved the asymptotic normality of the estimation error of the vectorized empirical correlation matrix under the assumption of finite fourth moments of all components. Based on this, test statistics from a quite general group of quadratic forms were presented, and a bootstrap technique was developed to match their asymptotic distribution. To investigate the properties of the corresponding bootstrap test, an extensive simulation study was done. This also allows checking our test statistic based on a Fisher z-transformation. Here, hypotheses for one and two groups were considered, and the type-I-error control and the power to detect deviations from the null hypothesis were compared to existing test procedures. The developed tests outperform existing procedures for some hypotheses, while they offer good and interesting alternatives for others. Also, it is a known fact that for testing correlations, a large sample size is required. Here, for group sample sizes larger than 50, Bradley's liberal criterion was often fulfilled. 
Especially our Taylor-based approach was convincing for small sample sizes with multiple groups.
An illustrative data analysis completed our investigations.\\

In future research, {we will take a closer look at our newly proposed combined test for simultaneously testing the hypotheses of equal covariance matrices and correlation matrices. Thereto extensive simulations will be done, to, among other things, examine the performance of the corresponding Taylor-approach. This relates to studying} the large number of possible null hypotheses included in our model. For
example, tests for given covariance structures (such as compound symmetry or autoregressive) or structures of correlation matrices with unknown parameters are of great interest. Since there are heterogeneous versions of many popular structures, testing such structures or other patterns can be seen as a combination of testing hypotheses regarding covariance matrices and hypotheses regarding correlation matrices. 
Moreover, an investigation of Monte-Carlo approaches for a higher order of Taylor approximation for real small sample sizes could be interesting.

\section{Statements and Declarations}
\subsection{\textsc{Acknowledgment}}
{We like to thank two anonymous referees and the editor. A special thank goes to the Associate Editor for suggesting the idea of simultaneously studying covariances and correlations.
Moreover, we} would like to thank the German Research Foundation for the support received within project PA 2409/3-2.
\subsection{Competing Interests}
The authors report there are no competing interests to declare.

\newpage
\part*{Appendix}

\section{Further test statistics}
Here we want to introduce two further approaches that define corresponding quadratic form base test procedures.
\subsection{Fisher z transformation}
The Fisher z transformation, based on the function $f:\mathbb{R}{\to} \mathbb{R}$, $x\mapsto  \frac 1 2 \ln\left(\frac{1+x}{1-x}\right)$ is frequently used to transform correlation coefficients. Working with vectorized covariance matrices, this means transforming $\vr$ and afterwards multiplicating with the hypothesis matrix. Since this approach was too liberal in our simulations, which was similarly mentioned in \cite{omPa:2012}, we use it in another way.
With $\vf:\mathbb{R}^{ap_u}{\to } \mathbb{R}^{ap_u}$, $(x_1,...,x_{ap_u}){\mapsto} \left( \frac 1 2 \ln\left(\frac{1+x_1}{1-x_1}\right),...,\frac 1 2 \ln\left(\frac{1+x_{ap_u}}{1-x_{ap_u}}\right)\right)$ and the delta method,  we get
\[\sqrt{N} (\vf(\vC\widehat \vr) -\vf(\vzeta))\stackrel{\mathcal H_0}{=}\sqrt{N} (\vf(\vC\widehat \vr) -\vf(\vC\vr))\stackrel{\mathcal {D}}{\longrightarrow} {\mathcal{N}_{m}\left(\vnull_{m},\vf'(\vC\vr)\vC\vUpsilon\vC^\top \vf'(\vC\vr)\right)},\]

while $ \vf'(x_1,...,x_{ap_u})=\diag(1-x_1^2,....,1-x_{ap_u}^2)^{-1}$. Through the {consistency} of $\widehat \vUpsilon$ we can estimate this unknown covariance matrix through   
$\vf'(\vC\widehat\vr)\vC\widehat\vUpsilon\vC^\top  \vf'(\vC\widehat \vr)$. 
This allows us to formulate appropriate quadratic forms to check the null hypothesis $\mathcal{H}_0:\vC\vr=\vzeta$. Another consistent estimator under the null hypothesis would be
$ \vf'(\vzeta)\vC\widehat\vUpsilon \vC^\top  \vf'(\vzeta)$, but we expect less power to detect derivation from the null hypothesis.
\subsection{Wild Bootstrap}
With i.i.d. random weights  $W_{i1},...,W_{i n_i}$, $i=1,...,a,$ independent of the data, with $\E(W_{i1})=0$ and $\Var(W_{i1})=1$ and $\widetilde\vX_{ik}:=\vX_{ik}-\overline \vX_i$ we define the wild bootstrap sample through
\[\vY_{i k}^\star=W_{ik}\cdot  \vM(\widehat \vv_i,\widehat \vr_i)\left[\vech(\widetilde \vX_{ik} \widetilde \vX_{ik}^\top)-n_i^{-1} \sum\nolimits_{\ell=1}^{n_i}\vech(\widetilde \vX_{i\ell}\widetilde {\vX}_{i\ell}^\top) \right].\]
Hereby we multiplicated the wild bootstrap sample from \cite{sattler2022}, which we here notated with $\vZ_{ik}^\star$,  with the matrix $\vM(\widehat \vv_i,\widehat \vr_i)$ to adapted it for the required covariance matrix $\vUpsilon_i$. With $\widehat{\vSigma}_i^\star$ as the empirical covariance matrix of $\vY_{i1}^\star,...,\vY_{i n_i}^\star$ and $\widehat{\vSigma}^\star:=\bigoplus_{i=1}^a N/n_i\cdot \widehat{\vSigma}_i^\star$, we can formulate the wild bootstrap versions of a quadratic form by

\[\begin{array}{cl}Q_{\vr}^\star &=N[ \vC\overline \vY^\star  ]^\top \vE(\vC,\widehat \vUpsilon^\star)[ \vC\ \overline \vY^\star  ].
\end{array}\]
The following Lemma allows us to use these bootstrap versions to calculate critical values for tests.

\begin{Le}\label{WBTheorem}
If  Assumption (A1) is fulfilled, it holds: {Given the data, the conditional distribution of}\\
\begin{itemize}

\item[(a)]  { $\sqrt{N}\ \overline \vY_i^\star $, for i=1,...a, converges weakly to $ \mathcal{N}_{p_u}\left(\vnull_{p_u},{\kappa_i}^{-1}\vUpsilon_i\right)$ in probability}.% Since we have $\widehat \vUpsilon_i^\star  \stackrel{\mathcal{P}}{\to} \vUpsilon_i$, the unknown covariance matrix $\vUpsilon_i$ can be estimated through $\widehat \vUpsilon_i^\star $.\\
\item[(b)]  {$\sqrt{N}\ \overline \vY^\star $  converges weakly to $ \mathcal{N}_{ap_u}\left(\vnull_{a p_u}, \vUpsilon\right)$ in probability.} %Since we have $\widehat \vUpsilon^\star  \stackrel{\mathcal{P}}{\to}  \vUpsilon$, the unknown covariance matrix $\vUpsilon$ can be estimated through $\widehat \vUpsilon^\star .$
\end{itemize}
{
Since we have $\widehat \vUpsilon_i^\star  \stackrel{\mathcal{P}}{\to} \vUpsilon_i$ and $\widehat \vUpsilon^\star  \stackrel{\mathcal{P}}{\to}  \vUpsilon$ the unknown covariance matrices can be estimated through these estimators.}\\
\begin{proof}
Since it holds $\vY_{ik}^\star\stackrel{\mathcal{D}}{=}\vM(\widehat \vv_i,\widehat \vr_i)\vZ_{ik}^\star$ and therefore $\overline \vY_{i}^\star\stackrel{\mathcal{D}}{=}\vM(\widehat \vv_i,\widehat \vr_i)\overline \vZ_{i}^\star$ we can use part (a) of Theorem 4 from \cite{sattler2022}. With the consistency of $\vM(\widehat \vv_i,\widehat \vr_i)$ from Slutzky's theorem it directly follows 
\[\sqrt{N}\ \overline \vY_{i}^\star\stackrel{\mathcal{D}}{=} \vM(\widehat \vv_i,\widehat \vr_i) \sqrt{N}\ \overline \vZ_{i}^\star\stackrel{\mathcal{D}}{\to} \mathcal{N}_{p_u}(\vnull_{p_u},\kappa_i^{-1}\vM( \vv_i, \vr_i)\vSigma_i \vM( \vv_i, \vr_i)^\top)=\mathcal{N}_{p_u}(\vnull_{p_u},\kappa_i^{-1} \vUpsilon_i).\]

Similar we can show the consistency of  $\widehat\vSigma^\star(\vY)$ the empirical covariance matrix for $\vY_{i1}^\star,...,\vY_{in_i}^\star$. {The empirical covariance matrix $\widehat\vSigma^\star(\vZ)$ for $\vZ_{i1}^\star,...,\vZ_{in_i}^\star$ fulfills $\widehat \vSigma_i^\star(\vZ)\stackrel{\mathcal{P}}{\to}\vSigma_i^\star(\vZ)$ and therefore with Slutzky's theorem it holds}
\[\widehat \vSigma_i^\star(\vY)\stackrel{\mathcal{D}}{=}\vM(\widehat \vv_i,\widehat \vr_i) \widehat \vSigma_i^\star(\vZ) \vM(\widehat \vv_i,\widehat \vr_i)^\top\stackrel{\mathcal{P}}{\to}\vM( \vv_i, \vr_i)\vSigma_i\vM( \vv_i, \vr_i).\]

The results from part (b) follow through the independence of groups.
\end{proof}
\end{Le}
   \section{Proofs from the main paper}

\begin{proof}[Proof of \Cref{CorTheorem1}]\label{proofmain}

This proof is based on the proof from \citet{browne} and \citet{nel1985}, where a similar situation is considered. Some {adaptions} must be done because we are just interested in the matrix's upper triangular.  \\\\

With $\vDelta_i:= \sqrt{n_i}(\widehat \vV_i- \vV_i)$ and  $\vU_i=\vV_{i,0}^{-1/2} \vDelta_i \vV_{i,0}^{-1/2}$, it can be calculated\\

$\begin{array}{ll}\widehat \vR_i&= \widehat \vV_{i,0}^{-1/2}  \widehat\vV_i \widehat \vV_{i,0}^{-1/2}\\&=\left(\vV_{i,0}+\frac 1 {\sqrt{n_i}} \vDelta_{i,0}\right)^{-1/2} \left(\vV_{i}+\frac{1}{\sqrt{n_i}} \vDelta_i\right)\left(\vSigma_0+\frac 1 {\sqrt{n_i}} \vDelta_{i,0}\right)^{-1/2}\\
[1.1ex]&= \left(\vI_p+\frac 1 {\sqrt{n_i}} \vU_{i,0}\right)^{-1/2} \left(\vR_i+\frac{1}{\sqrt{n_i}}\vU_i\right)\left(\vI_p+\frac 1 {\sqrt{n_i}} \vU_{i,0}\right)^{-1/2}.\end{array}$\\\\

The Taylor series of $x\mapsto x^{-1/2}$  in point $1$ leads to $x^{-1/2}=1-(x-1)/2+\frac 3 8 (x-1)^2+\Lan((x-1)^3)$.

For diagonal matrices, we can use this Taylor expansion for each component separately, thereby getting a result as considered in the above equation.

{We first consider the corresponding remainder by using the single components of $1+n_i^{-1/2} \vU_{i,0}$ for $x$}. Since $\vU_{i,0}$ converges to a normally distributed random variable with independent components, from Slutzky's theorem, we know that the remainder is $\lan_P\left(n_i^{-q-1/2}\right)$ for some $q\in(0,1)$.
This leads to

\[\left(\vI_p+\frac 1 {\sqrt{n_i}} \vU_{i,0}\right)^{- 1/ 2}=\left(\vI_p-\frac 1 {2\sqrt{n_i}} \vU_{i,0}\right)+\frac 3 8\left(\frac 1 {\sqrt{n_i}} \vU_{i,0}\right)^2+\lan_P\left(n_i^{-q-1/2}\right)\] and hence\\\\
$\begin{array}{lll}\widehat \vR_i
&=& \left(\vI_p-\frac 1 {2\sqrt{n_i}} \vU_{i,0}\right)\left(\vR_i+\frac{1}{\sqrt{n_i}}\vU_i\right)\left(\vI_p-\frac 1 {2\sqrt{n_i}} \vU_{i,0}\right)
\\[1.3ex]&& +\frac 3 8\left(\frac 1 {\sqrt{n_i}} \vU_{i,0}\right)^2\vR_i+\frac 3 8\vR_i\left(\frac 1 {\sqrt{n_i}} \vU_{i,0}\right)^2

+\lan_P\left(n_i^{-q-1/2}\right)
\\[1.3ex]
&=& \vR_i+\frac{1}{\sqrt{n_i}}\vU_i
-\frac 1 {2\sqrt{n_i}} \vU_{i,0}\vR_i-\frac 1 {2\sqrt{n_i}}\vR_i \vU_{i,0} +\lan_P\left(n_i^{-q-1/2}\right)\\&&
+\frac 1 {4{n_i}} \vU_{i,0}\vR_i \vU_{i,0}
-\frac 1 {2{n_i}} \vU_{i,0}\vU_i
+\frac 3 {8n_i}\left(\vU_{i,0}\right)^2\vR_i+\frac 3 {8n_i}\vR_i\left( \vU_{i,0}\right)^2

-\frac{1}{2{n_i}}\vU_i\vU_{i,0}

,\end{array}$\\\\\\
Here, we used again that $\vU_{i,0}$ and $\vU_{i}$ converge to normal distributed random variables. Again with Slutzky's theorem, the product of such a random variable and an expression like $\lan_P\left(n_i^{-q-1/2}\right)$ has to be also $\lan_P\left(n_i^{-q-1/2}\right)$. Therefore many parts of the initial product are now cumulated in  $\lan_P\left(n_i^{-q-1/2}\right)$. For the 
derivation of the Taylor-base Monte-Carlo approach, we will use the above expression, but for the asymptotic normality, we use Slutzky's theorem and get \\\\
$\widehat \vR_i
= \left(\vI_p-\frac 1 {2\sqrt{n_i}} \vU_{i,0}\right)\left(\vR_i+\frac{1}{\sqrt{n_i}}\vU_i\right)\left(\vI_p-\frac 1 {2\sqrt{n_i}} \vU_{i,0}\right)+\lan_P\left(n_i^{-q-1/2}\right)
$.\\\\

Multiplication with $\sqrt{n_i}$ leads to\\

$
\sqrt{n_i}\ \widehat \vR_i=\sqrt{n_i}\ \vR_i+\left[\vU_i-\frac 1 2\left(\vU_{i,0 }\vR_i +\vR_i\vU_{i,0}\right)\right] 
+ \lan_P\left(n_i^{-q}\right).
$\\\\

We define
\[\vM_1=\sum\nolimits_{\ell=1}^{p_u} \ve_{\ell,p_u} (\ve_{{\vh_1}_{\ell},p}+\ve_{{\vh_2}_\ell,p})^\top\] and

  \[\vM_2:= \sum\nolimits_{\ell=1}^p \ve_{\ell,p} \ \ve_{{\vh_4}_\ell,p}^\top\quad\vM_3:=\sum\nolimits_{\ell=1}^p \ve_{\ell,p}\ \ve_{{\vh_3}_\ell,p}^\top\]\[ \vM_4:=\vM_2+\vM_3\quad \vM_5:=\diag(\vech(\vI_d)), \]
  
  with  $\vh_3=\vech(\vH)$ and $\vh_4=\vech(\vH^\top)$. With these matrices, it is easy to check that the following equations 
\[\begin{array}{ll}\vech(\vU_{i,0} \vR_i)&=\diag(\vech(\vR_i))  \vM_2 \vech(\vU_{i,0}),\\ \vech(\vR_i\vU_{i,0})&=\diag(\vech(\vR_i))\cdot \vM_3\Cdot \vech(\vU_{i,0})\end{array}\]
and \[\vech(\vU_{i,0} \vR_i) +\vech(\vR_i\vU_{i,0})=\vech(\vR_i)) \left(  \vM_2 + \vM_3\right)  \]

hold, and therefore with $\vech(\vU_{i,0})=\vM_5\vech(\vU_{i})$ we get\\\\
$\begin{array}{ll}
&\sqrt{n_i}\vech(\widehat \vR_i-\vR_i)\\
=&\left[\vech(\vU_i)-\frac 1 2\left(\vech(\vU_{i,0} \vR_i) +\vech(\vR_i\vU_{i,0})\right)\right]+\lan_P\left(n_i^{-q}\right)\\
%=&\left[\vech(\vU_i)-\frac 1 2\left( \diag(\vech(\vR_i)) \vM_2 \ \vech(\vU_{i,0}) +\diag(\vech(\vR_i)) \cdot  \vM_3 \ \vech(\vU_{i,0})  \right)\right]+\lan_P\left(n_i^{-q}\right)\\
=&\left[\vech(\vU_i)-\frac 1 2\diag(\vech(\vR_i)) \left(  \vM_2 + \vM_3\right)  \vech(\vU_{i,0})  \right]+\lan_P\left(n_i^{-q}\right)\\
=&\left[\vech(\vU_i)-\frac 1 2 \diag(\vech(\vR_i)) \left(  \vM_2  + \vM_3\right) \vM_5   \vech(\vU_i)  \right]+\lan_P\left(n_i^{-q}\right)\\
=&\left[\vI_{p}-\frac 1 2 \diag(\vech(\vR_i))  \left( \vM_2 + \vM_3\right)\vM_5  \right]\ \vech(\vU_i) +\lan_P\left(n_i^{-q}\right).\\
\end{array}$\\\\

Multiplication with the matrix $\vM_5$ changes nothing in this case because it just picks the columns unequal to zero and drops the rest.
So all in all  with $\vM_4=\vM_2+\vM_3$  it holds\\

$ \begin{array}{l}
\sqrt{n_i}\vech(\widehat \vR_i-\vR_i)
=\left[\vI_{p}-\frac 1 2 \diag(\vech(\vR_i))  \vM_4   \right] \vech(\vU_i) +\lan_P\left(n_i^{-q}\right)
\end{array}$\\

Now to adopt this result for the upper-half-vectorization, we use the particular elimination matrix $\vL_p^u$ which gives a connection between $\vech$ and $\vech^-$

$ \begin{array}{ll}&\sqrt{n_i}(\widehat{ \vr}_i-\vr_i)\\
=&\vL_p^u\left[\vI_{p}-\frac 1 2 \diag(\vech(\vR_i))  \vM_4   \right] \vech(\vU_i) +\lan_P\left(n_i^{-q}\right)\\[1.1ex]
=&\left[\vL_p^u-\frac 1 2 \diag(\vr_i)\vL_p^u  \vM_4   \right] \vech(\vU_i) +\lan_P\left(n_i^{-q}\right)\\[1.1ex]
=&\left[\vL_p^u-\frac 1 2 \diag(\vr_i) \vM_1   \right] \vech(\vU_i) +\lan_P\left(n_i^{-q}\right).

\end{array}$\\\\\\
Here, we used the relation $\vL_p^u  \vM_4=\vM_1$ and because of 
\[ \vech(\vU_i)= \diag(\vech((v_{i11},...,v_{i dd})^\top (v_{i11},...,v_{i dd})))^{-\frac 1 2}\vech(\vDelta_i)\] it is useful to define
\[\vM(\vv_i,\vr_i)=\left[\vL_p^u-\frac 1 2\diag(\vr_i)\vM_1   \right]\diag(\vech((v_{i11},...,v_{i dd})^\top (v_{i11},...,v_{i dd})))^{-\frac 1 2}.\]\\
Therefore, it holds
\[\sqrt{n_i}(\widehat{ \vr}_i-\vr_i)=
{\vM(\vv_i,\vr_i)}\vech(\vDelta_i) +\lan_P(1)\]
\\\
 and because of Theorem 1 from \cite{sattler2022} it follows
\[\sqrt{n_i}(\widehat \vr_i-\vr_i)\stackrel{\mathcal{D}}{\To}\mathcal{N}_{p_u}\Big(\vnull_{p_u},\underbrace{\vM(\vv_i,\vr_i)\vSigma_i\vM(\vv_i,\vr_i)^\top}_{=:\vUpsilon_i}\Big).\]

We could get the same result by using the {delta} method on the results for the vectorized covariance matrices. {We believe} the approach of \citet{browne}, together with \citet{nel1985}, is preferable due to its stepwise structure. Therefore it is more suitable to get an understanding of the used matrices. Another important argument is that we later have a Monte-Carlo approximation based on this Taylor approximation.
\end{proof}

\begin{proof}[Proof of \Cref{Verteilung}]
With the result from \Cref{CorTheorem1}, the asymptotic distribution of the quadratic form follows exactly from Theorem 2 from \cite{sattler2022}.
\end{proof}

\begin{proof}[Proof of \Cref{PBTheorem1}]
We only prove the first part because the second part directly follows the single groups' results.\\
For an application of the multivariate Lindeberg-Feller-Theorem (given the data), we need to check the conditions. As $\vY_{ik}^\dagger$ under $\vX$ is $p_u$-dimensional normal distributed with  expectation $\vnull_{p_u}$ and variance  $\widehat{\vUpsilon}_{i}$ it holds:\\

$\begin{array}{lll}1.)\hspace*{-0.3cm}&\textcolor{white}{=} &\sum\limits_{k=1}^{n_i}\E\left(\frac {\sqrt{N}} {n_i}\vY_{i k}^\dagger\Big\lvert \vX\right)=\sum\limits_{k=1}^{n_i}\frac {\sqrt{N}} {n_i}\cdot \E\left(\vY_{i k}^\dagger\Big\lvert \vX\right)=\vnull.\end{array}$\\\\
$\begin{array}{lll}2.)\hspace*{-0.3cm}&\textcolor{white}{=} &\sum\limits_{k=1}^{n_i}\Cov\left(\frac {\sqrt{N}} {n_i}\vY_{i k}^\dagger\Big\lvert \vX\right)= \sum\limits_{k=1}^{n_i}\frac {{N}} {n_i^2}  \widehat\vUpsilon_{i}\stackrel{\mathcal P}{\To} \frac 1 {\kappa_i} \vUpsilon_{i}.\end{array}$\\

$\begin{array}{lll}3).\hspace*{-0.3cm}&& \lim\limits_{N\to \infty}\sum\limits_{k=1}^{n_i} \E\left( \Big \lvert \Big \lvert \frac {\sqrt{N}} {n_i}\vY_{i k}^\dagger  \Big \lvert \Big \lvert^2\cdot \ind_{ \big \lvert \big \lvert \frac {\sqrt{N}} {n_i}\vY_{i k}^\dagger  \big \lvert \big \lvert>\delta}\ \Big\lvert \vX\right)\\[1.4ex]

&=& \lim\limits_{N\to \infty}  \frac {{N}} {n_i^2}\sum\limits_{k=1}^{n_i} \E\left( \big  \lvert \big  \lvert \vY_{i 1}^\dagger  \big  \lvert \big  \lvert^2\cdot \ind_{ \lvert \lvert \vY_{i 1}^\dagger   \lvert  \lvert>\delta \frac {n_i}{\sqrt{N}} }\ \Big\lvert \vX\right)\\[1.6ex]
&=& \frac 1 {\kappa_i}  \cdot \lim\limits_{N\to \infty}  \E\left( \big  \lvert \big  \lvert \vY_{i 1}^\dagger\big \lvert \big \lvert^2\cdot \ind_{  \lvert \lvert \vY_{i 1}^\dagger\lvert \lvert>\delta\frac {n_i}{\sqrt{N}} }\ \Big\lvert \vX\right)\\[1.6ex]
&\leq &\frac 1 {\kappa_i}  \cdot \lim\limits_{N\to \infty} \sqrt{ \E\left(  \lvert   \lvert \vY_{i 1}^\dagger \lvert \lvert^4 \ \lvert \vX\right)}\cdot \sqrt{\E\left(\ind_{  \lvert  \lvert \vY_{i 1}^\dagger \lvert  \lvert>\delta\frac {n_i}{\sqrt{N}} }\ \Big\lvert \vX\right)}=0. 

\end{array}$\\\\\\
For the last part, we used the Cauchy-Bunjakowski-Schwarz-Inequality and that we know $\E\left( \big  \lvert \big  \lvert \vY_{i 1}^\dagger\big \lvert \big \lvert^4\ \big\lvert \vX\right)<\infty$. Finally through  $ {n_i}/ N \to \kappa_i$ and therefore $\delta\cdot  {n_i}/{\sqrt{N}}\to \infty $, it holds $P\left( \big \lvert \big \lvert \vY_{i 1}^\dagger \big \lvert \big \lvert>\delta \cdot {n_i}/\sqrt{N} \right)\to 0$, which leads to the result.\\

Given the data $\vX$, it follows that ${\sqrt{N}}\ \overline \vY_{i }^\dagger$  converges in distribution to $\vLambda_i\sim\mathcal{N}_{p_u}\left(\vnull_{p_u}, 1/ {\kappa_i}\cdot  \vUpsilon_i\right)$ and therefore  because of independence of groups $ {\sqrt{N}}\ \overline \vY^\dagger$ converges in distribution to $\vLambda\sim\mathcal{N}_{a p_u}\left(\vnull_{a p_u},\vUpsilon\right)$. \\

We can use the results from \cite{sattler2022} for the consistency of the empirical covariance matrix. Through the construction of the bootstrap sample it holds $\vY_{ik}^\dagger
\stackrel{\mathcal{D}}{=} \vM(\widehat \vv_i,\widehat \vr_i) \vZ_{ik}^\dagger$, while $\vZ_{ik}^\dagger\sim \mathcal N_{p}(\vnull_p,\widehat \vSigma_i)$. Therefore with the consistency of $\widehat \vSigma_i^\dagger(\vZ)$, the empirical covariance matrix for $\vZ_{i1}^\dagger,...,\vZ_{in_i}^\dagger$ it holds

\[\widehat \vSigma_i^\dagger(\vY)\stackrel{\mathcal{D}}{=}\vM(\widehat \vv_i,\widehat \vr_i) \widehat \vSigma_i^\dagger(\vZ) \vM(\widehat \vv_i,\widehat \vr_i)^\top\stackrel{\mathcal{P}}{\to}\vM( \vv_i, \vr_i)\vSigma_i\vM( \vv_i, \vr_i)^\top.\]

Again, with the independence of groups, the results from part (b) follow directly.
\end{proof}  

\begin{Le}\label{TayLe}
Let be
\[\vM_6=\sum\nolimits_{\ell=1}^{d} \ve_{\ell,d}\cdot \ve_{a_\ell,p}^\top
\]
 and
 
\[\vLambda_i(\widehat\vv_i):=\diag(\vech((\widehat v_{i11},...,\widehat v_{i dd})^\top (\widehat v_{i11},...,\widehat v_{i dd})))^{-\frac 1 2}.\]
Then for $\vY_i\sim\mathcal{N}_{p}\left(\vnull_{p},\vSigma_i\right)$ it holds \[\sqrt{n_i}(\widehat \vr_i-\vr_i)=\vM(\widehat\vv_i,\widehat\vr_i) \vY_i+ f_{\widehat\vv_i,\widehat \vR_i}(\vY_i)/\sqrt{n_i}+\lan_P\left(\sqrt{{n_i^{-1}}}\right) \]
with
\[\begin{array}{lll}f_{\widehat\vv_i,\widehat \vR_i}(\vx)&=&\frac{1}{4}\vL_p^u\diag(\vech(\vM_6 \vLambda_i(\widehat\vv_i)\vx \vx^\top\vLambda_i(\widehat\vv_i)^\top \vM_6^\top))  \vech(\widehat\vR_i)\\
   %      &&+\frac{1}{4{n_i}}\vL_p^u\diag(\vech(\vM_6 \vech(\vU_{i})\vech(\vU_{i})^\top \vM_6^\top))   \vech(\vU_i) \\
&&- \frac{1}{2} \vL_p^u\diag(\vLambda_i(\widehat\vv_i) \vx) \vM_4  \vM_5 \vLambda_i(\widehat\vv_i) \vx    \\
          &&+ \frac{3}{8} \diag(\widehat \vr_i) \vM_1   \vech((\vM_6 \vLambda_i(\widehat\vv_i)\vx)(\vM_6\vLambda_i(\widehat\vv_i)\vx)^\top).\end{array}\]
\end{Le}
\begin{proof}
From the proof of \Cref{PBTheorem1} we know\\\\
$\begin{array}{lll}\widehat \vR_i

&=& \vR_i+\frac{1}{\sqrt{n_i}}\vU_i
-\frac 1 {2\sqrt{n_i}} \vU_{i,0}\vR_i-\frac 1 {2\sqrt{n_i}}\vR_i \vU_{i,0}
+\lan_P\left(n_i^{-q-0.5}\right)
\\&&
+\frac 1 {4{n_i}} \vU_{i,0}\vR_i \vU_{i,0}
-\frac 1 {2{n_i}} \vU_{i,0}\vU_i
+\frac 3 8\left(\frac 1 {\sqrt{n_i}} \vU_{i,0}\right)^2\vR_i+\frac 3 8\vR_i\left(\frac 1 {\sqrt{n_i}} \vU_{i,0}\right)^2 
-\frac{1}{2{n_i}}\vU_i\vU_{i,0}

.\end{array}$\\

Again the aim is to express the vectorization of these matrices by using $\vech(\vR_i)$ and $\vech(\vU_i)$. 
Analogous to the proof of \Cref{PBTheorem1} it follows directly 
\[\begin{array}{l} \vech(\vU_{i,0}\vU_i)=\diag(\vech(\vU_i))\cdot \vM_2\cdot \vech(\vU_{i,0}),\\
\vech(\vU_i\vU_{i,0})=\diag(\vech(\vU_i))\cdot \vM_3\cdot \vech(\vU_{i,0}),\\
\vech((\vU_{i,0})^2\vR_i)=\diag(\vech(\vR_i))\cdot \vM_2\cdot \vech((\vU_{i,0})^2),\\
\vech(\vR_i(\vU_{i,0})^2)=\diag(\vech(\vR_i))\cdot \vM_3\cdot \vech((\vU_{i,0})^2).

\end{array}
\]
Moreover, with  $\vM_6$ it is easy to check \[\vech(\vU_{i,0}\vR_i \vU_{i,0})=\diag(\vech(\vM_6 \vech(\vU_{i})\cdot\vech(\vU_{i})^\top \vM_6^\top))  \cdot \vech(\vR_i) \]
as well as 
\[\vech((\vU_{i,0})^2)= \vech((\vM_6 \vech(\vU_{i}))(\vM_6 \vech(\vU_{i}))^\top).\]

So together this leads to\\
$\begin{array}{lll}&&\sqrt{n_i}(\widehat \vr_i-\vr_i)\\                               
  &=&\vM(\vv_i,\vr_i)\vech(\vDelta_i) +\lan_P\left(n_i^{-q}\right)  \\
         &&+\frac{1}{4\sqrt{n_i}}\vL_p^u\diag(\vech(\vM_6 \vech(\vU_{i})\vech(\vU_{i})^\top \vM_6^\top))   \vech(\vR_i)\\
   %      &&+\frac{1}{4{n_i}}\vL_p^u\diag(\vech(\vM_6 \vech(\vU_{i})\vech(\vU_{i})^\top \vM_6^\top))   \vech(\vU_i) \\
&&- \frac{1}{2\sqrt{n_i}} \vL_p^u\diag(\vech(\vU_i)) (\vM_2+\vM_3)  \vM_5 \vech(\vU_{i})     \\
          &&+ \frac{3}{8\sqrt{n_i}} \vL_p^u\diag(\vech(\vR_i)) (\vM_2+\vM_3)   \vech((\vM_6 \vech(\vU_{i}))(\vM_6 \vech(\vU_{i}))^\top)\\
%  &=&\vM(\vv_i,\vr_i)\vech(\vDelta_i) +\lan_P\left(n_i^{-q}\right)  \\
%         &&+\frac{1}{4\sqrt{n_i}}\vL_p^u\diag(\vech(\vM_6 \vech(\vU_{i})\vech(\vU_{i})^\top \vM_6^\top))   \vech(\widehat\vR_i)\\
%   %      &&+\frac{1}{4{n_i}}\vL_p^u\diag(\vech(\vM_6 \vech(\vU_{i})\vech(\vU_{i})^\top \vM_6^\top))   \vech(\vU_i) \\
%&&- \frac{1}{2\sqrt{n_i}} \vL_p^u\diag(\vech(\vU_i)) \vM_4  \vM_5 \vech(\vU_{i})     \\
%          &&+ \frac{3}{8\sqrt{n_i}} \vL_p^u\diag(\vech(\widehat \vR_i)) \vM_4   \vech((\vM_6 \vech(\vU_{i}))(\vM_6 \vech(\vU_{i}))^\top)\\
            &=&\vM(\vv_i,\vr_i)\vech(\vDelta_i) +\lan_P\left(n_i^{-q}\right)  \\
         &&+\frac{1}{4\sqrt{n_i}}\vL_p^u\diag(\vech(\vM_6 \vech(\vU_{i})\vech(\vU_{i})^\top \vM_6^\top))   \vech(\widehat\vR_i)\\
   %      &&+\frac{1}{4{n_i}}\vL_p^u\diag(\vech(\vM_6 \vech(\vU_{i})\vech(\vU_{i})^\top \vM_6^\top))   \vech(\vU_i) \\
&&- \frac{1}{2\sqrt{n_i}} \vL_p^u\diag(\vech(\vU_i)) \vM_4  \vM_5 \vech(\vU_{i})     \\
          &&+ \frac{3}{8\sqrt{n_i}} \diag(\widehat \vr_i) \vM_1   \vech((\vM_6 \vech(\vU_{i}))(\vM_6 \vech(\vU_{i}))^\top),

\end{array}$\\
where we used the consistency of $\widehat \vR$. With $\vLambda_i(\widehat\vv_i)$ and the results from \cite{sattler2022}
it holds
\[ \vech(\vU_i)= \vLambda_i(\widehat\vv_i)\vY_i+\lan_P\left(1\right)\] while $\vY_i\sim\mathcal{N}_{p}\left(\vnull_{p},\vSigma_i\right)$.
From this reason we define,

\[\begin{array}{lll}f_{\widehat\vv_i,\widehat \vR_i}(\vx)&=&\frac{1}{4}\vL_p^u\diag(\vech(\vM_6 \vLambda_i(\widehat\vv_i)\vx\vx^\top\vLambda_i(\widehat\vv_i)^\top \vM_6^\top))   \vech(\widehat\vR_i)\\
   %      &&+\frac{1}{4{n_i}}\vL_p^u\diag(\vech(\vM_6 \vech(\vU_{i})\vech(\vU_{i})^\top \vM_6^\top))   \vech(\vU_i) \\
&&- \frac{1}{2} \vL_p^u\diag(\vLambda_i(\widehat\vv_i) \vx) \vM_4  \vM_5 \vLambda_i(\widehat\vv_i) \vx    \\
          &&+ \frac{3}{8} \diag(\widehat \vr_i) \vM_1   \vech((\vM_6 \vLambda_i(\widehat\vv_i)\vx)(\vM_6\vLambda_i(\widehat\vv_i)\vx)^\top).\end{array}\]
\end{proof}
From a theoretical point of view, $f_{\widehat\vv_i,\widehat \vR_i}$ is no function but a family of functions since it strongly depends on $\widehat\vv_i$ and $\widehat \vR_i$.\\\\
Through the construction of $\vY^{Tay}$ it is clear that $\sqrt{N}\vY^{Tay}\stackrel{\mathcal{D}}{\To} \vZ^{Tay}\sim\mathcal{N}_{a p_u}(\vnull_{ap_u},\vUpsilon)$ and from \Cref{Verteilung} we know $\sqrt{N}(\widehat \vr- \vr)\stackrel{\mathcal{D}}{\To} \vZ\sim\mathcal{N}_{a p_u}(\vnull_{ap_u},\vUpsilon)$.
Therefore, with $\vE(\vC,\widehat\vUpsilon)\stackrel{\mathcal{D}}{\To}\vE(\vC,\vUpsilon)$ it follows that our test based on this approach is asymptotic correct.

\section{Combined tests for covariance and correlations}

To investigate the asymptotic distribution of $\vT$  we use the matrix $\vA$ fullfilling

\[\vA\widehat \vv_i=\begin{pmatrix}
\widehat V_{i11}\\
\vdots\\
\widehat V_{idd}
\end{pmatrix}.\]
Together with
\[\sqrt{n_i}(\widehat \vr_i-\vr_i)=\vM(\vv_i,\vr_i)\sqrt{n_i} (\widehat \vv_i-\vv_i)+\lan_P(1)\]
this lead to
\[\sqrt{N}\left(
\begin{pmatrix}
\widehat V_{i 11}\\
\widehat V_{i 22}\\
\vdots\\
\widehat V_{i dd}\\
\widehat \vr_i
\end{pmatrix}-
\begin{pmatrix}
V_{i 11}\\
V_{i 22}\\
\vdots\\
V_{i dd}\\
\vr_i
\end{pmatrix}\right)\stackrel{\mathcal{D}}{\To} \mathcal{N}_p\left(\vnull_p,\kappa_i^{-1} \begin{pmatrix}
\vA\\
\vM(\vv_i,\vr_i)
\end{pmatrix}\vSigma_i\begin{pmatrix}
\vA\\
\vM(\vv_i,\vr_i)
\end{pmatrix}^\top\right)\]
and therefore, under the hypothesis of equal covariance matrices
\[\sqrt{N}\left(
\begin{pmatrix}
\widehat V_{1 11}\\
\widehat V_{1 22}\\
\vdots\\
\widehat V_{1 dd}\\
\widehat \vr_1
\end{pmatrix}-
\begin{pmatrix}
\widehat V_{2 11}\\
\widehat V_{2 22}\\
\vdots\\
\widehat V_{2 dd}\\
\widehat \vr_2
\end{pmatrix}\right)\stackrel{\mathcal{D}}{\To} \mathcal{N}_p\left(\vnull_p,\underbrace{\sum\limits_{i=1}^2\kappa_i^{-1} \begin{pmatrix}
\vA\\
\vM(\vv_i,\vr_i)
\end{pmatrix}\vSigma_i\begin{pmatrix}
\vA\\
\vM(\vv_i,\vr_i)
\end{pmatrix}^\top}_{\vGamma}\right).\]\\

But with \Cref{TayLe}, it also holds

\[\sqrt{n_i}\left(
\begin{pmatrix}
\widehat V_{i 11}\\
\widehat V_{i 22}\\
\vdots\\
\widehat V_{i dd}\\
\widehat \vr_i
\end{pmatrix}-
\begin{pmatrix}
V_{i 11}\\
V_{i 22}\\
\vdots\\
V_{i dd}\\
\vr_i
\end{pmatrix}\right)= \begin{pmatrix}
\vA\\
\vM(\vv_i,\vr_i)
\end{pmatrix} \vY_i+ \frac 1 {\sqrt{n_i}}\begin{pmatrix}\vnull_d\\
f_{\widehat\vv_i,\widehat \vR_i}(\vY_i)
\end{pmatrix}+\lan_P\left(\sqrt{n_i^{-1}}\right)\]
with 
$\vY_i\sim\mathcal{N}_{p}(\vnull_p,\vSigma_i)$. Therefore, for the transformed vectors $\vY_i^{Tay}$ it holds  \[\vY_i^{Tay}\stackrel{\mathcal{D}}{\To} \mathcal{N}_p\left(\vnull_p, \begin{pmatrix}
\vA\\
\vM(\vv_i,\vr_i)
\end{pmatrix}\vSigma_i\begin{pmatrix}
\vA\\
\vM(\vv_i,\vr_i)
\end{pmatrix}^\top\right)\]
and because of independence and Slutzky's theorem
$\vT^{Tay}_1=\sqrt{N}(n_1^{-1/2}\vY_1^{Tay}-n_{2}^{-1/2}\vY_2^{Tay})\stackrel{\mathcal{D}}\To \mathcal{N}_p(\vnull_p, \vGamma)$. So the asymptotic distribution of the $\vT^{Tay}_b$ coincides with the asymptotic distribution of $\vT$ under the null hypothesis, which leads to an asymptotic correct test.

\section*{Further Simulations} 
\subsection*{Type-I-error rate}
Here all hypotheses are investigated with more test statistics, like a wild bootstrap ATS. For hypothesis  $A_{\vr}$ in addition to a Toeplitz covariance matrix, also an autoregressive matrix $\vV_{ij}=0.6^{\lvert i-j\lvert}$ is simulated to see the influence of the chosen covariance matrix. 

Moreover, we investigated hypothesis  $A_{\vr}$ for the special case of dimension $d=2$ and therefore $p_u=1$. We also want to compare with the permutation-based approach from \cite{omPa:2012}, with 999 permutations. Since this procedure was specially developed for this dimension, while we allow higher dimensions, this is particularly interesting. We also adapted the sample sizes for this smaller dimension to have the same relation between $N$ and $d$ and get $\vN=(20,40,100,200)$. {To see the influence of the dimension $d$, we also consider the case $d=7$ and therefore $p_u=21$. Here the sample sizes are  $\vN=(35,70,175,350)$ for one group and $\vN=(70,140,350,700)$ for two groups, which is the same relation as earlier, but substantially smaller regarding the number of unknown parameters $p_u$.}\\ 
Since the diagonality of a covariance matrix can be checked by using the covariance matrix or the correlation matrix, for $B_{\vr}$ we used the covariance-based ATS approaches from \cite{sattler2022} {for  $d=5$}.

We also added one more hypothesis,  $C_{\vr}$ $\mathcal{H}_0^{\vr}:\vr_{1}=...=\vr_{p_u}$ to examine whether all correlation components in a group are the same. This hypothesis corresponds with a compound symmetry structure of the correlation matrix. Finally, we consider the gamma distribution as further distribution for all settings. 

    \begin{table}[htbp]
\centering
\begin{footnotesize}
\begin{tabular}{x{1.3pt}lx{1.3pt}c|c|c|cx{1.3pt}c|c|c|cx{1.3pt}}\specialrule{1.3pt}{0pt}{0pt}
   \multicolumn{1}{x{1.3pt}lx{1.3pt}}{} &\multicolumn{4}{ cx{1.3pt}}{$t_9$}&\multicolumn{4}{cx{1.3pt}}{Normal}\\
   %\specialrule{1.3pt}{0pt}{0pt}
\specialrule{1.3pt}{0pt}{0pt}
    \hspace*{.1cm}N&   50&100&250&500&   50&100&250&500
       \\ \specialrule{1.3pt}{0pt}{0pt}
 ATS-Par &  .0755 & .0657 & .0568 & \bf{.0538} & .0703 & .0627 &\bf{ .0524} & \bf{.0519} \\   \hline
 ATS-Wild& .0832 & .0656 & .0593 & .0550 & .0763 & .0661 & \bf{.0536} & \bf{.0526} \\    \hline
 ATS-Par-m & .0653 & .0600 & .0553 & \bf{.0528} & .0609 & .0579 &\bf{ .0502} & \bf{.0511} \\  \hline
 ATS &  .0818 & .0652 & .0581 & \bf{.0542} & .0776 & .0640 & \bf{.0527} & \bf{.0511} \\   \hline
ATS-Tay &.0603 & \bf{.0528} & \bf{.0523} & \bf{.0511} & .0567 & \bf{.0536} & \bf{.0482} &\bf{ .0495} \\ \hline

    ATSFz & .0781 & .0634 & .0568 & \bf{.0540} & .0726 & .0628 &\bf{ .0516} & \bf{.0508} \\ \hline
 ATSFz-m& .0683 & .0592 & .0556 & \bf{.0530} & .0623 & .0582 & \bf{.0503} & \bf{.0501} \\   \hline
  %WTS &.2186 & .1126 & .0710 & .0594 & .1648 & .0770 & .0600 & .0553 \\   \hline
   Steiger & $<$.001&$<$.001 &$<$.001&$<$.001 &$<$.001 &$<$.001 &$<$.001&$<$.001 \\ \hline
   SteigerFz & .0931 & .1106 & .1303 & .1296 & .0908 & .1130 & .1249 & .1229 \\   \hline
 Jennrich & .3683 & .4214 & .4508 & .4621 & .3657 & .4045 & .4249 & .4336  \\  
 \specialrule{1.3pt}{0pt}{0pt}\multicolumn{9}{x{1.3pt}cx{1.3pt}}{}\\\specialrule{1.3pt}{0pt}{0pt}
   &\multicolumn{4}{cx{1.3pt}}{Skew Normal} &\multicolumn{4}{cx{1.3pt}}{Gamma}\\\specialrule{1.3pt}{0pt}{0pt}

%\specialrule{1.3pt}{0pt}{0pt}
    \hspace*{.1cm}N&   50&100&250&500&   50&100&250&500
       \\ \specialrule{1.3pt}{0pt}{0pt}
 ATS-Par & .0734 & .0635 & \bf{.0530} & .0567 & .0861 & .0707 & .0605 & .0575 \\  \hline
 ATS-Wild & .0803 & .0671 & .0551 & .0578 & .0912 & .0732 & .0628 & .0566 \\  \hline
 ATS-Par-m & .0626 & .0602 & \bf{.0518} & .0557 & .0753 & .0670 & .0586 & .0565 \\  \hline
 ATS & .0798 & .0666 & \bf{.0536} & .0567 & .0915 & .0723 & .0613 & .0565 \\ \hline
ATS-Tay &  .0579 & \bf{.0536} & \bf{.0486} & \bf{.0539} & .0653 & .0591 & .0547 & \bf{.0528} \\ \hline

    ATSFz & .0757 & .0659 & \bf{.0528} & .0562 & .0883 & .0713 & .0607 & .0559 \\ \hline
 ATSFz-m & .0642 & .0604 &\bf{ .0512} & .0554 & .0766 & .0666 & .0588 & .0546 \\  \hline
%  WTS & .1948 & .1003 & .0699 & .0613 & .2892 & .1775 & .1122 & .0855 \\  \hline
   Steiger & $<$.001&$<$.001 &$<$.001&$<$.001 &$<$.001 &$<$.001 &$<$.001&$<$.001 \\ \hline
   SteigerFz & .0902 & .1158 & .1177 & .1272 & .0952 & .1212 & .1311 & .1372 \\  \hline
 Jennrich & .3738 & .4133 & .4484 & .4522 & .4006 & .4576 & .4995 & .5006 \\  
\specialrule{1.3pt}{0pt}{0pt}
\end{tabular}
\end{footnotesize}

\caption{Simulated type-I-error rates ($\alpha=5\%$) in scenario $A_{\vr}$ ($\mathcal{H}_0^{\vr}:\vR_{1}=\vR_{2}$)  for ATS, Steiger's and Jennrich's test. The observation vectors have dimension 5, covariance matrix $(\vV_1)_{ij}=1-{\lvert i-j\lvert/2d}$ resp. $\vV_2=\diag(1,1.2,...,1.8)\vV_1$ and it always holds $n_1:=0.6\cdot N$ resp. $n_2:=0.4\cdot N$.}\label{tab:SimA1CorS}\end{table}

The wild bootstrap approach seems to be too liberal in all settings, particularly in comparison with the parametric bootstrap or the Monte-Carlo approach. So this bootstrap technique is not recommendable for testing hypotheses regarding correlation matrices.

In \Cref{tab:SimA1CorS} and \Cref{tab:SimA2CorS}, it seems that the test performance partially depends on the underlying covariance structure. The results of our tests for the autoregressive covariance matrix are more liberal than for the Toeplitz covariance matrix.
{However, the influence is remarkable for the test from \cite{jennrich1970}, which has for a Toeplitz covariance matrix type-I error rates all the time higher than $0.36$.} It is {better} for the autoregressive one but worse than most of our tests.
Also, the test of Steiger without the Fisher transformation performs clearly better for the autoregressive covariance matrix, although large sample sizes are required to fulfil Bradley's liberal criterion.

{ Our test results} for dimension 2 from \Cref{tab:SimA3CorS} and \Cref{tab:SimA4CorS} are slightly more liberal than for dimension 5. For the Toeplitz matrix, Steiger has disastrous error rates and is not usable in our opinion, while SteigerFz seems to perform clearly better for small dimensions.
It is noticeable that while for dimension 5, ATSFz was favourable to the parametric bootstrap most of the time, it is now vice versa. Finally, for dimension 2 the type-I error rate of the permutation test from \cite{omPa:2012} as well as Jennrich and SteigerFz seem to depend on the covariance matrix and the distribution. Although ATS-Par, ATS-Par-m, ATS-Tay and ATSFz-m here often need 100 or more observations to fulfil Bradley's liberal criterion, they seem less dependent on the setting. Our approach works better for dimensions larger than 2, which was the main focus. Moreover, these tables show that the required large number of observations for correlation tests (see, for example, \cite{steiger1980}) are not in relation to the number of unknown parameters $p_u$.

\begin{table}[htbp]
\centering
\begin{footnotesize}
\begin{tabular}{x{1.3pt}lx{1.3pt}c|c|c|cx{1.3pt}c|c|c|cx{1.3pt}}\specialrule{1.3pt}{0pt}{0pt}
   \multicolumn{1}{x{1.3pt}lx{1.3pt}}{} &\multicolumn{4}{ cx{1.3pt}}{$t_9$}&\multicolumn{4}{cx{1.3pt}}{Normal}\\
   %\specialrule{1.3pt}{0pt}{0pt}
\specialrule{1.3pt}{0pt}{0pt}
    \hspace*{.1cm}N&   50&100&250&500&   50&100&250&500
       \\ \specialrule{1.3pt}{0pt}{0pt}
 ATS-Par & .1142 & .0778 & .0614 & .0558 & .1030 & .0770 & .0579 & .0551 \\  \hline
 ATS-Wild & .1365 & .0889 & .0656 & .0586 & .1283 & .0844 & .0616 & .0569 \\   \hline
 ATS-Par-m & .0977 & .0721 & .0592 & .0545 & .0879 & .0696 & .0560 & \bf{.0539} \\    \hline
 ATS & .1215 & .0788 & .0615 & .0552 & .1101 & .0777 & .0585 & .0550 \\  \hline
ATS-Tay & .0829 & .0636 & .0554 &\bf{ .0511} & .0723 & .0606 & \bf{.0512} & \bf{.0530} \\   \hline

    ATSFz & .1137 & .0757 & .0604 & .0546 & .1029 & .0744 & .0576 & .0547 \\  \hline
 ATSFz-m & .0960 & .0688 & .0582 & \bf{.0532} & .0870 & .0678 & .0544 & \bf{.0540} \\   \hline
%  WTS & .4688 & .2456 & .1162 & .0815 & .4175 & .2009 & .1032 & .0751 \\  \hline
   Steiger & .0103 & .0195 & .0232 & .0254 & .0123 & .0204 & .0244 & .0259 \\  \hline
   SteigerFz &.0564 & .0790 & .0922 & .0946 & .0601 & .0764 & .0899 & .0921 \\   \hline
  Jennrich & .3355 & .2869 & .2571 & .2531 & .3344 & .2817 & .2559 & .2456 \\    
 \specialrule{1.3pt}{0pt}{0pt}\multicolumn{9}{x{1.3pt}cx{1.3pt}}{}\\\specialrule{1.3pt}{0pt}{0pt}
   &\multicolumn{4}{cx{1.3pt}}{Skew Normal} &\multicolumn{4}{cx{1.3pt}}{Gamma}\\\specialrule{1.3pt}{0pt}{0pt}

    \hspace*{.1cm}N&   50&100&250&500&   50&100&250&500
       \\ \specialrule{1.3pt}{0pt}{0pt}
 ATS-Par & .1079 & .0812 & .0575 & .0558 & .1206 & .0895 & .0667 & .0606 \\\hline
 ATS-Wild & .1320 & .0896 & .0619 & .0581 & .1451 & .1031 & .0727 & .0619 \\    \hline
 ATS-Par-m &.0919 & .0739 & .0545 & .0546 & .1026 & .0831 & .0640 & .0596 \\ \hline
 ATS   & .1150 & .0812 & .0571 & .0551 & .1275 & .0921 & .0668 & .0598 \\   \hline
ATS-Tay & .0819 & .0658 & \bf{.0519} &\bf{ .0534} & .0833 & .0702 & .0584 & .0561 \\ \hline

    ATSFz &.1089 & .0791 & .0565 & \bf{.0540} & .1203 & .0896 & .0663 & .0590 \\  \hline
 ATSFz-m & .0912 & .0713 & .0548 & \bf{.0533} & .1022 & .0821 & .0641 & .0582 \\    \hline
 % WTS & .4453 & .2270 & .1105 & .0808 & .5231 & .3048 & .1581 & .1019 \\ \hline
   Steiger & .0099 & .0193 & .0233 & .0260 & .0079 & .0178 & .0259 & .0275 \\ \hline
   SteigerFz &.0563 & .0784 & .0842 & .0931 & .0596 & .0828 & .0959 & .0993 \\   \hline
  Jennrich& .3280 & .2815 & .2556 & .2485 & .3400 & .3007 & .2712 & .2607 \\  
\specialrule{1.3pt}{0pt}{0pt}
\end{tabular}

\end{footnotesize}

\caption{Simulated type-I-error rates ($\alpha=5\%$) in scenario $A_{\vr}$ ($\mathcal{H}_0^{\vr}:\vR_{1}=\vR_{2}$)  for ATS, Steiger's and Jennrich's test. The observation vectors have dimension 5, covariance matrices $(\vV_1)_{ij}=0.6^{\lvert i-j\lvert}$   resp. $\vV_2=\diag(1,1.2,...,1.8)\vV_1$ and it always holds $n_1:=0.6\cdot N$ resp. $n_2:=0.4\cdot N$.}\label{tab:SimA2CorS}
\end{table}

    \begin{table}[htbp]
\centering
\begin{footnotesize}
\begin{tabular}{x{1.3pt}lx{1.3pt}c|c|c|cx{1.3pt}c|c|c|cx{1.3pt}}\specialrule{1.3pt}{0pt}{0pt}
   \multicolumn{1}{x{1.3pt}lx{1.3pt}}{} &\multicolumn{4}{ cx{1.3pt}}{$t_9$}&\multicolumn{4}{cx{1.3pt}}{Normal}\\
   %\specialrule{1.3pt}{0pt}{0pt}
\specialrule{1.3pt}{0pt}{0pt}
    \hspace*{.1cm}N&   20&40&100&200&   20&40&100&200
       \\ \specialrule{1.3pt}{0pt}{0pt}
 ATS-Par & .0954 & .0676 & .0611 & .0567 & .0847 & .0625 & .0561 & \bf{.0535} \\  \hline
 ATS-Wild & .1026 & .0723 & .0645 & .0590 & .0899 & .0662 & .0572 & \bf{.0541} \\  \hline
 ATS-Par-m & .0699 & .0579 & .0583 & .0553 & .0623 & \bf{.0524} & \bf{.0528} & \bf{.0517} \\  \hline
 ATS & .1327 & .0807 & .0666 & .0584 & .1203 & .0754 & .0588 &\bf{ .0539} \\  \hline
 ATS-Tay& .0964 & .0651 & \bf{.0489} & \bf{.0536} & .0904 & .0608 & .0545 & \bf{.0473} \\  \hline
  ATSFz & .1217 & .0763 & .0652 & .0581 & .1095 & .0720 & .0573 & .0544 \\ \hline
 ATSFz-m & .0924 & .0647 & .0604 & .0558 & .0828 & .0606 & \bf{.0535} &\bf{ .0518} \\  \hline
 %WTS &  .1327 & .0800 & .0668 & .0586 & .1195 & .0752 & .0587 & \bf{.0542} \\  \hline
 Steiger &$<$.001 & $<$.001&$<$.001 & $<$.001 & $<$.001 &$<$.001 & $<$.001 &$<$.001 \\ \hline
  SteigerFz& .0101 & .0318 & \bf{.0469} & \bf{.0471} & .0082 & .0280 & .0398 & .0457 \\  \hline
 Jennrich  & .0303 & .0308 & .0351 & .0350 & .0294 & .0277 & .0292 & .0312 \\ 
  \hline
 Om-Pa & .0786 & .0644 & .0654 & .0628 & .0967 & .0851 & .0825 & .1033 \\ 

 \specialrule{1.3pt}{0pt}{0pt}\multicolumn{9}{x{1.3pt}cx{1.3pt}}{}\\\specialrule{1.3pt}{0pt}{0pt}
   &\multicolumn{4}{cx{1.3pt}}{Skew Normal} &\multicolumn{4}{cx{1.3pt}}{Gamma}\\\specialrule{1.3pt}{0pt}{0pt}

    \hspace*{.1cm}N&   20&40&100&200&   20&40&100&200
       \\ \specialrule{1.3pt}{0pt}{0pt}
 ATS-Par &.0931 & .0747 & .0659 & .0597 & .0936 & .0901 & .0791 & .0679 \\  \hline
 ATS-Wild & .0987 & .0774 & .0675 & .0601 & .0995 & .0935 & .0809 & .0699 \\ \hline
 ATS-Par-m &.0678 & .0634 & .0621 & .0576 & .0704 & .0792 & .0748 & .0656 \\ \hline
 ATS & .1307 & .0866 & .0691 & .0603 & .1339 & .1035 & .0806 & .0708 \\  \hline
 ATS-Tay&  .0974 & .0704 & .0552 & .0544 & .0905 & .0795 & .0654 & .0595 \\  \hline
   ATSFz &  .1222 & .0832 & .0678 & .0592 & .1254 & .1003 & .0791 & .0707 \\  \hline
 ATSFz-m &.0911 & .0716 & .0640 & .0575 & .0927 & .0865 & .0758 & .0685 \\  \hline
% WTS &.1306 & .0872 & .0687 & .0599 & .1341 & .1036 & .0807 & .0711 \\ \hline
Steiger  & $<$ .001 & $<$ .001 & $<$ .001 &$<$ .001 &$<$ .001 & $<$ .001 & $<$ .001& $<$ .001\\ \hline
   SteigerFz &.0093 & .0311 & \bf{.0461} &\bf{ .0479} & .0100 & .0322 & \bf{.0461} & \bf{.0484} \\  \hline
Jennrich & .0263 & .0287 & .0333 & .0332 & .0237 & .0288 & .0332 & .0324 \\  \hline
 Om-Pa  & .0916 & .0837 & .0777 & .0785 & .0855 & .0676 &\bf{ .0542} & .0437 \\  
\specialrule{1.3pt}{0pt}{0pt}
\end{tabular}
\end{footnotesize}
\caption{Simulated type-I-error rates ($\alpha=5\%$) in scenario $A_{\vr}$ ($\mathcal{H}_0^{\vr}:\vR_{1}=\vR_{2}$) for ATS,  Steiger's, Jennrich's test and the test from \cite{omPa:2012}.  The observation vectors have dimension 2, covariance matrix $(\vV_1)_{ij}=1-{\lvert i-j\lvert/2d}$ resp. $\vV_2=\diag(1,1.5)\vV_1$ and it always holds $n_1:=0.6\cdot N$ resp. $n_2:=0.4\cdot N$.}\label{tab:SimA3CorS}\end{table}

\begin{table}[htbp]
\centering
\begin{footnotesize}
\begin{tabular}{x{1.3pt}lx{1.3pt}c|c|c|cx{1.3pt}c|c|c|cx{1.3pt}}\specialrule{1.3pt}{0pt}{0pt}
   \multicolumn{1}{x{1.3pt}lx{1.3pt}}{} &\multicolumn{4}{ cx{1.3pt}}{$t_9$}&\multicolumn{4}{cx{1.3pt}}{Normal}\\
 
\specialrule{1.3pt}{0pt}{0pt}
    \hspace*{.1cm}N&   20&40&100&200&   20&40&100&200
       \\ \specialrule{1.3pt}{0pt}{0pt}
  ATS-Par &  .1263 & .0839 & .0700 & .0607 & .1163 & .0787 & .0622 & .0568 \\  \hline
  ATS-Wild & .1371 & .0902 & .0722 & .0623 & .1271 & .0824 & .0626 & .0575 \\  \hline
    ATS-Par-m & .1033 & .0735 & .0666 & .0577 & .0941 & .0683 & .0581 & .0548 \\  \hline
  ATS & .1652 & .0984 & .0747 & .0622 & .1538 & .0919 & .0645 & .0576 \\ \hline
  ATS-Tay & .1308 & .0824 & .0579 & .0579 & .1271 & .0780 & .0617 &\bf{.0521} \\  \hline
    ATSFz & .1481 & .0913 & .0719 & .0613 & .1367 & .0841 & .0619 & .0568 \\  \hline
   ATSFz-m & .1170 & .0775 & .0669 & .0593 & .1052 & .0714 & .0587 & .0546 \\  \hline
   % WTS & .1651 & .0995 & .0750 & .0623 & .1534 & .0917 & .0643 & .0578 \\ \hline
   Steiger & $<$ .001 & .0026 & .0018 & .0027 & $<$ .001& .0032 & .0037 & .0027 \\ \hline
     SteigerFz &.0113 & .0320 &  \bf{.0472} &  \bf{.0479} & .0089 & .0288 & .0398 &  \bf{.0459} \\ \hline
   Jennrich & .0402 & .0353 & .0357 & .0348 & .0375 & .0324 & .0301 & .0307  \\ \hline
     Om-Pa &.0655 & .0570 & .0548 & \bf{ .0537} & .0854 & .0766 & .0783 & .0959 \\ 
  \specialrule{1.3pt}{0pt}{0pt}\multicolumn{9}{x{1.3pt}cx{1.3pt}}{}\\\specialrule{1.3pt}{0pt}{0pt}
   &\multicolumn{4}{cx{1.3pt}}{Skew Normal} &\multicolumn{4}{cx{1.3pt}}{Gamma}\\\specialrule{1.3pt}{0pt}{0pt}

    \hspace*{.1cm}N&   20&40&100&200&   20&40&100&200
       \\ \specialrule{1.3pt}{0pt}{0pt}
  ATS-Par & .1259 & .0906 & .0729 & .0605 & .1268 & .1042 & .0833 & .0716\\ \hline
  ATS-Wild &  .1329 & .0948 & .0742 & .0623 & .1326 & .1088 & .0854 & .0725 \\   \hline
    ATS-Par-m & .1000 & .0790 & .0694 & .0593 & .0989 & .0921 & .0799 & .0700 \\  \hline
  ATS & .1617 & .1031 & .0750 & .0624 & .1673 & .1185 & .0856 & .0730 \\  \hline
   ATS-Tay & .1327 & .0922 & .0616 & .0579 & .1276 & .0963 & .0732 & .0625 \\ \hline
    ATSFz & .1451 & .0948 & .0732 & .0608 & .1522 & .1122 & .0835 & .0719 \\ \hline 
   ATSFz-m  & .1167 & .0839 & .0686 & .0588 & .1170 & .0991 & .0798 & .0698 \\  \hline
    % WTS &.1617 & .1021 & .0748 & .0623 & .1667 & .1178 & .0856 & .0729 \\  \hline
  Steiger & $<$ .001 & .0013 & .0016 & .0017 & $<$ .001 & $<$ .001 & .0018 & .0021 \\  \hline 
    SteigerFz  & .0092 & .0318 &  \bf{.0460} &  \bf{.0474} & .0102 & .0322 &  \bf{.0458} &  \bf{.0474} \\  \hline
    Jennrich &.0369 & .0312 & .0346 & .0330 & .0307 & .0326 & .0336 & .0317 \\ \hline
    Om-Pa &.0818 & .0744 & .0699 & .0716 & .0747 & .0572 & .0448 & .0347 \\ 
\specialrule{1.3pt}{0pt}{0pt}
\end{tabular}
\end{footnotesize}

\caption{Simulated type-I-error rates ($\alpha=5\%$) in scenario $A_{\vr}$ ($\mathcal{H}_0^{\vr}:\vR_{1}=\vR_{2}$)  for ATS,  Steiger's, Jennrich's test and the test from \cite{omPa:2012}.  The observation vectors have dimension 2, covariance matrices $(\vV_1)_{ij}=0.6^{\lvert i-j\lvert}$  resp. $\vV_2=\diag(1,1.5)\vV_1$ and it always holds $n_1:=0.6\cdot N$ resp. $n_2:=0.4\cdot N$.}\label{tab:SimA4CorS}
\end{table}

\begin{table}[htp]
\centering
\begin{footnotesize}
\begin{tabular}{x{1.3pt}lx{1.3pt}c|c|c|cx{1.3pt}c|c|c|cx{1.3pt}}\specialrule{1.3pt}{0pt}{0pt}
   \multicolumn{1}{x{1.3pt}lx{1.3pt}}{} &\multicolumn{4}{ cx{1.3pt}}{$t_9$}&\multicolumn{4}{cx{1.3pt}}{Normal}\\

\specialrule{1.3pt}{0pt}{0pt} 
    \hspace*{.1cm}N&   25&50&125&250&   25&50&125&250
       \\ \specialrule{1.3pt}{0pt}{0pt}
  ATS-Par& .0863 & .0620 & \bf{.0537} & \bf{.0501} & .0835 & .0605 & \bf{.0493} & .0559 \\   \hline
  ATS-Wild & .1783 & .1119 & .0732 & .0619 & .1577 & .0989 & .0650 & .0638 \\  \hline
  ATS-Par-m & \bf{.0507} & .0448 & \bf{.0473} & \bf{.0471} & .0455 & .0422 & .0429 & \bf{.0517} \\    \hline
ATS & .0974 & .0646 & \bf{.0533} & \bf{.0492} & .0940 & .0621 & \bf{.0506} & \bf{.0541} \\  \hline
  ATS-Tay & \bf{.0512} & .0394 & .0417 & .0431 & .0556 & .0436 & .0417 & \bf{.0500} \\  \hline
  ATSFz  & .0916 & .0607 & \bf{.0522} & \bf{.0490} & .0901 & .0605 & \bf{.0495} & \bf{.0541} \\  \hline
 ATSFz-m &\bf{ .0523} & \bf{.0459} & .0453 & \bf{.0458} &\bf{ .0508} & .0436 & .0420 & \bf{.0512} \\  \hline
  ATS-PCov & .0725 & \bf{.0517} & \bf{.0478} & .0443 & .0741 & \bf{.0538} & .0457 & \bf{.0499} \\  \hline
  ATS-WCov & .0774 & .0670 & .0632 & .0566 & .0832 & .0687 & .0593 & .0605 \\ \hline
  ATS-Cov & .0824 & .0552 & \bf{.0473} & .0441 & .0840 & .0559 &\bf{ .0465} & \bf{.0498} \\ \hline
  % WTS  & .8388 & .5225 & .2270 & .1320 & .8218 & .4913 & .2053 & .1218 \\   \hline
  Steiger & .0220 & .0326 & .0419 &\bf{ .0457} & .0255 & .0355 & .0421 & \bf{.0507} \\   \hline
 SteigerFz & \bf{.0515} & \bf{.0471} & \bf{.0489} & \bf{.0491} & .0558 & \bf{.0517} & \bf{.0492} & \bf{.0536} \\   \hline
  Bartlett & \bf{.0470} & \bf{.0480} & \bf{.0488} & \bf{.0500} & \bf{.0525} & \bf{.0535} & \bf{.0479} & \bf{.0539} \\   
\specialrule{1.3pt}{0pt}{0pt}\multicolumn{9}{x{1.3pt}cx{1.3pt}}{}\\\specialrule{1.3pt}{0pt}{0pt}
   &\multicolumn{4}{cx{1.3pt}}{Skew Normal}&\multicolumn{4}{cx{1.3pt}}{Gamma}\\\specialrule{1.3pt}{0pt}{0pt}

    \hspace*{.1cm}N&   25&50&125&250&   25&50&125&250
       \\ \specialrule{1.3pt}{0pt}{0pt}
  ATS-Par &  .0846 & .0612 & \bf{.0543} & \bf{.0509} & .1134 & .0712 & .0594 & \bf{.0514} \\   \hline
  ATS-Wild & .1712 & .1030 & .0754 & .0596 & .2214 & .1297 & .0884 & .0668 \\   \hline
  ATS-Par-m & \bf{.0484} & .0440 & \bf{.0470} & \bf{.0481} & .0662 & \bf{.0516} & \bf{.0518} & \bf{.0485} \\   \hline
 ATS & .0961 & .0630 & \bf{.0531} & \bf{.0507} & .1287 & .0758 & .0596 & \bf{.0507} \\  \hline
  ATS-Tay  &\bf{ .0504} & .0408 & .0433 & .0456 & \bf{.0526} & .0366 & .0401 & .0420 \\ \hline
 ATSFz & .0914 & .0606 & \bf{.0523} & \bf{.0499} & .1163 & .0698 & .0589 & \bf{.0500} \\    \hline
  ATSFz-m & \bf{.0526} & .0442 & \bf{.0462} & \bf{.0462} & .0657 & \bf{.0485} & \bf{.0514} & \bf{.0472} \\   \hline
 ATS-PCov & .0727 & \bf{.0527} & \bf{.0488} & \bf{.0456} & .0908 & .0567 & \bf{.0500} & \bf{.0470} \\  \hline
   ATS-WCov & .0768 & .0663 & .0606 & .0576 & .0806 & .0630 & .0632 & .0570 \\  \hline
  ATS-Cov & .0820 & .0565 & \bf{.0484} & .0452 & .1034 & .0598 &\bf{ .0509} & .0453 \\   \hline
 % WTS & .8299 & .5047 & .2196 & .1249 & .8572 & .5656 & .2754 & .1550 \\   \hline
   Steiger& .0252 & .0333 & \bf{.0458} & \bf{.0479} & .0277 & .0384 & \bf{.0462} & \bf{.0464} \\  \hline
   SteigerFz& .0548 & \bf{.0486} & \bf{.0542} & \bf{.0499} & .0640 & \bf{.0541} & .0555 & \bf{.0498} \\   \hline
   Bartlett & \bf{.0531} & \bf{.0456} & \bf{.0516} & \bf{.0506} & .0612 & \bf{.0527 }& \bf{.0525} &\bf{ .0504 }\\   
  \specialrule{1.3pt}{0pt}{0pt}
\end{tabular}
\end{footnotesize}

\caption{Simulated type-I-error rates ($\alpha=5\%$) in scenario $B_{\vr}$ ($\mathcal{H}_0^{\vr}:\vr=\vnull_{10}$)  for ATS,  Steiger's and Bartlett's test. The observation vectors have dimension 5 and covariance matrix $\vV=\diag(1,1.2,1.4,1.6,1.8)$.}\label{tab:SimBCorS}
\end{table}

In \Cref{tab:SimBCorS}, the three tests based on the vectorized covariance matrix {behave} comparably as for other hypotheses from \cite{sattler2022}. The corresponding ATS with a parametric bootstrap has overall error rates that are as good as the parametric bootstrap ATS for vectorized correlation matrices. It can be assumed that the variance of the components determines whether the correlation matrix or the covariance matrix is more suitable for this hypothesis. Variances greater than 1 make the corresponding correlation coefficients smaller than the components from the covariance matrix. This increases the type-I error rate for the covariance-based test. Similar for variances smaller than 1, which lowers the type-I error rate. To illustrate, we multiplicated the covariance matrix with an appropriate diagonal matrix, where the results are displayed in \Cref{tab:SimECor}. We included only individual test statistics since this does not influence the correlation matrix. For these small variances of the single components, the type-I error rate of the covariance-based tests clearly decreased. This dependency on the variance of the components makes the correlation-based approach more reliable, also through the standardization with the covariances, there could be more variance.

\begin{table}[htp]
\centering
\begin{footnotesize}
\begin{tabular}{x{1.3pt}lx{1.3pt}c|c|c|cx{1.3pt}c|c|c|cx{1.3pt}}\specialrule{1.3pt}{0pt}{0pt}
   \multicolumn{1}{x{1.3pt}lx{1.3pt}}{} &\multicolumn{4}{ cx{1.3pt}}{$t_9$}&\multicolumn{4}{cx{1.3pt}}{Normal}\\
   %\specialrule{1.3pt}{0pt}{0pt}
\specialrule{1.3pt}{0pt}{0pt} 
    \hspace*{.1cm}N&   25&50&125&250&   25&50&125&250
       \\ \specialrule{1.3pt}{0pt}{0pt}
ATS-Par & .0928 & .0639 & \bf{.0515} & \bf{ .0543} & .0798 & .0587 &\bf{ .0486} & \bf{.0526} \\ \hline
   ATS-Wild & .1859 & .1128 & .0768 & .0662 & .1584 & .0984 & .0644 & .0579 \\ \hline
  ATS & .1050 & .0681 & \bf{.0531} & \bf{.0532} & .0893 & .0617 & \bf{.0480} & \bf{.0514} \\ \hline
   ATS-PCov & .0645 & .0403 & .0314 & .0331 & .0587 & .0381 & .0288 & .0295 \\ \hline
  ATS-Cov & .0746 & .0426 & .0306 & .0335 & .0682 & .0398 & .0293 & .0300 \\

\specialrule{1.3pt}{0pt}{0pt}\multicolumn{9}{x{1.3pt}cx{1.3pt}}{}\\\specialrule{1.3pt}{0pt}{0pt}
   &\multicolumn{4}{cx{1.3pt}}{Skew Normal}&\multicolumn{4}{cx{1.3pt}}{Gamma}\\
\specialrule{1.3pt}{0pt}{0pt} 
    \hspace*{.1cm}N&   25&50&125&250&   25&50&125&250\\
\specialrule{1.3pt}{0pt}{0pt}

ATS-Par & .0872 & .0630 & .0568 &\bf{ .0516} & .1066 & .0740 & .0573 & \bf{.0519} \\ \hline
   ATS-Wild & .1686 & .1079 & .0792 & .0601 & .2132 & .1342 & .0810 & .0663 \\ \hline
  ATS & .0966 & .0674 & .0575 & \bf{.0507} & .1169 & .0776 & .0563 &\bf{ .0511} \\ \hline
   ATS-PCov & .0611 & .0415 & .0344 & .0306 & .0675 & \bf{.0489} & .0351 & .0339 \\ \hline
   ATS-Cov & .0689 & .0439 & .0358 & .0303 & .0779 & \bf{.0523} & .0356 & .0334 \\ 

  \specialrule{1.3pt}{0pt}{0pt}
\end{tabular}
\end{footnotesize}

\caption{Simulated type-I-error rates ($\alpha=5\%$) in scenario $B_{\vr}$ ($\mathcal{H}_0^{\vr}:\vr=\vnull_{10}$)  for the  ATS.  The observation vectors have dimension 5 and covariance matrix $\vV=\diag(0.2,0.4,0.6,0.8,1)$.}\label{tab:SimECor}
\end{table}

\begin{table}[ht]
\centering
\begin{footnotesize}
\begin{tabular}{x{1.3pt}lx{1.3pt}c|c|c|cx{1.3pt}c|c|c|cx{1.3pt}}\specialrule{1.3pt}{0pt}{0pt}
   \multicolumn{1}{x{1.3pt}lx{1.3pt}}{} &\multicolumn{4}{ cx{1.3pt}}{$t_9$}&\multicolumn{4}{cx{1.3pt}}{Normal}\\
  
\specialrule{1.3pt}{0pt}{0pt}
    \hspace*{.1cm}N&   25&50&125&250&   25&50&125&250
       \\ \specialrule{1.3pt}{0pt}{0pt} 
  ATS-Par &.0632 &\bf{ .0496} & \bf{.0476} & \bf{.0499} & \bf{.0523} & \bf{.0517} & .0456 & \bf{.0531} \\ \hline
   ATS-Wild & .1150 & .0789 & .0634 & .0591 & .0962 & .0752 & .0565 & .0585 \\  \hline
  ATS-Par-m &  .0328 & .0365 & .0419 & \bf{.0469} & .0282 & .0360 & .0395 &\bf{ .0501} \\ \hline
 ATS & .0713 & \bf{.0526} & .0477 & \bf{.0490} & .0618 & \bf{.0538} & .0454 & \bf{.0522} \\  \hline
  ATS-Tay& .0290 & .0285 & .0354 & .0409 & .0275 & .0300 & .0345 & \bf{.0459} \\  \hline

  ATSFz &.0666 & \bf{.0513} & \bf{.0463} & \bf{.0488} & .0563 & \bf{.0525} & .0454 & \bf{.0519} \\   \hline
   ATSFz-m & .0346 & .0369 & .0412 & \bf{.0460} & .0289 & .0353 & .0392 & \bf{.0486} \\  \hline
 %WTS & .6103 & .3214 & .1477 & .1038 & .5506 & .2831 & .1244 & .0887 \\ 

 \specialrule{1.3pt}{0pt}{0pt}\multicolumn{9}{x{1.3pt}cx{1.3pt}}{}\\\specialrule{1.3pt}{0pt}{0pt}
   &\multicolumn{4}{cx{1.3pt}}{Skew Normal}&\multicolumn{4}{cx{1.3pt}}{Gamma}\\\specialrule{1.3pt}{0pt}{0pt}

    \hspace*{.1cm}N&   25&50&125&250&   25&50&125&250\\ \specialrule{1.3pt}{0pt}{0pt}

  ATS-Par &  .0695 & \bf{.0521} & \bf{.0508} & \bf{.0508} & .0921 & .0727 & .0556 & \bf{.0526} \\  \hline
   ATS-Wild & .1197 & .0786 & .0637 & .0546 & .1575 & .1081 & .0726 & .0630 \\  \hline
  ATS-Par-m & .0388 & .0391 & .0447 & \bf{.0479} & \bf{.0531} & .0551 & \bf{.0500} & \bf{.0499} \\  \hline
 ATS & .0795 & .0551 & \bf{.0502} & \bf{.0497} & .1028 & .0752 & .0565 &\bf{ .0517} \\   \hline
  ATS-Tay& .0336 & .0313 & .0383 & .0439 & .0365 & .0339 & .0378 & .0403 \\    \hline
 ATSFz & .0751 & \bf{.0523} & \bf{.0490} & \bf{.0493} & .0986 & .0734 & .0551 & \bf{.0511} \\  \hline
 ATSFz-m& .0409 & .0399 & .0447 & \bf{.0470} & .0576 & .0575 & \bf{.0504} & \bf{.0481} \\  \hline
 % WTS & .5888 & .3176 & .1522 & .0950 & .6647 & .4136 & .2009 & .1243 \\ 
  \specialrule{1.3pt}{0pt}{0pt}
\end{tabular}
\end{footnotesize}

\caption{Simulated type-I-error rates ($\alpha=5\%$) in scenario  $C_{\vr}$ ($\mathcal{H}_0^{\vr}:r_{1}=r_{2}=...=r_{10}$)  for the ATS.  The observation vectors have dimension 5 and covariance matrix $\vV=\vI_5+\vJ_5$.}\label{tab:SimCCor}
\end{table}

In \Cref{tab:SimCCor} again $\varphi_{ATS^\dagger},\varphi_{ATS}$ and $\varphi_{ATSFz}$ have the best performance as well as their multiplicated versions. This multiplication makes the test less conservative, which is especially useful for smaller sample sizes but can lead to slightly conservative behaviour for larger sample sizes. But these tests fulfil Bradley's liberal criterion for $n_1\geq 50$ and often also for $n_1=25$. Overall,  it is noticeable that the difference between wild bootstrap and parametric bootstrap is essentially larger for one group than for two.
Moreover, the Taylor-based approach performs better for two groups since it gets slightly liberal for the settings with only one group.\\

{For higher dimensions, a worse small sample performance would be expected since the sample size is smaller in relation to the number of unknown parameters $p_u$. Indeed, this holds only for ATS-Par-m and ATSFz-m, while for other ones like the parametric bootstrap, the performance seems even better, as it can be seen in \Cref{tab:SimA1d7Cor}-\ref{tab:SimCd7Cor}. This result shows that the introduced techniques are also interesting for higher dimensions.}

    \begin{table}[htbp]
\centering
\begin{footnotesize}
\begin{tabular}{x{1.3pt}lx{1.3pt}c|c|c|cx{1.3pt}c|c|c|cx{1.3pt}}\specialrule{1.3pt}{0pt}{0pt}
   \multicolumn{1}{x{1.3pt}lx{1.3pt}}{} &\multicolumn{4}{ cx{1.3pt}}{$t_9$}&\multicolumn{4}{cx{1.3pt}}{Normal}\\
   %\specialrule{1.3pt}{0pt}{0pt}
\specialrule{1.3pt}{0pt}{0pt}
    \hspace*{.1cm}N&   70&140&350&700&   70&140&350&700
       \\ \specialrule{1.3pt}{0pt}{0pt}
 ATS-Par &  .0660 & .0601 & .0572 & \bf{.0492} & .0685 & .0584 & \bf{.0539} & \bf{.0505} \\  \hline
 ATS-Wild & .0707 & .0618 & .0586 & \bf{.0520} & .0734 & .0593 & \bf{.0542} & \bf{.0517} \\  \hline
 ATS-Par-m & .0598 & .0573 & .0557 & \bf{.0488} & .0620 & .0553 & \bf{.0526} & \bf{.0501} \\    \hline
 ATS & .0695 & .0608 & .0571 &\bf{ .0499 }& .0720 & .0574 & \bf{.0543} & \bf{.0514}\\  \hline
ATS-Tay &.0548 & \bf{.0534} & \bf{.0541} & \bf{.0477} & .0578 & \bf{.0516} & \bf{.0508} & \bf{.0490} \\   \hline

    ATSFz & .0680 & .0600 & .0569 &\bf{ .0496} & .0704 & .0572 & \bf{.0527} & \bf{.0512} \\  \hline
 ATSFz-m &.0618 & .0573 & .0555 & \bf{.0493} & .0629 & .0545 & \bf{.0518} & \bf{.0505}\\   \hline
%  WTS & .1948 & .1003 & .0699 & .0613 & .2892 & .1775 & .1122 & .0855 \\  \hline
   Steiger & $<$.001&$<$.001 &$<$.001&$<$.001 &$<$.001 &$<$.001 &$<$.001&$<$.001 \\ \hline
   SteigerFz &.1372 & .1601 & .1705 & .1744 & .1413 & .1472 & .1670 & .1698 \\   \hline
 Jennrich & .6115 & .7344 & .7978 & .8113 & .5985 & .7058 & .7738 & .7888 \\  
 \specialrule{1.3pt}{0pt}{0pt}\multicolumn{9}{x{1.3pt}cx{1.3pt}}{}\\\specialrule{1.3pt}{0pt}{0pt}
   &\multicolumn{4}{cx{1.3pt}}{Skew Normal} &\multicolumn{4}{cx{1.3pt}}{Gamma}\\\specialrule{1.3pt}{0pt}{0pt}

%\specialrule{1.3pt}{0pt}{0pt}
    \hspace*{.1cm}N&   70&140&350&700&   70&140&350&700
       \\ \specialrule{1.3pt}{0pt}{0pt}
 ATS-Par & .0701 & .0627 & \bf{.0537} & .0561 & .0756 & .0657 & .0602 & \bf{.0534} \\  \hline
 ATS-Wild &.0752 & .0654 & .0546 & .0564 & .0798 & .0698 & .0603 & \bf{.0535} \\   \hline
 ATS-Par-m & .0623 & .0592 & \bf{.0529} & .0556 & .0690 & .0630 & .0588 & \bf{.0532} \\    \hline
 ATS & .0724 & .0639 & \bf{.0542} & .0552 & .0783 & .0693 & .0581 & \bf{.0530} \\  \hline
ATS-Tay &.0586 & .0551 & \bf{.0513} & \bf{.0542} & .0602 & .0587 & \bf{.0542} & \bf{.0506} \\   \hline

    ATSFz & .0700 & .0627 & \bf{.0539} & .0554 & .0759 & .0665 & .0577 & \bf{.0531}\\  \hline
 ATSFz-m &.0636 & .0598 & \bf{.0529} & { .0549} & .0699 & .0641 & .0562 & \bf{.0521} \\    \hline
%  WTS & .1948 & .1003 & .0699 & .0613 & .2892 & .1775 & .1122 & .0855 \\  \hline
   Steiger & $<$.001&$<$.001 &$<$.001&$<$.001 &$<$.001 &$<$.001 &$<$.001&$<$.001 \\ \hline
   SteigerFz & .1398 & .1593 & .1639 & .1725 & .1504 & .1730 & .1868 & .1753 \\  \hline
 Jennrich & .6231 & .7257 & .7819 & .8018 & .6307 & .7456 & .8251 & .8374  \\  
\specialrule{1.3pt}{0pt}{0pt}
\end{tabular}
\end{footnotesize}

\caption{Simulated type-I-error rates ($\alpha=5\%$) in scenario $A_{\vr}$ ($\mathcal{H}_0^{\vr}:\vR_{1}=\vR_{2}$)  for ATS, Steiger's and Jennrich's test. The observation vectors have dimension 7, covariance matrix $(\vV_1)_{ij}=1-{\lvert i-j\lvert/2d}$ resp. $\vV_2=\diag(7,8,9,10,11,12,13)/7\vV_1$ and it always holds $n_1:=0.6\cdot N$ resp. $n_2:=0.4\cdot N$.}\label{tab:SimA1d7Cor}\end{table}

\begin{table}[htbp]
\centering
\begin{footnotesize}
\begin{tabular}{x{1.3pt}lx{1.3pt}c|c|c|cx{1.3pt}c|c|c|cx{1.3pt}}\specialrule{1.3pt}{0pt}{0pt}
   \multicolumn{1}{x{1.3pt}lx{1.3pt}}{} &\multicolumn{4}{ cx{1.3pt}}{$t_9$}&\multicolumn{4}{cx{1.3pt}}{Normal}\\
   %\specialrule{1.3pt}{0pt}{0pt}
\specialrule{1.3pt}{0pt}{0pt}
    \hspace*{.1cm}N&   70&140&350&700&   70&140&350&700
       \\ \specialrule{1.3pt}{0pt}{0pt}
 ATS-Par & .0921 & .0750 & .0583 & .0588 & .0895 & .0688 & .0591 & \bf{.0533} \\  \hline
 ATS-Wild & .1111 & .0843 & .0633 & .0592 & .1082 & .0762 & .0630 & .0559 \\   \hline
 ATS-Par-m &  .0799 & .0705 & .0568 & .0580 & .0790 & .0635 & .0568 & \bf{.0526} \\    \hline
 ATS &  .0938 & .0764 & .0582 & .0561 & .0921 & .0687 & .0588 & .0546 \\  \hline
ATS-Tay &.0649 & .0585 & \bf{.0524} & \bf{.0535} & .0668 & .0567 & .0551 & \bf{.0500} \\   \hline

    ATSFz & .0901 & .0747 & .0574 & .0563 & .0894 & .0672 & .0587 & \bf{.0530} \\  \hline
 ATSFz-m & .0794 & .0700 & .0562 & .0556 & .0782 & .0628 & .0569 & \bf{.0525} \\   \hline
%  WTS & .4688 & .2456 & .1162 & .0815 & .4175 & .2009 & .1032 & .0751 \\  \hline
   Steiger &  .0073 & .0122 & .0125 & .0127 & .0083 & .0117 & .0119 & .0112 \\   \hline
   SteigerFz &.0908 & .1129 & .1177 & .1190 & .0890 & .1016 & .1180 & .1195 \\   \hline
  Jennrich & .3053 & .2304 & .1864 & .1825 & .2975 & .2143 & .1771 & .1587  \\    
 \specialrule{1.3pt}{0pt}{0pt}\multicolumn{9}{x{1.3pt}cx{1.3pt}}{}\\\specialrule{1.3pt}{0pt}{0pt}
   &\multicolumn{4}{cx{1.3pt}}{Skew Normal} &\multicolumn{4}{cx{1.3pt}}{Gamma}\\\specialrule{1.3pt}{0pt}{0pt}

    \hspace*{.1cm}N&   70&140&350&700&   70&140&350&700
       \\ \specialrule{1.3pt}{0pt}{0pt}
 ATS-Par & .0913 & .0737 & .0563 & .0591 & .1049 & .0791 & .0649 & .0553 \\ \hline
 ATS-Wild & .1132 & .0818 & .0606 & .0597 & .1259 & .0900 & .0702 & .0585 \\   \hline
 ATS-Par-m &.0796 & .0689 & .0554 & .0580 & .0922 & .0734 & .0636 & \bf{.0541} \\ \hline
 ATS   &.0956 & .0751 & .0554 & .0584 & .1062 & .0793 & .0651 & \bf{.0538} \\ \hline
ATS-Tay & .0643 & .0585 & \bf{.0510} & .0550 & .0679 & .0598 & .0575 & \bf{.0505} \\ \hline

    ATSFz &.0930 & .0737 & .0554 & .0584 & .1043 & .0783 & .0641 & .0545 \\   \hline
 ATSFz-m &  .0792 & .0681 & \bf{ .0537} & .0574 & .0921 & .0724 & .0627 & \bf{.0531} \\  \hline
 % WTS & .4453 & .2270 & .1105 & .0808 & .5231 & .3048 & .1581 & .1019 \\ \hline
   Steiger & .0080 & .0117 & .0108 & .0137 & .0082 & .0091 & .0157 & .0133  \\ \hline
   SteigerFz &.0923 & .1068 & .1156 & .1207 & .0982 & .1189 & .1315 & .1280 \\  \hline
  Jennrich& .3017 & .2257 & .1759 & .1654 & .3125 & .2366 & .2018 & .1879 \\  
\specialrule{1.3pt}{0pt}{0pt}
\end{tabular}

\end{footnotesize}

\caption{Simulated type-I-error rates ($\alpha=5\%$) in scenario $A_{\vr}$ ($\mathcal{H}_0^{\vr}:\vR_{1}=\vR_{2}$)  for ATS, Steiger's and Jennrich's test. The observation vectors have dimension 7, covariance matrices $(\vV_1)_{ij}=0.6^{\lvert i-j\lvert}$   resp. $\vV_2=\diag(7,8,9,10,11,12,13)/7\vV_1$ and it always holds $n_1:=0.6\cdot N$ resp. $n_2:=0.4\cdot N$.}\label{tab:SimA2d7Cor}
\end{table}

\begin{table}[htp]
\centering
\begin{footnotesize}
\begin{tabular}{x{1.3pt}lx{1.3pt}c|c|c|cx{1.3pt}c|c|c|cx{1.3pt}}\specialrule{1.3pt}{0pt}{0pt}
   \multicolumn{1}{x{1.3pt}lx{1.3pt}}{} &\multicolumn{4}{ cx{1.3pt}}{$t_9$}&\multicolumn{4}{cx{1.3pt}}{Normal}\\

\specialrule{1.3pt}{0pt}{0pt} 
    \hspace*{.1cm}N&   35&70&175&350&   35&70&175&350
       \\ \specialrule{1.3pt}{0pt}{0pt}
  ATS-Par& .0615 & \bf{.0520} & .0452 & .0450 & \bf{.0538} & .0456 & \bf{.0483} & \bf{.0508} \\   \hline
  ATS-Wild &  .1821 & .1149 & .0708 & .0606 & .1591 & .0908 & .0695 & .0598 \\ \hline
  ATS-Par-m & .0304 & .0368 & .0399 & .0423 & .0272 & .0325 & .0433 & \bf{.0483}  \\  \hline
ATS & .0658 & \bf{.0528} & .0444 & \bf{.0458} & .0588 & \bf{.0459} & \bf{.0496} & \bf{.0497} \\ \hline
  ATS-Tay & .0245 & .0304 & .0349 & .0393 & .0296 & .0304 & .0413 & .0452\\  \hline
  ATSFz  & .0646 & \bf{.0496} & .0436 & .0454 & .0589 & .0455 & \bf{.0484} & \bf{.0491} \\ \hline
 ATSFz-m & .0318 & .0363 & .0391 & .0424 & .0295 & .0325 & .0422 & \bf{.0463} \\  \hline
 
  % WTS  & .8388 & .5225 & .2270 & .1320 & .8218 & .4913 & .2053 & .1218 \\   \hline
  Steiger &  .0266 & .0407 & .0429 & \bf{.0477} & .0255 & .0350 & \bf{.0481} & \bf{.0482} \\   \hline
 SteigerFz & \bf{.0530} & .0554 & \bf{.0483} & \bf{.0511} & .0548 & \bf{.0479} & \bf{.0531} & \bf{.0512} \\   \hline
  Bartlett & \bf{.0506} & .0549 & \bf{.0469} & \bf{.0514} & \bf{.0531} & \bf{.0465} & \bf{.0519} & \bf{.0501}\\   
\specialrule{1.3pt}{0pt}{0pt}\multicolumn{9}{x{1.3pt}cx{1.3pt}}{}\\\specialrule{1.3pt}{0pt}{0pt}
   &\multicolumn{4}{cx{1.3pt}}{Skew Normal}&\multicolumn{4}{cx{1.3pt}}{Gamma}\\\specialrule{1.3pt}{0pt}{0pt}

    \hspace*{.1cm}N&   35&70&175&350&  35&70&175&350
       \\ \specialrule{1.3pt}{0pt}{0pt}
  ATS-Par &  .0604 & \bf{.0487} & .0443 & \bf{.0491} & .0737 & \bf{.0520} & .0456 & \bf{.0491} \\    \hline
  ATS-Wild &.1748 & .0999 & .0692 & .0627 & .2242 & .1306 & .0813 & .0667 \\   \hline
  ATS-Par-m &  .0306 & .0348 & .0392 & \bf{.0460} & .0411 & .0370 & .0397 & \bf{.0464} \\   \hline
 ATS &.0655 & \bf{.0482} & .0440 & \bf{.0491} & .0785 & \bf{.0534} & \bf{.0463} & \bf{.0475} \\   \hline
  ATS-Tay  &.0261 & .0288 & .0349 & .0449 & .0260 & .0222 & .0288 & .0381 \\  \hline
 ATSFz & .0659 & \bf{.0469} & .0432 & \bf{.0494} & .0730 & \bf{.0508} & .0455 & \bf{.0480} \\   \hline
  ATSFz-m &.0331 & .0342 & .0390 & \bf{.0464} & .0413 & .0353 & .0390 & .0445 \\    \hline
 % WTS & .8299 & .5047 & .2196 & .1249 & .8572 & .5656 & .2754 & .1550 \\   \hline
   Steiger& .0271 & .0374 & .0419 & \bf{.0492} & .0288 & .0379 & \bf{.0464} & \bf{.0510} \\   \hline
   SteigerFz& .0582 & \bf{.0518} & \bf{.0474} & \bf{.0523} & .0578 & .0547 & \bf{.0521} & \bf{.0536} \\   \hline
   Bartlett & .0556 & \bf{.0500} & \bf{.0468} & \bf{.0510} & .0566 & \bf{.0506} & \bf{.0536} & \bf{.0541} \\    
  \specialrule{1.3pt}{0pt}{0pt}
\end{tabular}
\end{footnotesize}

\caption{Simulated type-I-error rates ($\alpha=5\%$) in scenario $B_{\vr}$ ($\mathcal{H}_0^{\vr}:\vr=\vnull_{21}$)  for ATS,  Steiger's and Bartlett's test. The observation vectors have dimension 7 and covariance matrix $\vV=\diag(7,8,9,10,11,12,13)/7$.}\label{tab:SimBd7Cor}
\end{table}

\begin{table}[ht]
\centering
\begin{footnotesize}
\begin{tabular}{x{1.3pt}lx{1.3pt}c|c|c|cx{1.3pt}c|c|c|cx{1.3pt}}\specialrule{1.3pt}{0pt}{0pt}
   \multicolumn{1}{x{1.3pt}lx{1.3pt}}{} &\multicolumn{4}{ cx{1.3pt}}{$t_9$}&\multicolumn{4}{cx{1.3pt}}{Normal}\\
  
\specialrule{1.3pt}{0pt}{0pt}
    \hspace*{.1cm}N&   35&70&175&350&   35&70&175&350
       \\ \specialrule{1.3pt}{0pt}{0pt} 
  ATS-Par &\bf{.0459} & .0443 & .0436 & \bf{.0461} & .0438 & .0431 & \bf{.0468} & \bf{.0496} \\  \hline
   ATS-Wild & .1117 & .0837 & .0644 & .0551 & .0995 & .0749 & .0588 & .0556 \\   \hline
  ATS-Par-m &   .0249 & .0340 & .0385 & .0428 & .0242 & .0318 & .0420 & \bf{.0475} \\ \hline
 ATS &  \bf{.0501} & \bf{.0466} & .0433 & .0447 & \bf{.0476} & .0440 & \bf{.0469} & \bf{.0490} \\   \hline
  ATS-Tay&  .0154 & .0218 & .0288 & .0369 & .0167 & .0226 & .0338 & .0425 \\  \hline

  ATSFz &\bf{.0467} & .0448 & .0431 & .0440 & {.0457} & .0427 & \bf{.0469} & \bf{.0488} \\  \hline
   ATSFz-m &  .0259 & .0322 & .0384 & .0421 & .0232 & .0311 & .0422 & .\bf{0464} \\   \hline
 %WTS & .6103 & .3214 & .1477 & .1038 & .5506 & .2831 & .1244 & .0887 \\ 

 \specialrule{1.3pt}{0pt}{0pt}\multicolumn{9}{x{1.3pt}cx{1.3pt}}{}\\\specialrule{1.3pt}{0pt}{0pt}
   &\multicolumn{4}{cx{1.3pt}}{Skew Normal}&\multicolumn{4}{cx{1.3pt}}{Gamma}\\\specialrule{1.3pt}{0pt}{0pt}

    \hspace*{.1cm}N&   35&70&175&350&   35&70&175&350\\ \specialrule{1.3pt}{0pt}{0pt} 

  ATS-Par & \bf{ .0505} & .0453 & .0457 & \bf{ .0461} & .0693 & .0575 & \bf{ .0518} & \bf{ .0523} \\   \hline
   ATS-Wild & .1127 & .0781 & .0604 & .0547 & .1546 & .1031 & .0722 & .0624 \\   \hline
  ATS-Par-m &  .0284 & .0326 & .0404 & .0439 & .0422 & { .0453} & \bf{ .0460} & \bf{ .0503} \\   \hline
 ATS & .0546 & .0455 & .0456 & \bf{ .0463} & .0731 & .0603 & \bf{ .0513} & \bf{ .0518} \\  \hline
  ATS-Tay& .0207 & .0231 & .0322 & .0378 & .0211 & .0234 & .0330 & .0425 \\    \hline
 ATSFz &\bf{  .0525} & .0447 & .0453 & .0454 & .0727 & .0590 & \bf{ .0507} & \bf{ .0512} \\   \hline
 ATSFz-m&  .0294 & .0334 & .0396 & .0426 & .0433 & .0449 & \bf{ .0464} & \bf{ .0490} \\  \hline
 % WTS & .5888 & .3176 & .1522 & .0950 & .6647 & .4136 & .2009 & .1243 \\ 
  \specialrule{1.3pt}{0pt}{0pt}
\end{tabular}
\end{footnotesize}

\caption{Simulated type-I-error rates ($\alpha=5\%$) in scenario  $C_{\vr}$ ($\mathcal{H}_0^{\vr}:r_{1}=r_{2}=...=r_{21}$)  for the ATS.  The observation vectors have dimension 7 and covariance matrix $\vV=\vI_7+\vJ_7$.}\label{tab:SimCd7Cor}
\end{table}

\subsection*{Power}
To examine the powers dependency on the underlying distribution, the power simulation is now done with the same setting but a gamma distribution instead of a skewed normal distribution.
As it can be seen in   \Cref{fig:PowerCorABS} the results are very similar to the power for the skewed normal distribution. This suggests that the results do not depend on the underlying distribution.\\

\captionsetup[figure]{position=below}
\begin{figure}[htb]
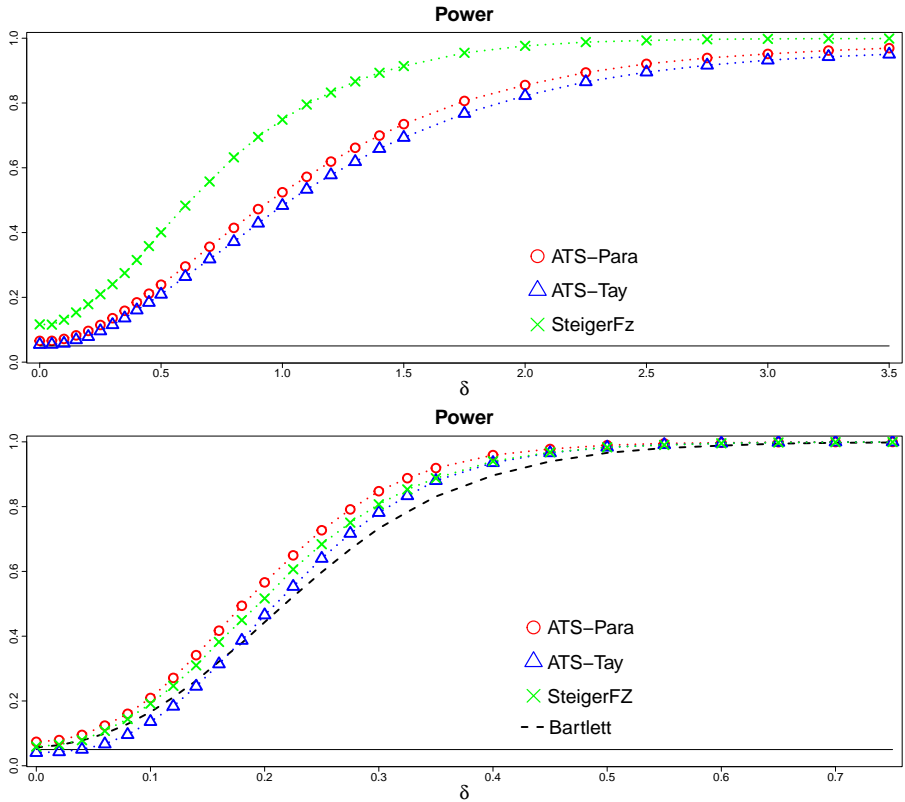

\begin{minipage}[t]{0.99\textwidth}\vspace{0pt} 
\includegraphics[width=\textwidth,trim= 11.5mm 3mm 10.mm 6mm,clip]{PowerCorA4} 
\end{minipage}\\ 
\begin{minipage}[t]{0.99\textwidth}\vspace{0pt} 
\includegraphics[width=\textwidth,trim= 11.5mm 4mm 10mm 5mm,clip]{PowerCorB4} 
\end{minipage} % 
\caption{ 
Simulated power curves of different tests for the hypothesis $A_{\vr}$ ($\mathcal H_0^{\vr}: \vR_1 =  \vR_2 $) above, and
hypothesis $B_{\vr}$ ($\mathcal H_0^{\vr}: \vr_1 =  \vnull_{10} $) below, for a 5-dimensional gamma distribution. For hypothesis $A_{\vr}$ it holds $n_1=60$, $n_2=40$ and the covariance matrix is $(\vV_2)_{ij}=1-\lvert i-j\lvert/2d$ resp. $\vV_1=\vV_2+\delta \vJ_5$.  The covariance matrix  for hypothesis $B_{\vr}$ is $\vV=\vI_5+\delta \vJ_5$ and $n_1=50$.   }

\label{fig:PowerCorABS} 
\end{figure} 
One last time, this shows that our developed test offers an alternative to the existing test procedure, which is often preferable due to its flexibility and good power and type-I error rate in different settings.

\bibliographystyle{apacite}
\newpage
\phantomsection
\addcontentsline{toc}{chapter}{Bibliography}

\clearpage
\bibliography{Literatur}

\begin{thebibliography}{}

\bibitem [\protect \citeauthoryear {%
Aitkin%
, Nelson%
\BCBL {}\ \BBA {} Reinfurt%
}{%
Aitkin%
\ \protect \BOthers {.}}{%
{\protect \APACyear {1968}}%
}]{%
aitkin1968}
\APACinsertmetastar {%
aitkin1968}%
\begin{APACrefauthors}%
Aitkin, M\BPBI A.%
, Nelson, W\BPBI C.%
\BCBL {}\ \BBA {} Reinfurt, K\BPBI H.%
\end{APACrefauthors}%
\unskip\
\newblock
\APACrefYearMonthDay{1968}{07}{}.
\newblock
{\BBOQ}\APACrefatitle {{Tests for correlation matrices}} {{Tests for
  correlation matrices}}.{\BBCQ}
\newblock
\APACjournalVolNumPages{Biometrika}{55}{2}{327-334}.
\newblock
\begin{APACrefURL} \url{https://doi.org/10.1093/biomet/55.2.327}
  \end{APACrefURL}
\newblock
\begin{APACrefDOI} \doi{10.1093/biomet/55.2.327} \end{APACrefDOI}
\PrintBackRefs{\CurrentBib}

\bibitem [\protect \citeauthoryear {%
Bartlett%
}{%
Bartlett%
}{%
{\protect \APACyear {1951}}%
}]{%
bartlett1951}
\APACinsertmetastar {%
bartlett1951}%
\begin{APACrefauthors}%
Bartlett, M\BPBI S.%
\end{APACrefauthors}%
\unskip\
\newblock
\APACrefYearMonthDay{1951}{}{}.
\newblock
{\BBOQ}\APACrefatitle {The Effect of Standardization on a chi square
  Approximation in Factor Analysis} {The effect of standardization on a chi
  square approximation in factor analysis}.{\BBCQ}
\newblock
\APACjournalVolNumPages{Biometrika}{38}{3/4}{337--344}.
\newblock
\begin{APACrefURL} \url{http://www.jstor.org/stable/2332580} \end{APACrefURL}
\PrintBackRefs{\CurrentBib}

\bibitem [\protect \citeauthoryear {%
Bartlett%
\ \BBA {} Rajalakshman%
}{%
Bartlett%
\ \BBA {} Rajalakshman%
}{%
{\protect \APACyear {1953}}%
}]{%
bartlett1953}
\APACinsertmetastar {%
bartlett1953}%
\begin{APACrefauthors}%
Bartlett, M\BPBI S.%
\BCBT {}\ \BBA {} Rajalakshman, D\BPBI V.%
\end{APACrefauthors}%
\unskip\
\newblock
\APACrefYearMonthDay{1953}{}{}.
\newblock
{\BBOQ}\APACrefatitle {Goodness of Fit Tests for Simultaneous Autoregressive
  Series} {Goodness of fit tests for simultaneous autoregressive
  series}.{\BBCQ}
\newblock
\APACjournalVolNumPages{Journal of the Royal Statistical Society. Series B
  (Methodological)}{15}{1}{107--124}.
\newblock
\begin{APACrefURL} \url{http://www.jstor.org/stable/2983727} \end{APACrefURL}
\PrintBackRefs{\CurrentBib}

\bibitem [\protect \citeauthoryear {%
Bathke%
\ \protect \BOthers {.}}{%
Bathke%
\ \protect \BOthers {.}}{%
{\protect \APACyear {2018}}%
}]{%
bathke2018}
\APACinsertmetastar {%
bathke2018}%
\begin{APACrefauthors}%
Bathke, A\BPBI C.%
, Friedrich, S.%
, Pauly, M.%
, Konietschke, F.%
, Staffen, W.%
, Strobl, N.%
\BCBL {}\ \BBA {} Höller, Y.%
\end{APACrefauthors}%
\unskip\
\newblock
\APACrefYearMonthDay{2018}{03}{}.
\newblock
{\BBOQ}\APACrefatitle {Testing Mean Differences among Groups: Multivariate and
  Repeated Measures Analysis with Minimal Assumptions} {Testing mean
  differences among groups: Multivariate and repeated measures analysis with
  minimal assumptions}.{\BBCQ}
\newblock
\APACjournalVolNumPages{Multivariate Behavioral Research}{53}{}{348-359}.
\newblock
\begin{APACrefDOI} \doi{10.1080/00273171.2018.1446320} \end{APACrefDOI}
\PrintBackRefs{\CurrentBib}

\bibitem [\protect \citeauthoryear {%
Boos%
\ \BBA {} Brownie%
}{%
Boos%
\ \BBA {} Brownie%
}{%
{\protect \APACyear {2004}}%
}]{%
boos2004}
\APACinsertmetastar {%
boos2004}%
\begin{APACrefauthors}%
Boos, D\BPBI D.%
\BCBT {}\ \BBA {} Brownie, C.%
\end{APACrefauthors}%
\unskip\
\newblock
\APACrefYearMonthDay{2004}{11}{}.
\newblock
{\BBOQ}\APACrefatitle {Comparing Variances and Other Measures of Dispersion}
  {Comparing variances and other measures of dispersion}.{\BBCQ}
\newblock
\APACjournalVolNumPages{Statist. Sci.}{19}{4}{571--578}.
\newblock
\begin{APACrefURL} \url{https://doi.org/10.1214/088342304000000503}
  \end{APACrefURL}
\newblock
\begin{APACrefDOI} \doi{10.1214/088342304000000503} \end{APACrefDOI}
\PrintBackRefs{\CurrentBib}

\bibitem [\protect \citeauthoryear {%
Box%
}{%
Box%
}{%
{\protect \APACyear {1950}}%
}]{%
box1950}
\APACinsertmetastar {%
box1950}%
\begin{APACrefauthors}%
Box, G\BPBI E\BPBI P.%
\end{APACrefauthors}%
\unskip\
\newblock
\APACrefYearMonthDay{1950}{}{}.
\newblock
{\BBOQ}\APACrefatitle {Problems in the Analysis of Growth and Wear Curves}
  {Problems in the analysis of growth and wear curves}.{\BBCQ}
\newblock
\APACjournalVolNumPages{Biometrics}{6}{4}{362--389}.
\newblock
\begin{APACrefURL} \url{http://www.jstor.org/stable/3001781} \end{APACrefURL}
\PrintBackRefs{\CurrentBib}

\bibitem [\protect \citeauthoryear {%
Bradley%
}{%
Bradley%
}{%
{\protect \APACyear {1978}}%
}]{%
bradley1978}
\APACinsertmetastar {%
bradley1978}%
\begin{APACrefauthors}%
Bradley, J\BPBI V.%
\end{APACrefauthors}%
\unskip\
\newblock
\APACrefYearMonthDay{1978}{}{}.
\newblock
{\BBOQ}\APACrefatitle {Robustness?} {Robustness?}{\BBCQ}
\newblock
\APACjournalVolNumPages{British Journal of Mathematical and Statistical
  Psychology}{31}{2}{144-152}.
\newblock
\begin{APACrefURL}
  \url{https://onlinelibrary.wiley.com/doi/abs/10.1111/j.2044-8317.1978.tb00581.x}
  \end{APACrefURL}
\newblock
\begin{APACrefDOI} \doi{10.1111/j.2044-8317.1978.tb00581.x} \end{APACrefDOI}
\PrintBackRefs{\CurrentBib}

\bibitem [\protect \citeauthoryear {%
Bretz%
, Genz%
\BCBL {}\ \BBA {} A.~Hothorn%
}{%
Bretz%
\ \protect \BOthers {.}}{%
{\protect \APACyear {2001}}%
}]{%
bretz2001}
\APACinsertmetastar {%
bretz2001}%
\begin{APACrefauthors}%
Bretz, F.%
, Genz, A.%
\BCBL {}\ \BBA {} A.~Hothorn, L.%
\end{APACrefauthors}%
\unskip\
\newblock
\APACrefYearMonthDay{2001}{}{}.
\newblock
{\BBOQ}\APACrefatitle {On the Numerical Availability of Multiple Comparison
  Procedures} {On the numerical availability of multiple comparison
  procedures}.{\BBCQ}
\newblock
\APACjournalVolNumPages{Biometrical Journal}{43}{5}{645-656}.
\newblock
\begin{APACrefURL}
  \url{https://onlinelibrary.wiley.com/doi/abs/10.1002/1521-4036%28200109%2943%3A5%3C645%3A%3AAID-BIMJ645%3E3.0.CO%3B2-F}
  \end{APACrefURL}
\newblock
\begin{APACrefDOI}
  \doi{https://doi.org/10.1002/1521-4036(200109)43:5<645::AID-BIMJ645>3.0.CO;2-F}
  \end{APACrefDOI}
\PrintBackRefs{\CurrentBib}

\bibitem [\protect \citeauthoryear {%
Browne%
\ \BBA {} Shapiro%
}{%
Browne%
\ \BBA {} Shapiro%
}{%
{\protect \APACyear {1986}}%
}]{%
browne}
\APACinsertmetastar {%
browne}%
\begin{APACrefauthors}%
Browne, M.%
\BCBT {}\ \BBA {} Shapiro, A.%
\end{APACrefauthors}%
\unskip\
\newblock
\APACrefYearMonthDay{1986}{}{}.
\newblock
{\BBOQ}\APACrefatitle {The asymptotic covariance matrix of sample correlation
  coefficients under general conditions} {The asymptotic covariance matrix of
  sample correlation coefficients under general conditions}.{\BBCQ}
\newblock
\APACjournalVolNumPages{Linear Algebra and its Applications}{82}{}{169-176}.
\newblock
\begin{APACrefURL}
  \url{http://www.sciencedirect.com/science/article/pii/0024379586901503}
  \end{APACrefURL}
\newblock
\begin{APACrefDOI} \doi{https://doi.org/10.1016/0024-3795(86)90150-3}
  \end{APACrefDOI}
\PrintBackRefs{\CurrentBib}

\bibitem [\protect \citeauthoryear {%
Efron%
}{%
Efron%
}{%
{\protect \APACyear {1988}}%
}]{%
efron1988}
\APACinsertmetastar {%
efron1988}%
\begin{APACrefauthors}%
Efron, B.%
\end{APACrefauthors}%
\unskip\
\newblock
\APACrefYearMonthDay{1988}{}{}.
\newblock
{\BBOQ}\APACrefatitle {Bootstrap confidence intervals: good or bad?} {Bootstrap
  confidence intervals: good or bad?}{\BBCQ}
\newblock
\APACjournalVolNumPages{Psychological bulletin}{104}{2}{293}.
\PrintBackRefs{\CurrentBib}

\bibitem [\protect \citeauthoryear {%
Fisher%
}{%
Fisher%
}{%
{\protect \APACyear {1921}}%
}]{%
fisher1921}
\APACinsertmetastar {%
fisher1921}%
\begin{APACrefauthors}%
Fisher, R.%
\end{APACrefauthors}%
\unskip\
\newblock
\APACrefYearMonthDay{1921}{}{}.
\newblock
{\BBOQ}\APACrefatitle {On the “probable error” of a correlation coefficient
  deduced from a small sample} {On the “probable error” of a correlation
  coefficient deduced from a small sample}.{\BBCQ}
\newblock
\APACjournalVolNumPages{Metron}{1}{4}{3--32}.
\PrintBackRefs{\CurrentBib}

\bibitem [\protect \citeauthoryear {%
Friedrich%
, Brunner%
\BCBL {}\ \BBA {} Pauly%
}{%
Friedrich%
\ \protect \BOthers {.}}{%
{\protect \APACyear {2017}}%
}]{%
friedrich2017permuting}
\APACinsertmetastar {%
friedrich2017permuting}%
\begin{APACrefauthors}%
Friedrich, S.%
, Brunner, E.%
\BCBL {}\ \BBA {} Pauly, M.%
\end{APACrefauthors}%
\unskip\
\newblock
\APACrefYearMonthDay{2017}{}{}.
\newblock
{\BBOQ}\APACrefatitle {Permuting longitudinal data in spite of the
  dependencies} {Permuting longitudinal data in spite of the
  dependencies}.{\BBCQ}
\newblock
\APACjournalVolNumPages{Journal of Multivariate Analysis}{153}{}{255--265}.
\PrintBackRefs{\CurrentBib}

\bibitem [\protect \citeauthoryear {%
Friedrich%
, Konietschke%
\BCBL {}\ \BBA {} Pauly%
}{%
Friedrich%
\ \protect \BOthers {.}}{%
{\protect \APACyear {2019}}%
}]{%
manova}
\APACinsertmetastar {%
manova}%
\begin{APACrefauthors}%
Friedrich, S.%
, Konietschke, F.%
\BCBL {}\ \BBA {} Pauly, M.%
\end{APACrefauthors}%
\unskip\
\newblock
\APACrefYearMonthDay{2019}{}{}.
\newblock
{\BBOQ}\APACrefatitle {MANOVA.RM: Analysis of Multivariate Data and Repeated
  Measures Designs} {Manova.rm: Analysis of multivariate data and repeated
  measures designs}{\BBCQ}\ [\bibcomputersoftwaremanual].
\newblock
\begin{APACrefURL} \url{https://cran.r-project.org/package=MANOVA.RM}
  \end{APACrefURL}
\newblock
\APACrefnote{R package version 0.3.2}
\PrintBackRefs{\CurrentBib}

\bibitem [\protect \citeauthoryear {%
Friedrich%
\ \BBA {} Pauly%
}{%
Friedrich%
\ \BBA {} Pauly%
}{%
{\protect \APACyear {2017}}%
}]{%
friedrich2017mats}
\APACinsertmetastar {%
friedrich2017mats}%
\begin{APACrefauthors}%
Friedrich, S.%
\BCBT {}\ \BBA {} Pauly, M.%
\end{APACrefauthors}%
\unskip\
\newblock
\APACrefYearMonthDay{2017}{04}{}.
\newblock
{\BBOQ}\APACrefatitle {MATS: Inference for potentially singular and
  heteroscedastic MANOVA} {Mats: Inference for potentially singular and
  heteroscedastic manova}.{\BBCQ}
\newblock
\APACjournalVolNumPages{Journal of Multivariate Analysis}{165}{}{166-179}.
\newblock
\begin{APACrefDOI} \doi{10.1016/j.jmva.2017.12.008} \end{APACrefDOI}
\PrintBackRefs{\CurrentBib}

\bibitem [\protect \citeauthoryear {%
Gaißer%
\ \BBA {} Schmid%
}{%
Gaißer%
\ \BBA {} Schmid%
}{%
{\protect \APACyear {2010}}%
}]{%
gaißer2010}
\APACinsertmetastar {%
gaißer2010}%
\begin{APACrefauthors}%
Gaißer, S.%
\BCBT {}\ \BBA {} Schmid, F.%
\end{APACrefauthors}%
\unskip\
\newblock
\APACrefYearMonthDay{2010}{}{}.
\newblock
{\BBOQ}\APACrefatitle {On testing equality of pairwise rank correlations in a
  multivariate random vector} {On testing equality of pairwise rank
  correlations in a multivariate random vector}.{\BBCQ}
\newblock
\APACjournalVolNumPages{Journal of Multivariate Analysis}{101}{10}{2598-2615}.
\newblock
\begin{APACrefURL}
  \url{https://www.sciencedirect.com/science/article/pii/S0047259X10001557}
  \end{APACrefURL}
\newblock
\begin{APACrefDOI} \doi{https://doi.org/10.1016/j.jmva.2010.07.008}
  \end{APACrefDOI}
\PrintBackRefs{\CurrentBib}

\bibitem [\protect \citeauthoryear {%
Gupta%
\ \BBA {} Xu%
}{%
Gupta%
\ \BBA {} Xu%
}{%
{\protect \APACyear {2006}}%
}]{%
gupta2006}
\APACinsertmetastar {%
gupta2006}%
\begin{APACrefauthors}%
Gupta, A\BPBI K.%
\BCBT {}\ \BBA {} Xu, J.%
\end{APACrefauthors}%
\unskip\
\newblock
\APACrefYearMonthDay{2006}{Mar}{01}.
\newblock
{\BBOQ}\APACrefatitle {On Some Tests of the Covariance Matrix Under General
  Conditions} {On some tests of the covariance matrix under general
  conditions}.{\BBCQ}
\newblock
\APACjournalVolNumPages{Annals of the Institute of Statistical
  Mathematics}{58}{1}{101--114}.
\newblock
\begin{APACrefURL} \url{https://doi.org/10.1007/s10463-005-0010-z}
  \end{APACrefURL}
\newblock
\begin{APACrefDOI} \doi{10.1007/s10463-005-0010-z} \end{APACrefDOI}
\PrintBackRefs{\CurrentBib}

\bibitem [\protect \citeauthoryear {%
Jennrich%
}{%
Jennrich%
}{%
{\protect \APACyear {1970}}%
}]{%
jennrich1970}
\APACinsertmetastar {%
jennrich1970}%
\begin{APACrefauthors}%
Jennrich, R\BPBI I.%
\end{APACrefauthors}%
\unskip\
\newblock
\APACrefYearMonthDay{1970}{}{}.
\newblock
{\BBOQ}\APACrefatitle {An Asymptotic chi square Test for the Equality of Two
  Correlation Matrices} {An asymptotic chi square test for the equality of two
  correlation matrices}.{\BBCQ}
\newblock
\APACjournalVolNumPages{Journal of the American Statistical
  Association}{65}{330}{904-912}.
\newblock
\begin{APACrefURL} \url{https://doi.org/10.1080/01621459.1970.10481133}
  \end{APACrefURL}
\newblock
\begin{APACrefDOI} \doi{10.1080/01621459.1970.10481133} \end{APACrefDOI}
\PrintBackRefs{\CurrentBib}

\bibitem [\protect \citeauthoryear {%
Joereskog%
}{%
Joereskog%
}{%
{\protect \APACyear {1978}}%
}]{%
joereskog}
\APACinsertmetastar {%
joereskog}%
\begin{APACrefauthors}%
Joereskog, K.%
\end{APACrefauthors}%
\unskip\
\newblock
\APACrefYearMonthDay{1978}{}{}.
\newblock
{\BBOQ}\APACrefatitle {Structural analysis of covariance and correlation
  matrices} {Structural analysis of covariance and correlation
  matrices}.{\BBCQ}
\newblock
\APACjournalVolNumPages{Psychometrika}{43}{4}{443-477}.
\newblock
\begin{APACrefURL}
  \url{https://EconPapers.repec.org/RePEc:spr:psycho:v:43:y:1978:i:4:p:443-477}
  \end{APACrefURL}
\PrintBackRefs{\CurrentBib}

\bibitem [\protect \citeauthoryear {%
Konietschke%
, Bösiger%
, Brunner%
\BCBL {}\ \BBA {} Hothorn%
}{%
Konietschke%
\ \protect \BOthers {.}}{%
{\protect \APACyear {2013}}%
}]{%
konietschke2013}
\APACinsertmetastar {%
konietschke2013}%
\begin{APACrefauthors}%
Konietschke, F.%
, Bösiger, S.%
, Brunner, E.%
\BCBL {}\ \BBA {} Hothorn, L\BPBI A.%
\end{APACrefauthors}%
\unskip\
\newblock
\APACrefYearMonthDay{2013}{}{}.
\newblock
{\BBOQ}\APACrefatitle {Are Multiple Contrast Tests Superior to the ANOVA?} {Are
  multiple contrast tests superior to the anova?}{\BBCQ}
\newblock
\APACjournalVolNumPages{The International Journal of
  Biostatistics}{9}{1}{63--73}.
\newblock
\begin{APACrefURL} [{2023-06-07}]\url{https://doi.org/10.1515/ijb-2012-0020}
  \end{APACrefURL}
\newblock
\begin{APACrefDOI} \doi{doi:10.1515/ijb-2012-0020} \end{APACrefDOI}
\PrintBackRefs{\CurrentBib}

\bibitem [\protect \citeauthoryear {%
Magnus%
\ \BBA {} Neudecker%
}{%
Magnus%
\ \BBA {} Neudecker%
}{%
{\protect \APACyear {1980}}%
}]{%
magnus1980}
\APACinsertmetastar {%
magnus1980}%
\begin{APACrefauthors}%
Magnus, J.%
\BCBT {}\ \BBA {} Neudecker, H.%
\end{APACrefauthors}%
\unskip\
\newblock
\APACrefYearMonthDay{1980}{}{}.
\newblock
\APACrefbtitle {The elimination matrix: Some lemmas and applications} {The
  elimination matrix: Some lemmas and applications}\ \APACbVolEdTR {}{Other
  publications TiSEM}.
\newblock
\APACaddressInstitution{}{Tilburg University, School of Economics and
  Management}.
\newblock
\begin{APACrefURL}
  \url{https://EconPapers.repec.org/RePEc:tiu:tiutis:0e3315d3-846c-4bc5-928e-f9f025fa05b5}
  \end{APACrefURL}
\PrintBackRefs{\CurrentBib}

\bibitem [\protect \citeauthoryear {%
Nel%
}{%
Nel%
}{%
{\protect \APACyear {1985}}%
}]{%
nel1985}
\APACinsertmetastar {%
nel1985}%
\begin{APACrefauthors}%
Nel, D.%
\end{APACrefauthors}%
\unskip\
\newblock
\APACrefYearMonthDay{1985}{}{}.
\newblock
{\BBOQ}\APACrefatitle {A matrix derivation of the asymptotic covariance matrix
  of sample correlation coefficients} {A matrix derivation of the asymptotic
  covariance matrix of sample correlation coefficients}.{\BBCQ}
\newblock
\APACjournalVolNumPages{Linear Algebra and its Applications}{67}{}{137-145}.
\newblock
\begin{APACrefURL}
  \url{http://www.sciencedirect.com/science/article/pii/0024379585901910}
  \end{APACrefURL}
\newblock
\begin{APACrefDOI} \doi{https://doi.org/10.1016/0024-3795(85)90191-0}
  \end{APACrefDOI}
\PrintBackRefs{\CurrentBib}

\bibitem [\protect \citeauthoryear {%
Nowak%
\ \BBA {} Konietschke%
}{%
Nowak%
\ \BBA {} Konietschke%
}{%
{\protect \APACyear {2021}}%
}]{%
nowak2021}
\APACinsertmetastar {%
nowak2021}%
\begin{APACrefauthors}%
Nowak, C\BPBI P.%
\BCBT {}\ \BBA {} Konietschke, F.%
\end{APACrefauthors}%
\unskip\
\newblock
\APACrefYearMonthDay{2021}{}{}.
\newblock
{\BBOQ}\APACrefatitle {Simultaneous inference for Kendall’s tau}
  {Simultaneous inference for kendall’s tau}.{\BBCQ}
\newblock
\APACjournalVolNumPages{Journal of Multivariate Analysis}{185}{}{104767}.
\newblock
\begin{APACrefURL}
  \url{https://www.sciencedirect.com/science/article/pii/S0047259X21000452}
  \end{APACrefURL}
\newblock
\begin{APACrefDOI} \doi{https://doi.org/10.1016/j.jmva.2021.104767}
  \end{APACrefDOI}
\PrintBackRefs{\CurrentBib}

\bibitem [\protect \citeauthoryear {%
Omelka%
\ \BBA {} Pauly%
}{%
Omelka%
\ \BBA {} Pauly%
}{%
{\protect \APACyear {2012}}%
}]{%
omPa:2012}
\APACinsertmetastar {%
omPa:2012}%
\begin{APACrefauthors}%
Omelka, M.%
\BCBT {}\ \BBA {} Pauly, M.%
\end{APACrefauthors}%
\unskip\
\newblock
\APACrefYearMonthDay{2012}{}{}.
\newblock
{\BBOQ}\APACrefatitle {Testing equality of correlation coefficients in two
  populations via permutation methods} {Testing equality of correlation
  coefficients in two populations via permutation methods}.{\BBCQ}
\newblock
\APACjournalVolNumPages{Journal of Statistical Planning and
  Inference}{142}{6}{1396--1406}.
\PrintBackRefs{\CurrentBib}

\bibitem [\protect \citeauthoryear {%
Perreault%
, Nešlehová%
\BCBL {}\ \BBA {} Duchesne%
}{%
Perreault%
\ \protect \BOthers {.}}{%
{\protect \APACyear {2022}}%
}]{%
perreault2022}
\APACinsertmetastar {%
perreault2022}%
\begin{APACrefauthors}%
Perreault, S.%
, Nešlehová, J\BPBI G.%
\BCBL {}\ \BBA {} Duchesne, T.%
\end{APACrefauthors}%
\unskip\
\newblock
\APACrefYearMonthDay{2022}{}{}.
\newblock
{\BBOQ}\APACrefatitle {Hypothesis Tests for Structured Rank Correlation
  Matrices} {Hypothesis tests for structured rank correlation matrices}.{\BBCQ}
\newblock
\APACjournalVolNumPages{Journal of the American Statistical
  Association}{0}{0}{1-12}.
\newblock
\begin{APACrefURL} \url{https://doi.org/10.1080/01621459.2022.2096619}
  \end{APACrefURL}
\newblock
\begin{APACrefDOI} \doi{10.1080/01621459.2022.2096619} \end{APACrefDOI}
\PrintBackRefs{\CurrentBib}

\bibitem [\protect \citeauthoryear {%
Revelle%
}{%
Revelle%
}{%
{\protect \APACyear {2019}}%
}]{%
psych}
\APACinsertmetastar {%
psych}%
\begin{APACrefauthors}%
Revelle, W.%
\end{APACrefauthors}%
\unskip\
\newblock
\APACrefYearMonthDay{2019}{}{}.
\newblock
{\BBOQ}\APACrefatitle {psych: Procedures for Psychological, Psychometric, and
  Personality Research} {psych: Procedures for psychological, psychometric, and
  personality research}{\BBCQ}\ [\bibcomputersoftwaremanual].
\newblock
\APACaddressPublisher{Evanston, Illinois}{}.
\newblock
\begin{APACrefURL} \url{https://CRAN.R-project.org/package=psych}
  \end{APACrefURL}
\newblock
\APACrefnote{R package version 1.9.12}
\PrintBackRefs{\CurrentBib}

\bibitem [\protect \citeauthoryear {%
Sakaori%
}{%
Sakaori%
}{%
{\protect \APACyear {2002}}%
}]{%
sakaori2002}
\APACinsertmetastar {%
sakaori2002}%
\begin{APACrefauthors}%
Sakaori, F.%
\end{APACrefauthors}%
\unskip\
\newblock
\APACrefYearMonthDay{2002}{01}{}.
\newblock
{\BBOQ}\APACrefatitle {Permutation test for equality of correlation
  coefficients in two populations} {Permutation test for equality of
  correlation coefficients in two populations}.{\BBCQ}
\newblock
\APACjournalVolNumPages{Communications in Statistics-simulation and
  Computation}{31}{}{641-651}.
\newblock
\begin{APACrefDOI} \doi{10.1081/SAC-120004317} \end{APACrefDOI}
\PrintBackRefs{\CurrentBib}

\bibitem [\protect \citeauthoryear {%
Sattler%
, Bathke%
\BCBL {}\ \BBA {} Pauly%
}{%
Sattler%
\ \protect \BOthers {.}}{%
{\protect \APACyear {2022}}%
}]{%
sattler2022}
\APACinsertmetastar {%
sattler2022}%
\begin{APACrefauthors}%
Sattler, P.%
, Bathke, A\BPBI C.%
\BCBL {}\ \BBA {} Pauly, M.%
\end{APACrefauthors}%
\unskip\
\newblock
\APACrefYearMonthDay{2022}{}{}.
\newblock
{\BBOQ}\APACrefatitle {Testing hypotheses about covariance matrices in general
  MANOVA designs} {Testing hypotheses about covariance matrices in general
  manova designs}.{\BBCQ}
\newblock
\APACjournalVolNumPages{Journal of Statistical Planning and
  Inference}{219}{}{134-146}.
\newblock
\begin{APACrefURL}
  \url{https://www.sciencedirect.com/science/article/pii/S0378375821001269}
  \end{APACrefURL}
\newblock
\begin{APACrefDOI} \doi{https://doi.org/10.1016/j.jspi.2021.12.001}
  \end{APACrefDOI}
\PrintBackRefs{\CurrentBib}

\bibitem [\protect \citeauthoryear {%
Staffen%
\ \protect \BOthers {.}}{%
Staffen%
\ \protect \BOthers {.}}{%
{\protect \APACyear {2014}}%
}]{%
staffen2014}
\APACinsertmetastar {%
staffen2014}%
\begin{APACrefauthors}%
Staffen, W.%
, Strobl, N.%
, Zauner, H.%
, Höller, Y.%
, Dobesberger, J.%
\BCBL {}\ \BBA {} Trinka, E.%
\end{APACrefauthors}%
\unskip\
\newblock
\APACrefYearMonthDay{2014}{}{}.
\newblock
{\BBOQ}\APACrefatitle {Combining SPECT and EEG analysis for assessment of
  disorders with amnestic symptoms to enhance accuracy in early diagnostics}
  {Combining spect and eeg analysis for assessment of disorders with amnestic
  symptoms to enhance accuracy in early diagnostics}.{\BBCQ}
\newblock
\APACjournalVolNumPages{Poster A19 Presented at the 11th Annual Meeting of the
  Austrian Society of Neurology, 26–29 March, Salzburg, Austria}{}{}{}.
\PrintBackRefs{\CurrentBib}

\bibitem [\protect \citeauthoryear {%
Steiger%
}{%
Steiger%
}{%
{\protect \APACyear {1980}}%
}]{%
steiger1980}
\APACinsertmetastar {%
steiger1980}%
\begin{APACrefauthors}%
Steiger, J\BPBI H.%
\end{APACrefauthors}%
\unskip\
\newblock
\APACrefYearMonthDay{1980}{}{}.
\newblock
{\BBOQ}\APACrefatitle {Testing Pattern Hypotheses On Correlation Matrices:
  Alternative Statistics And Some Empirical Results} {Testing pattern
  hypotheses on correlation matrices: Alternative statistics and some empirical
  results}.{\BBCQ}
\newblock
\APACjournalVolNumPages{Multivariate Behavioral Research}{15}{3}{335-352}.
\newblock
\APACrefnote{PMID: 26794186}
\newblock
\begin{APACrefDOI} \doi{10.1207/s15327906mbr1503_7} \end{APACrefDOI}
\PrintBackRefs{\CurrentBib}

\bibitem [\protect \citeauthoryear {%
Tian%
\ \BBA {} Wilding%
}{%
Tian%
\ \BBA {} Wilding%
}{%
{\protect \APACyear {2008}}%
}]{%
tian2008}
\APACinsertmetastar {%
tian2008}%
\begin{APACrefauthors}%
Tian, L.%
\BCBT {}\ \BBA {} Wilding, G\BPBI E.%
\end{APACrefauthors}%
\unskip\
\newblock
\APACrefYearMonthDay{2008}{}{}.
\newblock
{\BBOQ}\APACrefatitle {Confidence interval estimation of a common correlation
  coefficient} {Confidence interval estimation of a common correlation
  coefficient}.{\BBCQ}
\newblock
\APACjournalVolNumPages{Computational statistics \& data
  analysis}{52}{10}{4872--4877}.
\PrintBackRefs{\CurrentBib}

\bibitem [\protect \citeauthoryear {%
Umlauft%
, Placzek%
, Konietschke%
\BCBL {}\ \BBA {} Pauly%
}{%
Umlauft%
\ \protect \BOthers {.}}{%
{\protect \APACyear {2019}}%
}]{%
umlauft2019}
\APACinsertmetastar {%
umlauft2019}%
\begin{APACrefauthors}%
Umlauft, M.%
, Placzek, M.%
, Konietschke, F.%
\BCBL {}\ \BBA {} Pauly, M.%
\end{APACrefauthors}%
\unskip\
\newblock
\APACrefYearMonthDay{2019}{}{}.
\newblock
{\BBOQ}\APACrefatitle {Wild bootstrapping rank-based procedures: Multiple
  testing in nonparametric factorial repeated measures designs} {Wild
  bootstrapping rank-based procedures: Multiple testing in nonparametric
  factorial repeated measures designs}.{\BBCQ}
\newblock
\APACjournalVolNumPages{Journal of Multivariate Analysis}{171}{}{176-192}.
\newblock
\begin{APACrefURL}
  \url{https://www.sciencedirect.com/science/article/pii/S0047259X18301271}
  \end{APACrefURL}
\newblock
\begin{APACrefDOI} \doi{https://doi.org/10.1016/j.jmva.2018.12.005}
  \end{APACrefDOI}
\PrintBackRefs{\CurrentBib}

\bibitem [\protect \citeauthoryear {%
van~der Vaart%
\ \BBA {} Wellner%
}{%
van~der Vaart%
\ \BBA {} Wellner%
}{%
{\protect \APACyear {1996}}%
}]{%
vanderVaart1996}
\APACinsertmetastar {%
vanderVaart1996}%
\begin{APACrefauthors}%
van~der Vaart, A\BPBI W.%
\BCBT {}\ \BBA {} Wellner, J\BPBI A.%
\end{APACrefauthors}%
\unskip\
\newblock
\APACrefYearMonthDay{1996}{}{}.
\newblock
{\BBOQ}\APACrefatitle {Weak Convergence} {Weak convergence}.{\BBCQ}
\newblock
\BIn{} \APACrefbtitle {Weak Convergence and Empirical Processes: With
  Applications to Statistics} {Weak convergence and empirical processes: With
  applications to statistics}\ (\BPGS\ 16--28).
\newblock
\APACaddressPublisher{New York, NY}{Springer New York}.
\newblock
\begin{APACrefURL} \url{https://doi.org/10.1007/978-1-4757-2545-2_3}
  \end{APACrefURL}
\newblock
\begin{APACrefDOI} \doi{10.1007/978-1-4757-2545-2_3} \end{APACrefDOI}
\PrintBackRefs{\CurrentBib}

\bibitem [\protect \citeauthoryear {%
Welz%
, Doebler%
\BCBL {}\ \BBA {} Pauly%
}{%
Welz%
\ \protect \BOthers {.}}{%
{\protect \APACyear {2022}}%
}]{%
welz2022}
\APACinsertmetastar {%
welz2022}%
\begin{APACrefauthors}%
Welz, T.%
, Doebler, P.%
\BCBL {}\ \BBA {} Pauly, M.%
\end{APACrefauthors}%
\unskip\
\newblock
\APACrefYearMonthDay{2022}{}{}.
\newblock
{\BBOQ}\APACrefatitle {Fisher transformation based confidence intervals of
  correlations in fixed-and random-effects meta-analysis} {Fisher
  transformation based confidence intervals of correlations in fixed-and
  random-effects meta-analysis}.{\BBCQ}
\newblock
\APACjournalVolNumPages{British Journal of Mathematical and Statistical
  Psychology}{75}{1}{1--22}.
\PrintBackRefs{\CurrentBib}

\bibitem [\protect \citeauthoryear {%
Wilks%
}{%
Wilks%
}{%
{\protect \APACyear {1946}}%
}]{%
wilks1946}
\APACinsertmetastar {%
wilks1946}%
\begin{APACrefauthors}%
Wilks, S\BPBI S.%
\end{APACrefauthors}%
\unskip\
\newblock
\APACrefYearMonthDay{1946}{}{}.
\newblock
{\BBOQ}\APACrefatitle {Sample criteria for testing equality of means, equality
  of variances, and equality of covariances in a normal multivariate
  distribution} {Sample criteria for testing equality of means, equality of
  variances, and equality of covariances in a normal multivariate
  distribution}.{\BBCQ}
\newblock
\APACjournalVolNumPages{The Annals of Mathematical
  Statistics}{17}{3}{257--281}.
\PrintBackRefs{\CurrentBib}

\bibitem [\protect \citeauthoryear {%
Wu%
, Weng%
, Wang%
, Wang%
\BCBL {}\ \BBA {} Liu%
}{%
Wu%
\ \protect \BOthers {.}}{%
{\protect \APACyear {2018}}%
}]{%
wu2018}
\APACinsertmetastar {%
wu2018}%
\begin{APACrefauthors}%
Wu, L.%
, Weng, C.%
, Wang, X.%
, Wang, K.%
\BCBL {}\ \BBA {} Liu, X.%
\end{APACrefauthors}%
\unskip\
\newblock
\APACrefYearMonthDay{2018}{}{}.
\newblock
{\BBOQ}\APACrefatitle {Test of Covariance and Correlation Matrices} {Test of
  covariance and correlation matrices}.{\BBCQ}
\newblock
\APACjournalVolNumPages{arXiv: Methodology}{}{}{}.
\PrintBackRefs{\CurrentBib}

\end{thebibliography}

\end{document}